\documentclass{aomart}
\pdfoutput=1

\copyrightnote{\textcopyright~2022 by the authors.
This paper may be reproduced, in its entirely, for noncommercial purposes.}
\fancypagestyle{firstpage}{%
  \fancyhf{}%
  \chead{}%
   \cfoot{\footnotesize\thepage}}%
\usepackage{microtype}
\usepackage{tikz}
\def\arXiv#1{arXiv:\href{https://arXiv.org/abs/#1}{#1}}

\theoremstyle{plain}
\newtheorem{theorem}{Theorem}
\newtheorem{corollary}[theorem]{Corollary}
\newtheorem{lemma}[theorem]{Lemma}
\newtheorem{proposition}[theorem]{Proposition}
\newtheorem{conjecture}[theorem]{Conjecture}
\newtheorem{problem}[theorem]{Open Problem}
\numberwithin{theorem}{section}

\theoremstyle{definition}
\newtheorem{definition}[theorem]{Definition}

\numberwithin{equation}{section}
\numberwithin{figure}{section}

\newcommand\e{\varepsilon}

\newcommand{\R}{{\mathbb R}}
\newcommand{\Q}{{\mathbb Q}}
\newcommand{\Z}{{\mathbb Z}}
\newcommand{\C}{{\mathbb C}}
\newcommand{\mC}{{\mathcal C}}
\newcommand{\N}{{\mathbb N}}

\newcommand{\Hyp}{{\mathbb H}}
\newcommand{\T}{{\mathcal{T}}}
\newcommand{\K}{{\mathcal{K}}}

\newcommand{\slz}{\mathrm{SL}_2(\mathbb{Z})}
\newcommand{\SL}{\mathop{\textup{SL}}\nolimits}
\newcommand{\Hom}{\mathop{\textup{Hom}}\nolimits}
\newcommand{\GL}{\mathop{\textup{GL}}\nolimits}
\newcommand{\SO}{\mathop{\textup{SO}}\nolimits}
\newcommand{\PSL}{\mathop{\textup{PSL}}\nolimits}
\newcommand{\vol}{\mathop{\textup{vol}}}
\renewcommand{\Im}{\mathop{\textup{Im}}}
\renewcommand{\Re}{\mathop{\textup{Re}}}
\newcommand{\Schw}{\mathcal{S}}
\newcommand{\Neighbors}{\Omega}
\newcommand{\neighbor}{\omega}
\newcommand{\twelvecubed}{N}
\newcommand{\Proj}{\mathbb{P}}

\title[Universal optimality of the $E_8$ and Leech lattices]
{Universal optimality of the $E_8$ and Leech lattices and interpolation formulas}

\author[Cohn]{Henry Cohn}
\address{Microsoft Research New England\\
Cambridge, MA, USA} \email{cohn@microsoft.com}

\author[Kumar]{Abhinav Kumar}
\address{Stony Brook University\\
Stony Brook, NY, USA}
\email{thenav@gmail.com}

\author[Miller]{Stephen D.\ Miller}
\address{Rutgers University\\
Piscataway, NJ, USA} \email{miller@math.rutgers.edu}

\author[Radchenko]{\\Danylo Radchenko} % ad hoc line break
\address{Max Planck Institute for Mathematics\\
Bonn, Germany}
\email{danradchenko@gmail.com}

\author[Viazovska]{Maryna Viazovska}
\address{\'Ecole Polytechnique F\'ed\'erale de Lausanne\\
Lausanne, Switzerland}
\email{viazovska@gmail.com}

\thanks{Miller's research was supported by National Science Foundation grants
CNS-1526333 and CNS-1815562, and Viazovska's research was supported by
Swiss National Science Foundation project 184927.  The authors thank the Rutgers
Office of Advanced Research Computing for their computational support and resources.}

\begin{document}

\begin{abstract}
We prove that the $E_8$ root lattice and the Leech lattice are universally
optimal among point configurations in Euclidean spaces of dimensions $8$
and $24$, respectively. In other words, they minimize energy for every
potential function that is a completely monotonic function of squared
distance (for example, inverse power laws or Gaussians), which is a strong
form of robustness not previously known for any configuration in more than
one dimension. This theorem implies their recently shown optimality as
sphere packings, and broadly generalizes it to allow for long-range
interactions.

The proof uses sharp linear programming bounds for energy. To construct the
optimal auxiliary functions used to attain these bounds, we prove a new
interpolation theorem, which is of independent interest.  It reconstructs a
radial Schwartz function $f$ from the values and radial derivatives of $f$
and its Fourier transform $\widehat{f}$ at the radii $\sqrt{2n}$ for
integers $n\ge1$ in $\R^8$ and $n \ge 2$ in $\R^{24}$.  To prove this
theorem, we construct an interpolation basis using integral transforms of
quasimodular forms, generalizing Viazovska's work on sphere packing and
placing it in the context of a more conceptual theory.
\end{abstract}

\maketitle

\tableofcontents

\section{Introduction} \label{sec:intro}

What is the best way to arrange a discrete set of points in $\R^d$?  Of
course the answer depends on the objective: there are many different ways to
measure the quality of a configuration for interpolation, quadrature,
discretization, error correction, or other problems.  A configuration that is
optimal for one purpose will often be good for others, but usually not
optimal for them as well.  Those that optimize many different objectives
simultaneously play a special role in mathematics.  In this paper, we prove a
broad optimality theorem for the $E_8$ and Leech lattices, via a new
interpolation formula for radial Schwartz functions.  Our results help
characterize the exceptional nature of these lattices. (See \cite{E} and
\cite{SPLAG} for their definitions and basic properties.)

\subsection{Potential energy minimization}

One particularly fruitful family of objectives to optimize is energy under
different potential functions. Given a \emph{potential function} $p \colon
(0,\infty) \to \R$, we define the \emph{potential energy} of a finite subset
$\mC$ of $\R^d$ to be
\[
\sum_{\substack{x,y \in \mC\\x \ne y}} p\big(|x-y|\big),
\]
where $|\cdot|$ is the Euclidean norm (note that we include each pair of
points twice, which differs from the convention in physics). Our
primary interest is in infinite sets $\mC$, for which potential energy
requires renormalization because the double sum often diverges. Define a
\emph{point configuration}, or just \emph{configuration}, $\mC$ to be a
nonempty, discrete, closed subset of $\R^d$ (i.e., every ball in $\R^d$
contains only finitely many points of $\mC$).  We say $\mC$ has
\emph{density~$\rho$} if
\[
\lim_{r \to \infty} \frac{\big|{\kern 0.06em}\mC \cap B_r^d(0)\big|}{\vol\mathopen{}\big(B_r^d(0)\big)\mathclose{}} = \rho,
\]
where $B_r^d(0)$ denotes the closed ball of radius $r$ about $0$ in $\R^d$.
For such a set, we can renormalize the energy by considering the average
energy per particle, as follows.

\begin{definition}
Let $p \colon (0,\infty) \to \R$ be any function. The \emph{lower $p$-energy}
of a point configuration $\mC$ in $\R^d$ is
\[
E_p(\mC) := \liminf_{r \to \infty} \frac{1}{\big|{\kern 0.06em}\mC \cap B_r^d(0)\big|}
\sum_{\substack{x,y \in \mC \cap B_r^d(0)\\x \ne y}} p\big(|x-y|\big).
\]
If the limit of the above quantity exists, and not just its limit inferior,
then we call $E_p(\mC)$ the \emph{$p$-energy} of $\mC$ (and say that its
$p$-energy exists).  We allow the possibility that the energy may be $\pm
\infty$.
\end{definition}

The simplest case is when the configuration is a \emph{lattice} $\Lambda$,
i.e., the $\Z$-span of a basis of $\R^d$.  In that case, it has density
\[
\frac{1}{\vol\mathopen{}\big(\R^d / \Lambda \big)\mathclose{}}
\]
and $p$-energy
\[
\sum_{x \in \Lambda \setminus\{0\}} p\big(|x|\big),
\]
assuming this sum is absolutely convergent.  More generally, a \emph{periodic
configuration} is the union of finitely many orbits under the translation
action of some lattice, i.e., the union of pairwise disjoint translates
$\Lambda+v_j$ of a lattice $\Lambda$, with $1 \le j \le N$.  Such a
configuration has density
\[
\frac{N}{\vol\mathopen{}\big(\R^d / \Lambda \big)\mathclose{}}
\]
and $p$-energy
\begin{equation} \label{eq:periodicenergy}
\frac{1}{N}\sum_{j,k=1}^N \sum_{x \in \Lambda \setminus\{v_k-v_j\}} p\big(|x+v_j-v_k|\big),
\end{equation}
again assuming absolute convergence.  Many important configurations are
periodic, but others are not, and we do not assume periodicity in our main
theorems.

Typically we take the potential function $p$ to be decreasing, and we
envision the points of $\mC$ as particles subject to a repulsive force.  Our
framework is purely classical and does not incorporate quantum effects; thus,
we should not think of the particles as atoms.  However, classical models
have other applications \cite{BG}, such as describing mesoscale materials.
Our goal is then to arrange these particles so as to minimize their
$p$-energy, subject to maintaining a fixed density.\footnote{Fixing the
density prevents the particles from minimizing energy by receding to
infinity.  If fixing the density seems unphysical, we could instead impose a
chemical potential that penalizes decreasing the density of the
configuration, to account for exchange with the external environment.  That
turns out to be equivalent, in the sense that one can achieve any desired
density by choosing an appropriate chemical potential. Specifically, the
chemical potential is a Lagrange multiplier that converts the
density-constrained optimization problem to an unconstrained problem. This
approach is called the \emph{grand canonical ensemble}, and it is typically
set up for finite systems before taking a thermodynamic limit. See, for
example, Section~1.2.1(c) and Theorem~3.4.6 in \cite{Ru}.} More precisely, we
compare with the lower $p$-energies of other configurations:

\begin{definition}
Let $\mC$ be a point configuration in $\R^d$ with density $\rho$, where
$\rho>0$, and let $p \colon (0,\infty) \to \R$ be any function.  We say that
$\mC$ \emph{minimizes energy for $p$} if its $p$-energy $E_p(\mC)$ exists and
every configuration in $\R^d$ of density $\rho$ has lower $p$-energy at least
$E_p(\mC)$.  We also call $\mC$ a \emph{ground state} for $p$.
\end{definition}

In certain contrived cases it is easy to determine the minimal energy.  For
example, if $p$ vanishes at the square roots of positive integers and is
nonnegative elsewhere, then $\Z^d$ clearly minimizes $p$-energy. However,
rigorously determining ground states seems hopelessly difficult in general,
because of the complexity of analyzing long-range interactions. This issue
arises in physics and materials science as the \emph{crystallization problem}
\cite{BL,Simon}: how can we understand why particles so often arrange
themselves periodically at low temperatures?  Even the simplest mathematical
models of crystallization are enormously subtle, and surprisingly little has
been proved about them.

Two important classes of potential functions are inverse power laws $r
\mapsto 1/r^s$ with $s>0$ and Gaussians $r \mapsto e^{-\alpha r^2}$ with
$\alpha>0$.  Inverse power laws are special because they are homogeneous,
which implies that their ground states are scale-free: if $\mC$ is a ground
state in $\R^d$ with density $1$, then $\rho^{-1/d}\mC$ is a ground state
with density $\rho$. Gaussians lack this property, and the shape of their
ground states may depend on density (see, for example, \cite{CKS}).  In
applications, Gaussian potential functions are typically used to approximate
effective potential functions for more complex materials. For example, for a
dilute solution of polymer chains in an athermal solvent, Gaussians describe
the effective interaction between the centers of mass of the polymers (see
Sections~3.4 and~3.5 of \cite{Likos}). Point particles with Gaussian
interactions are known as the \emph{Gaussian core model} in physics \cite{S}.

Both inverse power laws and Gaussians arise naturally in number theory (see,
for example, \cite{M} and \cite{SaSt}). Given a lattice $\Lambda$ in $\R^d$,
its \emph{Epstein zeta function} is defined by
\[
\zeta_\Lambda(s) = \sum_{x \in \Lambda\setminus\{0\}} \frac{1}{|x|^{2s}}
\]
for $\Re(s) > d/2$ (and by analytic continuation for all $s$ except for a
pole at $s=d/2$), and its \emph{theta series} is defined by
\[
\Theta_\Lambda(z) = \sum_{x \in \Lambda} e^{\pi i z |x|^2}
\]
for $\Im(z) > 0$.  Then the energy of $\Lambda$ under $r \mapsto 1/r^s$ is
$\zeta_\Lambda(s/2)$ when $s>d$, while its energy under $r \mapsto e^{-\alpha
r^2}$ is $\Theta_\Lambda(i\alpha/\pi)-1$.  In other words, minimizing energy
among lattices amounts to seeking extreme values for number-theoretic special
functions. The restriction to lattices makes this problem much more tractable
than the crystallization problem, but it remains difficult.  The answer is
known in certain special cases, such as sufficiently large $s$ when $d \le 8$
(see \cite{Ry}), but the only cases in which the answer was previously known
for all $s$ and $\alpha$ were one or two dimensions \cite{M}.

Energy minimization also generalizes the sphere packing problem in $\R^d$, in
which we wish to maximize the minimal distance between neighboring particles
while fixing the particle density.  (Centering non-overlapping spheres at the
particles then yields a densest sphere packing.)  One simple way to see why
is to pick a constant $r_0$ and use the potential function
\[
p(r) = \begin{cases}
0 & \textup{if $r \ge r_0$, and}\\
$1$ & \textup{if $0 < r < r_0$.}
\end{cases}
\]
Then the $p$-energy of a periodic configuration vanishes if and only if the
configuration has minimal distance at least $r_0$.  More generally, we can
use a steep potential function such as $r \mapsto 1/r^s$ with $s$ large (or,
similarly, $r \mapsto e^{-\alpha r^2}$ with $\alpha$ large).  As $s \to
\infty$, the contribution to energy from short distances becomes increasingly
important, and in the limit minimizing energy requires maximizing the minimal
distance. In many cases, the ground state will be slightly distorted at any
finite $s$, compared with the limit as $s \to \infty$.  For example, that
seems to happen in five or seven dimensions \cite{CKS}. However, the $E_6$
root lattice in $\R^6$ appears to minimize energy for all sufficiently large
$s$, provably among lattices \cite{Ry} and perhaps among all configurations
\cite{CKS}.

\subsection{Universal optimality}

In contrast to the sphere packing problem, even one-dimensional energy
minimization is not easy to analyze.  Ventevogel and Nijboer \cite{VN} proved
that the integer lattice $\Z$ in $\R$ minimizes energy for every completely
monotonic function of squared distance, i.e., every function of the form $r
\mapsto g\big(r^2\big)$ with $g$ completely monotonic. Recall that a function
$g \colon (0,\infty) \to \R$ is \emph{completely monotonic} if it is
infinitely differentiable and satisfies the inequalities $(-1)^k g^{(k)} \ge
0$ for all $k \ge 0$. In other words, $g$ is nonnegative, weakly decreasing,
convex, and so on. For example, inverse power laws are completely monotonic,
as are decreasing exponential functions. By Bernstein's theorem
\cite[Theorem~9.16]{Simon2}, every completely monotonic function $g \colon
(0,\infty) \to \R$ can be written as a convergent integral
\[
g(r) = \int e^{-\alpha r} \, d\mu(\alpha)
\]
for some measure $\mu$ on $[0,\infty)$. Equivalently, the completely
monotonic functions of squared distance are the cone spanned by the Gaussians
and the constant function $1$. For example, inverse power laws can be
obtained via
\[
\frac{1}{r^s} = \int_0^\infty e^{-\alpha r^2} \frac{\alpha^{s/2-1}}{\Gamma(s/2)}
\, d\alpha
\]
with $s>0$. By contrast, Ventevogel and Nijboer showed that $\Z$ does not
minimize energy for the potential function $r \mapsto 1/(16 + r^4)$, which is
not a completely monotonic function of squared distance.

It follows from Bernstein's theorem that if a periodic configuration is a
ground state for every Gaussian, then the same is true for every completely
monotonic function of squared distance (by monotone convergence, because the
potential is an increasing limit of weighted sums of finitely many
Gaussians). Following Cohn and Kumar \cite{CK07}, who gave a different proof
of optimality in one dimension, we call such a configuration universally
optimal:

\begin{definition}
Let $\mC$ be a point configuration in $\R^d$ with density $\rho$, where $\rho
> 0$. We say $\mC$ is \emph{universally optimal} if it minimizes $p$-energy
whenever $p\colon (0,\infty) \to \R$ is a completely monotonic function of
squared distance.
\end{definition}

Note that the role of density in this definition is purely bookkeeping.  If
$\mC$ is a universal optimum in $\R^d$ with density $1$, then
$\rho^{-1/d}\mC$ is a universal optimum with density $\rho$ for any $\rho>0$,
because the set of completely monotonic functions is invariant under
rescaling the input variable.  We can also reformulate universal optimality
by fixing a Gaussian and varying the density: a periodic configuration $\mC$
in $\R^d$ with density $1$ is universally optimal if and only if for every
$\rho>0$, the configuration $\rho^{-1/d}\mC$ is a ground state for $r \mapsto e^{-\pi r^2}$.
This perspective on the Gaussian core model is common in the physics
literature, such as \cite{S}, because varying the density of particles
governed by a fixed interaction is a common occurrence in physics, while
changing how they interact is more exotic.

It might seem more natural to use completely monotonic functions of distance,
rather than squared distance, but squared distance turns out to be a better
choice (for example, in allowing Gaussians).  One can check that every
completely monotonic function of distance is also a completely monotonic
function of squared distance; equivalently, if $r \mapsto g\big(r^2\big)$ is
completely monotonic, then so is $g$ itself.\footnote{That is, if $p$ is
completely monotonic on $(0,\infty)$, then so is $r \mapsto
p\big(r^{1/2}\big)$.  More generally, if $p$ and $q'$ are both completely
monotonic functions, then so is the composition $p \circ q$.  The reason is
that if one computes the $k$-th derivative $(p \circ q)^{(k)}$, for example
using Fa\`a di Bruno's formula, then each term has sign $(-1)^k$, as
desired.} Thus, using squared distance strengthens the definition.

When a configuration is universally optimal, it has an extraordinary degree
of robustness: it remains optimal for a broad range of potential functions,
rather than depending on the specific potential.  Numerical studies of energy
minimization indicate that universal optima are rare \cite{CKS}, and this
special property highlights their importance across different fields.

Before the present paper, no examples of universal optima in $\R^d$ with
$d>1$ had been rigorously proved. In fact, for $d>1$ no proof was known of a
ground state for any inverse power law or similarly natural repulsive
potential function. The most noteworthy theorem we are aware of along these
lines is a proof by Theil \cite{T} of crystallization for certain
Lennard-Jones-type potentials in $\R^2$. However, the potentials analyzed by
Theil are attractive at long distances, and the proof makes essential use of
this attraction.

Despite the lack of proof, the $A_2$ root lattice (i.e., the hexagonal
lattice) is almost certainly universally optimal in $\R^2$.  It is known to
be universally optimal among lattices \cite{M}, and proving its universal
optimality in full generality is an important open problem.

The case of three dimensions is surprisingly tricky even to describe. For the
potential function $r \mapsto e^{-\pi r^2}$, the appropriately scaled
face-centered cubic lattice is widely conjectured to be optimal among
lattices of density $\rho$ as long as $\rho \le 1$, while the body-centered
cubic lattice is conjectured to be optimal when $\rho \ge 1$. At density $1$,
they have the same energy by Poisson summation, because they are dual to each
other. However, one can sometimes lower the energy by moving beyond lattices:
Stillinger \cite{S} applied Maxwell's double tangent construction to obtain a
small neighborhood around density~$1$, namely $(0.99899854\dots,
1.00100312\dots)$, in which phase coexistence between these lattices improves
upon both of them by a small amount. Specifically, at density~$1$ phase
coexistence lowers the energy by approximately $0.0004\%$, in a way that
seemingly cannot be achieved exactly by any periodic configuration. Thus, the
behavior of the Gaussian core model in three dimensions is more complex than
one might expect from the case of lattices.  Even guessing the ground states
on the basis of simulations is far from straightforward, and proofs seem to
be well beyond present-day mathematics.

In contrast, we completely resolve the cases of eight and twenty-four
dimensions, as conjectured in \cite{CK07}:

\begin{theorem} \label{theorem:univopt}
The $E_8$ root lattice and the Leech lattice are universally optimal in
$\R^8$ and $\R^{24}$, respectively.  Furthermore, they are unique among
periodic configurations, in the following sense.  Let $\mC$ be $E_8$ or the
Leech lattice, and let $\mC'$ be any periodic configuration in the same
dimension with the same density.  If there exists a completely monotonic
function of squared distance $p$ such that $E_p(\mC') = E_p(\mC) < \infty$,
then $\mC'$ is isometric to $\mC$.
\end{theorem}

Of course, the uniqueness assertion cannot hold among all configurations,
because removing a single particle changes neither the density nor the
energy. Uniqueness also trivially fails when $p$ decays slowly enough that
$E_p(\mC) = \infty$, because universal optimality then implies that
$E_p(\mC') = \infty$ for every $\mC'$.  One could attempt to renormalize a
divergent potential (analogously to the analytic continuation of the Epstein
zeta function), but we will not address that possibility. See \cite{HSS} and
\cite{SaSt} for more information about renormalization.

Even for lattices, Theorem~\ref{theorem:univopt} has numerous consequences,
including extreme values of the theta and Epstein zeta functions. For another
application, consider a flat torus $T = \R^d/\Lambda$, where $\Lambda$ is a
lattice in $\R^d$.  The \emph{height} of $T$ is a regularization of $-\log
\det \Delta_T$, where $\Delta_T$ is the Laplacian on $T$ (see \cite{Chiu} or
\cite{SaSt}). If $T$ has volume $1$, then the height is a constant depending
only on $d$ plus
\[
\lim_{s \to d/2} \left(\pi^{-d/2} \Gamma(d/2) \zeta_\Lambda(s) - \frac{1}{s-d/2}\right),
\]
by Theorem~2.3 in \cite{Chiu}. Thus, if $\Lambda$ is universally optimal,
then $T$ minimizes height among all flat $d$-dimensional tori of fixed
volume. The minimal height was previously known only when $d=1$ (trivial),
$d=2$ (due to Osgood, Phillips, and Sarnak \cite{OPS}), and $d=3$ (due to
Sarnak and Str\"ombergsson \cite{SaSt}), to which we can now add $d=8$ and
$d=24$, as conjectured in \cite{SaSt}:

\begin{corollary} \label{cor:epstein}
Let $d$ be $8$ or $24$, and let $\Lambda_d$ be $E_8$ or the Leech lattice,
accordingly. Among all lattices $\Lambda$ in $\R^d$ with determinant $1$, the
minimum value of $\zeta_\Lambda(s)$ for each $s \in
(0,\infty)\setminus\{d/2\}$ is achieved when $\Lambda = \Lambda_d$, as is the
minimum value of $\Theta_\Lambda(it)$ for each $t>0$. Furthermore,
$\R^d/\Lambda_d$ has the smallest height among all $d$-dimensional flat tori
of volume $1$.  For each of these optimization problems, $\Lambda_d$ is the
unique optimal lattice with determinant~$1$, up to isometry.
\end{corollary}

Optimality and uniqueness for $\zeta_\Lambda(s)$ with $s>d/2$ and for
$\Theta_\Lambda(it)$ follow from Theorem~\ref{theorem:univopt}, as does
optimality for the height. To deal with $0<s<d/2$ and to prove uniqueness (as
well as optimality) for the height, we can use the formula
\[
\aligned
\pi^{-s} \Gamma(s) \zeta_\Lambda(s) &= \int_1^\infty \big(\Theta_\Lambda(it)-1\big) t^{s-1} \, dt - \frac{1}{s}\\
& \quad \phantom{}
+ \int_1^\infty \big(\Theta_{\Lambda^*}(it)-1\big) t^{d/2-s-1} \, dt  + \frac{1}{s-d/2}
\endaligned
\]
(see, for example, equation~(42) in \cite{SaSt}) to reduce to the case of
$\Theta(it)$, because $\Lambda_d^* = \Lambda_d$.  Note also that
Corollary~\ref{cor:epstein} and the functional equation for the Epstein zeta
function imply that $\zeta_\Lambda(s)$ is minimized at $\Lambda=\Lambda_d$
when $s<0$ and $\lfloor s \rfloor$ is even, and maximized when $s<0$ and
$\lfloor s \rfloor$ is odd.

Although this corollary was not previously known, it is in principle more
tractable than Theorem~\ref{theorem:univopt}, because the lattice hypothesis
reduces these assertions to optimization problems in a fixed, albeit large,
number of variables.

\subsection{Linear programming bounds}

Our impetus for proving Theorem~\ref{theorem:univopt} was Viazovska's
solution of the sphere packing problem in eight dimensions \cite{V}, as well
as our extension to twenty-four dimensions \cite{CKMRV}.  These papers proved
conjectures of Cohn and Elkies \cite{CE} about the existence of certain
special functions, which imply sphere packing optimality via linear
programming bounds.  The underlying analytic techniques are by no means
limited to sphere packing, and they intrinsically take into account
long-range interactions. In particular, by combining the approach of Cohn and
Elkies with techniques originated by Yudin \cite{Y}, Cohn and Kumar
\cite{CK07} extended this framework of linear programming bounds to potential
energy minimization in Euclidean space. To prove
Theorem~\ref{theorem:univopt}, we prove Conjecture~9.4 in \cite{CK07} for
eight and twenty-four dimensions.  (The case of two dimensions remains open.)
As a corollary of our construction, we also obtain the values at the origin
of the optimal auxiliary functions in these bounds, which agree with
Conjecture~6.1 in \cite{CM}.

Recall that linear programming bounds work as follows (see \cite{C10,C17} for
further background). A \emph{Schwartz function} $f \colon \R^d \to \C$ is an
infinitely differentiable function such that for all $c>0$, the function
$f(x)$ and all its partial derivatives of all orders decay as
$O\big(|x|^{-c}\big)$ as $|x| \to \infty$.  We normalize the Fourier
transform by
\[
\widehat{f}(y) = \int_{\R^d} f(x) e^{-2\pi i \langle x,y \rangle} \, dx,
\]
where $\langle \cdot,\cdot\rangle$ denotes the standard inner product on
$\R^d$.  Then linear programming bounds for energy amount to the following
proposition:

\begin{proposition} \label{prop:LP}
Let $p \colon (0,\infty) \to [0,\infty)$ be any function, and suppose $f
\colon \R^d \to \R$ is a Schwartz function.  If $f(x) \le p\big(|x|\big)$ for
all $x \in \R^d \setminus \{0\}$ and $\widehat{f}(y) \ge 0$ for all $y \in
\R^d$, then every point configuration in $\R^d$ with density $\rho$ has lower
$p$-energy at least $\rho \widehat{f}(0) - f(0)$.
\end{proposition}

In other words, we can certify a lower bound for $p$-energy by exhibiting an
auxiliary function $f$ satisfying specific inequalities.  There is no reason
to believe a sharp lower bound can necessarily be certified in this way.
Indeed, in most cases the certifiable bounds seem to be strictly less than
the true ground state energy, and the gap between them can be large when the
potential function is steep.  For example, for configurations of density~$1$
in $\R^3$ under the Gaussian potential function $r \mapsto e^{-\alpha r^2}$,
the best linear programming bound known (obtained by applying the
constructions from \cite{ACHLT} to potential energy) is roughly 3.59\% less
than the lowest energy known when $\alpha=\pi$, and 15.4\% less when
$\alpha=2\pi$. Nevertheless, this technique suffices to prove
Theorem~\ref{theorem:univopt}.

Cohn and Kumar \cite[Proposition~9.3]{CK07} proved Proposition~\ref{prop:LP}
for the special case of periodic configurations $\mC$. Since the proof is
short and motivates much of what we do in this paper, we include it here. The
proof uses the Poisson summation formula
\[
\sum_{x \in \Lambda} f(x+v) = \frac{1}{\vol\mathopen{}\big(\R^d / \Lambda \big)\mathclose{}}
\sum_{y \in \Lambda^*} \widehat{f}(y) e^{2\pi i \langle v,y \rangle},
\]
which holds when $f \colon \R^d \to \C$ is a Schwartz function, $v \in \R^d$,
$\Lambda$ is a lattice in $\R^d$, and
\[
\Lambda^* = \{y \in \R^d : \textup{$\langle x,y \rangle \in \Z$ for all $x \in \Lambda$}\}
\]
is its \emph{dual lattice}.

\begin{proof}[Proof of Proposition~\ref{prop:LP} for periodic configurations]
Because $\mC$ is periodic, we can write it as the disjoint union of
$\Lambda+v_j$ for $1 \le j \le N$, where $\Lambda$ is a lattice and
$v_1,\dots,v_N \in \R^d$. Then the inequality between $f$ and $p$ and the
formula \eqref{eq:periodicenergy} for energy yield the lower bound
\begin{align*}
E_p(\mC) &=
\frac{1}{N}\sum_{j,k=1}^N \sum_{x \in \Lambda \setminus\{v_k-v_j\}} p\big(|x+v_j-v_k|\big)\\
&\ge \frac{1}{N}\sum_{j,k=1}^N \sum_{x \in \Lambda \setminus\{v_k-v_j\}} f(x+v_j-v_k)\\
&= \frac{1}{N}\sum_{j,k=1}^N \sum_{x \in \Lambda} f(x+v_j-v_k) - f(0).
\end{align*}
(We can apply \eqref{eq:periodicenergy} because $p \ge 0$: if the sum
diverges, then $E_p(\mC) = \infty$ anyway.) Applying Poisson summation to
this lower bound and using the nonnegativity of $\widehat{f}$ and the
equation $N = \rho \vol\mathopen{}\big(\R^d / \Lambda \big)\mathclose{}$ then
shows that
\begin{align*}
E_p(\mC) &\ge  \frac{N}{\vol\mathopen{}\big(\R^d / \Lambda \big)\mathclose{}}
\sum_{y \in \Lambda^*} \widehat{f}(y) \left|\frac{1}{N} \sum_{j=1}^N
e^{2\pi i \langle v_j, y \rangle}\right|^2 - f(0)\\
&\ge \rho \widehat{f}(0) - f(0),
\end{align*}
as desired.
\end{proof}

This proof works only for periodic configurations, but
Proposition~\ref{prop:LP} makes no such assumption. The general case was
proved by Cohn and de~Courcy-Ireland in \cite[Proposition~2.2]{CdCI}.

Proposition~\ref{prop:LP} shows how to obtain a lower bound for $p$-energy
from an auxiliary function $f$ satisfying certain inequalities, but it says
nothing about how to construct $f$.  Optimizing the choice of $f$ to maximize
the resulting bound is an unsolved problem in general.  Without loss of
generality, we can assume that $f$ is a radial function (i.e., $f(x)$ depends
only on $|x|$), because all the constraints are invariant under rotation and
we can therefore radially symmetrize $f$ by averaging all its rotations.  We
are faced with an optimization problem over functions of just one radial
variable, but this problem too seems to be intractable in general.

Fortunately, one can characterize when the bound is sharp for a periodic
configuration.  For simplicity, consider a lattice $\Lambda$. Examining the
loss in the inequalities in the proof given above shows that $f$ proves a
sharp bound for $E_p(\Lambda)$ if and only if both
\begin{equation}
\label{eq:necessaryconds1}
\begin{aligned}
f(x) &= p(|x|)& \qquad&\text{ for all $x \in \Lambda\setminus\{0\}$, and}\\
\widehat{f}(y) &= 0& &\text{ for all $y \in \Lambda^*\setminus\{0\}$.}
\end{aligned}
\end{equation}
Furthermore, equality must hold to second order, because $f(x) \le p(|x|)$
and $\widehat{f}(y) \ge 0$ for all $x$ and $y$.  Equivalently, if $f$ is
radial, then the radial derivatives of $f$ and $\widehat{f}$ satisfy
\begin{equation}
\label{eq:necessaryconds2}
\begin{aligned}
f'(x) &= p'(|x|)& \qquad&\text{ for all $x \in \Lambda\setminus\{0\}$, and}\\
\widehat{f}\,'(y) &= 0& &\text{ for all $y \in \Lambda^*\setminus\{0\}$.}
\end{aligned}
\end{equation}
We will see that these conditions suffice to determine $f$ in $8$ or $24$
dimensions.

\subsection{Interpolation}

For the analogous problem of energy minimization in compact spaces studied in
\cite{CK07}, conditions \eqref{eq:necessaryconds1} and
\eqref{eq:necessaryconds2} make it simple to construct the optimal auxiliary
functions for most of the universal optima that are known. The point
configurations are finite sets, and thus we have only finitely many equality
constraints for $f$ to achieve a sharp bound.  To construct $f$, one can
simply take the lowest-degree polynomial that satisfies these constraints. It
is far from obvious that this construction works, i.e., that $f$ satisfies
the needed inequalities elsewhere, but at least describing this choice of $f$
is straightforward. The description amounts to polynomial interpolation (more
precisely, Hermite interpolation, since one must interpolate both values and
derivatives).

In Euclidean space, describing the optimal auxiliary functions is far more
subtle.  It again amounts to an interpolation problem, this time for radial
Schwartz functions.  The interpolation points are known explicitly for $\R^8$
and $\R^{24}$: the nonzero vectors in $E_8$ have lengths $\sqrt{2n}$ for
integers $n \ge 1$, and those in the Leech lattice have lengths $\sqrt{2n}$
for $n \ge 2$. What is required is to control the values and radial
derivatives of $f$ and $\widehat{f}$ at these infinitely many points.
However, simultaneously controlling $f$ and $\widehat{f}$ is not easy, and we
quickly run up against uncertainty principles \cite{CG}.  The feasibility of
interpolation depends on the exact points at which we are interpolating, and
we do not know how to resolve these questions in general.

The fundamental mystery is how polynomial interpolation generalizes to
infinite-dimensional function spaces.  One important case that has been
thoroughly analyzed is Shannon sampling, which amounts to interpolating the
values of a band-limited function (i.e., an entire function of exponential
type) at linearly spaced points; see \cite{Higgins} for an account of this
theory. Shannon sampling suffices to prove that $\Z$ is universally optimal
\cite[p.~142]{CK07}, but it cannot handle any higher-dimensional cases.

To construct the optimal auxiliary functions in $\R^8$ and $\R^{24}$, we
prove a new interpolation theorem for radial Schwartz functions. Let
$\Schw_{\textup{rad}}(\R^d)$ denote the set of radial Schwartz functions from
$\R^d$ to $\C$.  For $f \in \Schw_{\textup{rad}}(\R^d)$, we abuse notation by
applying $f$ directly to radial distances (i.e., if $r \in [0,\infty)$, then
$f(r)$ denotes the common value of $f(x)$ when $|x|=r$), and we let $f'$
denote the radial derivative. As above, $\widehat{f}$ denotes the
$d$-dimensional Fourier transform of $f$, which is again a radial function,
and $\widehat{f}\,'$ denotes the radial derivative of $\widehat{f}$.

\begin{theorem} \label{theorem:interpolation}
Let $(d,n_0)$ be $(8,1)$ or $(24,2)$. Then every $f \in
\Schw_{\textup{rad}}(\R^d)$ is uniquely determined by the values
$f\big(\sqrt{2n}\big)$, $f'\big(\sqrt{2n}\big)$,
$\widehat{f}\big(\sqrt{2n}\big)$, and $\widehat{f}\,'\big(\sqrt{2n}\big)$ for
integers $n \ge n_0$.  Specifically, there exists an interpolation basis
$a_n, b_n, \widetilde{a}_n, \widetilde{b}_n \in \Schw_{\textup{rad}}(\R^d)$
for $n \ge n_0$ such that for every $f \in \Schw_{\textup{rad}}(\R^d)$ and $x
\in \R^d$,
\begin{equation}\label{eqn:IF}
\begin{split}
f(x) &= \sum_{n=n_0}^\infty f\big(\sqrt{2n}\big)\,a_n(x)+\sum_{n=n_0}^\infty
f'\big(\sqrt{2n}\big)\,b_n(x)\\
&\quad\phantom{}+\sum_{n=n_0}^\infty \widehat{f}\big(\sqrt{2n}\big)\,\widetilde{a}_n(x)
+\sum_{n=n_0}^\infty \widehat{f}\,'\big(\sqrt{2n}\big)\,\widetilde{b}_n(x),
\end{split}	
\end{equation}
where these series converge absolutely.
\end{theorem}

One could likely weaken the decay and smoothness conditions on $f$, along the
lines of Proposition~4 in \cite{RV}, but determining the best possible
conditions seems difficult.

Theorem~\ref{theorem:interpolation} tells us that in $\R^8$ or $\R^{24}$, the
optimal auxiliary function $f$ for a potential function $p$ is uniquely
determined by the necessary conditions \eqref{eq:necessaryconds1} and
\eqref{eq:necessaryconds2}, assuming it is a Schwartz function.
Specifically, $f$ satisfies the conditions
\begin{align*}
f\big(\sqrt{2n}\big) &= p\big(\sqrt{2n}\big),&
f'\big(\sqrt{2n}\big) &= p'\big(\sqrt{2n}\big),\\
\widehat{f}\big(\sqrt{2n}\big) &= 0,&
\widehat{f}\,'\big(\sqrt{2n}\big) &= 0
\end{align*}
for $n \ge n_0$, and \eqref{eqn:IF} then gives a formula for $f$ in terms of
the interpolation basis, which we will explicitly construct as part of the
proof of Theorem~\ref{theorem:interpolation}. The same is also true for the
auxiliary functions for sphere packing constructed in \cite{V} and
\cite{CKMRV}:

\begin{corollary}
In $\R^8$ and $\R^{24}$, the optimal auxiliary functions for the linear
programming bounds for sphere packing or Gaussian potential energy
minimization are unique among radial Schwartz functions.
\end{corollary}

Theorem~\ref{theorem:interpolation} was conjectured by Viazovska as part of
the strategy for her solution of the sphere packing problem in $\R^8$.  Note
that the interpolation formula is not at all obvious, or even particularly
plausible. The lack of plausibility accounts for why it had not previously
been conjectured, despite the analogy with energy minimization in compact
spaces.

The proof of the interpolation formula develops the techniques introduced by
Viazovska in \cite{V} into a broader theory.  Radchenko and Viazovska took a
significant step in this direction by proving an interpolation formula for
single roots in one dimension \cite{RV}, but extending it to double roots
introduces further difficulties.

Theorem~\ref{theorem:interpolation} extends naturally to characterize exactly
which sequences can occur as $f\big(\sqrt{2n}\big)$, $f'\big(\sqrt{2n}\big)$,
$\widehat{f}\big(\sqrt{2n}\big)$, and $\widehat{f}\,'\big(\sqrt{2n}\big)$ for
$n \ge n_0$, with $f$ a radial Schwartz function. The only restriction on
these sequences is on their decay rate. To state the result precisely, let
$\Schw(\N)$ be the space of rapidly decreasing sequences of complex numbers.
In other words, $(x_n)_{n \ge 1} \in \Schw(\N)$ if and only if $\lim_{n \to
\infty} n^k x_n = 0$ for all natural numbers $k$.

\begin{theorem} \label{theorem:interpolation-isom}
Let $(d,n_0)$ be $(8,1)$ or $(24,2)$.  Then the map sending $f \in
\Schw_{\textup{rad}}(\R^d)$ to
\[
\left(\big(f\big(\sqrt{2n}\big)\big)_{n \ge n_0},\big(f'\big(\sqrt{2n}\big)\big)_{n \ge n_0},
\big(\widehat{f}\big(\sqrt{2n}\big)\big)_{n \ge n_0},\big(\widehat{f}\,'\big(\sqrt{2n}\big)\big)_{n \ge n_0}\right)
\]
is an isomorphism from $\Schw_{\textup{rad}}(\R^d)$ to $\Schw(\N)^4$, whose
inverse is given by \eqref{eqn:IF}. In other words, the inverse isomorphism maps
\[
\left(\big(\alpha_n)_{n \ge n_0},\big(\beta_n\big)_{n \ge n_0},
\big(\widetilde{\alpha}_n\big)_{n \ge n_0},\big(\widetilde{\beta}_n\big)_{n
\ge n_0}\right)
\]
to the function
\[
\sum_{n=n_0}^\infty \alpha_n\,a_n+\sum_{n=n_0}^\infty \beta_n\,b_n
+\sum_{n=n_0}^\infty
\widetilde{\alpha}_n\,\widetilde{a}_n+\sum_{n=n_0}^\infty
\widetilde{\beta}_n\,\widetilde{b}_n.
\]
\end{theorem}

One consequence of this theorem is that there are no systematic linear
relations between the values $f\big(\sqrt{2n}\big)$, $f'\big(\sqrt{2n}\big)$,
$\widehat{f}\big(\sqrt{2n}\big)$, and $\widehat{f}\,'\big(\sqrt{2n}\big)$ for
$n \ge n_0$ (i.e., relations that hold for all $f \in
\Schw_{\textup{rad}}(\R^d)$).  By contrast, Poisson summation over $E_8$ or
the Leech lattice gives such a relation between $f\big(\sqrt{2n}\big)$ and
$\widehat{f}\big(\sqrt{2n}\big)$ for $n \ge 0$.

Another consequence is that the values and derivatives of the interpolation
basis functions and their Fourier transforms at the interpolation points are
all $0$ except for a single $1$, which cycles through all these
possibilities.  It follows from the inverse isomorphism in
Theorem~\ref{theorem:interpolation-isom} that in terms of the Kronecker
delta,
\begin{align*}
a_n\big(\sqrt{2m}) &= \delta_{m,n}, & a_n'\big(\sqrt{2m}) &= 0, &
\widehat{a}_n\big(\sqrt{2m}) &= 0, & \widehat{a}_n\,\!\!\!'\big(\sqrt{2m}) &= 0,\\
b_n\big(\sqrt{2m}) &= 0, & b_n'\big(\sqrt{2m}) &= \delta_{m,n}, &
\widehat{b}_n\big(\sqrt{2m}) &= 0, & \widehat{b}_n\,\!\!\!'\big(\sqrt{2m}) &= 0,\\
\widetilde{a}_n\big(\sqrt{2m}) &= 0, & \widetilde{a}_n'\big(\sqrt{2m}) &= 0, &
\widehat{\widetilde{a}}_n\big(\sqrt{2m}) &= \delta_{m,n}, & \widehat{\widetilde{a}}_n\,\!\!\!'\big(\sqrt{2m}) &= 0,\\
\widetilde{b}_n\big(\sqrt{2m}) &= 0, & \widetilde{b}_n'\big(\sqrt{2m}) &= 0, &
\widehat{\widetilde{b}}_n\big(\sqrt{2m}) &= 0, & \widehat{\widetilde{b}}_n\,\!\!\!'\big(\sqrt{2m}) &= \delta_{m,n}
\end{align*}
for integers $m,n\ge n_0$. (By contrast, Theorem~\ref{theorem:interpolation}
allows the possibility that our interpolation basis could be redundant.)
These equations uniquely determine the functions $a_n$, $b_n$,
$\widetilde{a}_n$, and $\widetilde{b}_n$, by the interpolation theorem
itself.  Furthermore, it follows that $\widetilde{a}_n = \widehat{a}_n$ and
$\widetilde{b}_n = \widehat{b}_n$.

The function $b_1$ is characterized by $b_1\big(\sqrt{2n}\big) = 0$,
$\widehat{b}_1\big(\sqrt{2n}\big) = 0$, and
$\widehat{b}_1\,\!\!\!'\big(\sqrt{2n}\big) = 0$ for $n \ge n_0$, while
$b_1\,\!\!\!'\big(\sqrt{2n}\big) = 0$ for $n > n_0$ and
$b_1\,\!\!\!'\big(\sqrt{2n_0}\big) = 1$.  Up to a constant factor, these
equations are the same conditions satisfied by the sphere packing auxiliary
functions constructed in \cite{V} for $d=8$ and \cite{CKMRV} for $d=24$.
Thus the constructions in those papers are subsumed as a special case of the
interpolation basis.

The interpolation theorem amounts to constructing certain radial analogues of
Fourier quasicrystals. Define a \emph{radial Fourier quasicrystal} to be a
radial tempered distribution $T$ on $\R^d$ such that both $T$ and
$\widehat{T}$ are supported on discrete sets of radii, where $\widehat{T}$ is
characterized as usual by
\[
\int_{\R^d} \widehat{T}(x) f(x) \, dx = \int_{\R^d} T(x) \widehat{f}(x) \, dx
\]
for $f \in \Schw_{\textup{rad}}(\R^d)$.  To reformulate the
interpolation theorem in these terms, let $\delta_r$ denote a spherical delta
function with mass $1$ on the sphere of radius $r$ about the origin, or
equivalently let
\[
\int_{\R^d} f(x) \delta_r(x) \, dx = f(r)
\]
for $f \in \Schw_{\textup{rad}}(\R^d)$, and define $\mu_r$ by
\[
\int_{\R^d} f(x) \mu_r(x) \, dx = f'(r)
\]
for $f \in \Schw_{\textup{rad}}(\R^d)$. (Note that $\mu_r$ is also supported
on the sphere of radius~$r$ about the origin.)
If we set
\[
T_x = \sum_{n=n_0}^\infty a_n(x) \,\delta_{\sqrt{2n}}+\sum_{n=n_0}^\infty b_n(x) \,\mu_{\sqrt{2n}} - \delta_{|x|},
\]
then
\[
\widehat{T_x} = -\sum_{n=n_0}^\infty \widetilde{a}_n(x) \,\delta_{\sqrt{2n}}
-\sum_{n=n_0}^\infty \widetilde{b}_n(x) \,\mu_{\sqrt{2n}}
\]
by Theorem~\ref{theorem:interpolation}. Thus, $T_x$ is a radial Fourier
quasicrystal.

Dyson \cite{Dyson} highlighted the importance of classifying Fourier
quasicrystals in $\R^1$, and radial Fourier quasicrystals are a natural
generalization of this problem.  In the non-radial case, Fourier
quasicrystals satisfying certain positivity and uniformity hypotheses can be
completely described using Poisson summation \cite{LO}, but even a
conjectural classification remains elusive in general.

\subsection{Proof techniques}

In light of Theorem~\ref{theorem:interpolation} and its interpolation basis,
we can write down the only possible auxiliary function $f$ that could prove a
sharp bound for $E_8$ or the Leech lattice under a potential $p$ (with $d=8$
or $24$, accordingly), at least among radial Schwartz functions:
\begin{equation} \label{eq:defaux}
f(x) = \sum_{n=n_0}^\infty p\big(\sqrt{2n}\big)\,a_n(x)+\sum_{n=n_0}^\infty p'\big(\sqrt{2n}\big)\,b_n(x).
\end{equation}
The proof of Theorem~\ref{theorem:univopt} then amounts to checking that
$f(x) \le p\big(|x|\big)$ for all $x \in \R^d \setminus \{0\}$ and
$\widehat{f}(y) \ge 0$ for all $y \in \R^d$.  (Once we prove these
inequalities, the necessary conditions for a sharp bound become sufficient as
well.) As noted above, it suffices to prove the theorem when $p$ is a
Gaussian, i.e., $p(r) = e^{-\alpha r^2}$ for some constant $\alpha>0$.

Thus, our primary technical contribution is to construct the interpolation
basis. To do so, we analyze the generating functions
\begin{equation} \label{eq:FFourier}
F(\tau,x)=\sum_{n \ge n_0} a_n(x)\,e^{2\pi i n \tau}+2\pi i \tau\sum_{n \ge n_0}\sqrt{2n}\, b_n(x)\,e^{2\pi i n \tau}
\end{equation}
and
\begin{equation} \label{eq:FtFourier}
\widetilde{F}(\tau,x) = \sum_{n \ge n_0} \widetilde{a}_n(x)\,e^{2\pi i n \tau}
+2\pi i \tau\sum_{n \ge n_0}\sqrt{2n}\,\, \widetilde{b}_n(x)\,e^{2\pi i n \tau},
\end{equation}
where $x \in \R^d$ and $\tau$ is in the upper half-plane $\Hyp = \{z \in \C :
\Im(z) > 0\}$.  These generating functions determine the basis, as we show in
\eqref{eq:anformula} and \eqref{eq:bnformula}, and in
Section~\ref{subsec:basisformulas} we prove integral formulas for the basis
functions, which generalize the formulas for $b_1$ from the sphere packing
papers \cite{V,CKMRV}.

One motivation for the seemingly extraneous factors of $2\pi i \tau
\,\sqrt{2n}$ in these generating functions is that they match
\eqref{eq:defaux} with $\tau=i\alpha/\pi$ and $p(r)=e^{-\alpha r^2}=e^{\pi i
\tau r^2}$, because
\[
p\big(\sqrt{2n}\big) = e^{2\pi i n \tau} \qquad\textup{and}\qquad p'\big(\sqrt{2n}\big)
= 2\pi i \tau \,\sqrt{2n} \,e^{2\pi i n \tau}.
\]
In other words, the auxiliary function $f$ from \eqref{eq:defaux} for this
Gaussian potential function is given by $f(x) = F(\tau,x) = F(i\alpha/\pi,x)$.

We can write the interpolation formula for a complex Gaussian $x \mapsto
e^{\pi i \tau |x|^2}$ in terms of $F$ and $\widetilde{F}$. Specifically, the
Fourier transform of $x \mapsto e^{\pi i \tau |x|^2}$ as a function on $\R^d$
is $x \mapsto (i/\tau)^{d/2} e^{\pi i (-1/\tau) |x|^2}$, and hence the
interpolation formula \eqref{eqn:IF} for $x \mapsto e^{\pi i \tau |x|^2}$
amounts to the identity
\begin{equation} \label{eq:interpolation-identity}
F(\tau,x)+(i/\tau)^{d/2}\widetilde{F}(-1/\tau,x)=e^{\pi i \tau |x|^2 }.
\end{equation}
In Section~\ref{sec:sub:31}, we will show using a density argument in
$\Schw_{\textup{rad}}(\R^d)$ that it suffices to prove the interpolation
theorem for complex Gaussians of the form $x \mapsto e^{\pi i \tau |x|^2}$
with $\tau \in\Hyp$.  Thus, constructing the interpolation basis amounts to
solving the functional equation \eqref{eq:interpolation-identity} using
functions $F$ and $\widetilde{F}$ with expansions of the form
\eqref{eq:FFourier} and \eqref{eq:FtFourier}.

These expansions  for $F$ and $\widetilde{F}$ are not quite Fourier
expansions in $\tau$, because they contain terms proportional to $\tau
e^{2\pi i \tau n}$.  In particular, they are not periodic in $\tau$, but they
are annihilated by the second-order difference operator. In other words,
\begin{equation} \label{eq:2nddiffF}
F(\tau+2,x)-2F(\tau+1,x)+F(\tau,x)=0
\end{equation}
and
\begin{equation} \label{eq:2nddiffFt}
\widetilde{F}(\tau+2,x)-2\widetilde{F}(\tau+1,x)+\widetilde{F}(\tau,x)=0.
\end{equation}
Subject to suitable smoothness and growth conditions, proving the
interpolation theorem amounts to constructing functions $F$ and
$\widetilde{F}$ satisfying the functional equations
\eqref{eq:interpolation-identity}, \eqref{eq:2nddiffF}, and
\eqref{eq:2nddiffFt}, as we will show in Theorem~\ref{thm:FEimpliesIF}.

To solve these functional equations, we use Laplace transforms of
quasimodular forms.  This approach was introduced by Viazovska in her
solution of the sphere packing problem in eight dimensions \cite{V}, and we
make heavy use of her techniques.  In \cite{V}, only modular forms of level
at most $2$ and the quasimodular form $E_2$ were needed.  However, these
functions turn out to be insufficient to construct our interpolation basis,
and we must augment them with a logarithm of the Hauptmodul $\lambda$ for
$\Gamma(2)$. Using this enlarged set of functions, we can describe $F$ and
$\widetilde{F}$ explicitly.  The proof of the interpolation theorem then
amounts to verifying analyticity, growth bounds, and functional equations for
our formulas.

Once we have obtained these formulas, only one step remains in the proof of
Theorem~\ref{theorem:univopt}.  Let $p$ be a Gaussian potential function, and
define the auxiliary function $f$ by \eqref{eq:defaux}.  To complete the
proof of Theorem~\ref{theorem:univopt}, we must show that $f(x) \le
p\big(|x|\big)$ for all $x \in \R^d \setminus \{0\}$ and $\widehat{f}(y) \ge
0$ for all $y \in \R^d$.  These inequalities look rather different, but we
will see that they are actually equivalent to each other, thanks to a duality
transformation introduced in Section~6 of \cite{CM}.  We show in
Section~\ref{subsec:sharpbounds} that the underlying inequality
(Proposition~\ref{prop:positivity}) follows from the positivity of the kernel
in the Laplace transform, as well as a truncated version of the kernel when
$d=24$. Unfortunately we have no conceptual proof of this positivity, but we
are able to prove it by combining various analytic methods, including
interval arithmetic computations. This inequality then completes the proof of
universal optimality.

\subsection{Origins of the argument}

The way we present the proof in this paper differs substantially from how we
initially arrived at the ideas. Before we turn to the body of the paper, we will
briefly outline our original approach here, to show how it is rooted in the papers
\cite{V} and \cite{CKMRV}.

We began with the Ansatz that there were basis functions $a_n$, $b_n$,
$\widetilde{a}_n$, and $\widetilde{b}_n$ for which
Theorems~\ref{theorem:interpolation} and~\ref{theorem:interpolation-isom}
held. We then tried to construct these functions using the techniques of
\cite{V} and \cite{CKMRV}, by splitting them into eigenfunctions for the
Fourier transform and constructing each eigenfunction as an integral
transform $\mathcal{T}(\varphi)$ of a corresponding function $\varphi$.
Specifically,
\[
\mathcal{T}(\varphi)(x) = 4 \sin\mathopen{}\big(\pi |x|^2/2\big)^2\mathclose{}\,\int_0^\infty \varphi(it)\, e^{-\pi |x|^2 t} \, dt
\]
for sufficiently large $|x|$, and we extend it to small $|x|$ by analytic
continuation. To ensure that $\mathcal{T}(\varphi)$ is an eigenfunction of
the Fourier transform, the kernel $\varphi$ must satisfy certain
quasimodularity hypotheses. Obtaining eigenfunctions with the desired
behavior at the radii $\sqrt{2m}$ amounts to controlling the asymptotic
behavior of $\varphi(it)$ as $t \to \infty$ up to an additive $o(1)$ term
(i.e., specifying the non-decaying terms in its $q$-expansion).

In each case we examined with eigenvalue $+1$, we could construct a suitable
kernel $\varphi$ by using modular forms for $\SL_2(\Z)$ and the quasimodular
form $E_2$, just as in \cite{V} and \cite{CKMRV}. However, we were unable to
construct the $-1$ eigenfunctions using only modular forms of level $2$, as
we had hoped based on the earlier papers. These modular forms sufficed to
solve the sphere packing problem in dimensions~$8$ and~$24$, but there were
simply not enough of them to construct the putative interpolation basis.
Instead, we found a natural generalization of the proof technique that would
require a special function satisfying certain transformation laws, which we
will state as \eqref{loglambdaslash} below. This function would play a role
analogous to that of $E_2$.

We managed to obtain such a function as an integral, after which we
recognized it as a logarithm $\mathcal{L}$ of the Hauptmodul $\lambda$. Using
this expanded set of functions, we seemed to be able to obtain any given
interpolation basis function we wanted (by constructing both the $+1$ and the
$-1$ eigenfunction components), but the kernel $\varphi$ grew more
complicated as the index $n$ of the basis function increased, and the general
pattern remained unclear. Ad hoc constructions of individual basis functions
did not supply enough information for us to understand the basis as a whole.

To try to clarify the pattern, we sought generating functions
$\mathcal{K}(\tau,z)$ and $\widetilde{\mathcal{K}}(\tau,z)$ along the lines
of \eqref{eq:FFourier} and \eqref{eq:FtFourier} for the kernels, rather than
the interpolation basis functions. For example,
\[
\mathcal{K}(\tau,z) = \sum_{n \ge n_0} \varphi_n(z)\,e^{2\pi i n \tau}+2\pi i
\tau\sum_{n \ge n_0}\sqrt{2n}\, \psi_n(z)\,e^{2\pi i n \tau},
\]
where the kernels $\varphi_n$ and $\psi_n$ satisfy $\mathcal{T}(\varphi_n) =
a_n$ and $\mathcal{T}(\psi_n) = b_n$. Duke and Jenkins \cite{DJ} had analyzed
generating functions for a similar but simpler scenario, and we adapted their
techniques. We guessed by analogy with their work that our generating
functions would have denominator $\Delta(\tau) \Delta(z)^2 (j(\tau) - j(z))$
and numerators given by linear combinations of $1$, $E_2$, and $\mathcal{L}$
with modular forms of certain weights in $\tau$ and $z$ as coefficients. In
other words, the numerators could be parametrized by finitely many
undetermined coefficients. We then solved for the numerators by computing
enough kernels to determine them uniquely. Once we had arrived at explicit
formulas for the generating functions, we proved using techniques inspired by
\cite{DJ} that the formulas were correct, i.e., that the resulting
interpolation basis functions had the desired values and radial derivatives
at the radii $\sqrt{2m}$.

These generating functions provided a construction of the interpolation
basis, but we still did not know that this basis sufficed to reconstruct
every radial Schwartz function. To analyze the basis further, we had to deal
with analytic difficulties. For example, one might naively hope to obtain the
interpolation basis generating function $F(\tau,x)$ from the kernel
generating function $\mathcal{K}(\tau,z)$ when $|x|$ is sufficiently large
via
\[
F(\tau,x) = 4 \sin\mathopen{}\big(\pi |x|^2/2\big)^2\mathclose{}\,\int_0^\infty
\mathcal{K}(\tau,it)\, e^{-\pi |x|^2 t} \, dt,
\]
by interchanging the integral transform with the sum in the generating
function. However, the integral and sum cannot be interchanged, and this
formula is incorrect. Instead, the right side needs an extra $e^{\pi i \tau
|x|^2}$ term. These sorts of analytic issues arise frequently in the study of
$F$ and $\mathcal{K}$, because $z \mapsto \mathcal{K}(\tau,z)$ has poles at
many of the points in the $\SL_2(\Z)$-orbit of $\tau$ (where the
$j(\tau)-j(z)$ factor in the denominator vanishes). Understanding these poles
and their residues turns out to be crucial for our proof of the interpolation
theorem.

As we developed this theory further, our focus shifted to the generating
functions $F$ and $\widetilde{F}$ and their functional equations, and our
original derivation of the kernel generating functions was no longer needed.
In the present paper, the kernel generating functions are analyzed
differently in Section~\ref{sec:analysis}, and it is only after completing
the proof of the interpolation theorem that we show they fulfill their
original purpose in Section~\ref{subsec:basisformulas}. Instead of starting
with the kernel generating function, we begin instead by analyzing the
functional equations for $F$ and $\widetilde{F}$, before relating them to
kernels constructed using modular forms.

\subsection{Organization of the paper}

We begin by collecting background information about modular forms, elliptic
integrals, and radial Schwartz functions in Section~\ref{sec:background}.
Sections~\ref{sec:interpolation} and~\ref{sec:analysis} are the heart of the
paper.  Section~\ref{sec:interpolation} shows how to reduce the interpolation
theorem to the existence of generating functions with certain properties, and
Section~\ref{sec:analysis} describes the generating functions explicitly as
integral transforms of kernels obtained by carefully analyzing an action of
$\PSL_2(\Z)$.  It is not obvious that this construction has all the necessary
properties, and Section~\ref{sec:proofs} completes the proof of the
interpolation theorem by verifying analytic continuation and growth bounds.
Proving universal optimality requires additional inequalities, which are
established in Section~\ref{sec:inequality}.  Finally, we discuss
generalizations and open problems in Section~\ref{sec:generalizations}.

\section{Preliminaries and background on modular forms}
\label{sec:background}

Much of the machinery of our proof rests on properties of classical modular
forms, in particular their growth rates and transformation laws.  This
section summarizes those features which are used later in the paper, as well
as background about elliptic integrals and radial Schwartz functions.  For
more information about modular forms, see
\cite{CohenStromberg,Iwaniec,Serre,Z}.

\subsection{Modular and quasimodular forms}\label{sec:sub:modularforms}

Let $\Hyp$ denote the complex upper half-plane $\{z \in \C : \Im(z) > 0 \}$.
The group $\SL_2(\R)$ acts on $\Hyp$ by  fractional linear transformations
\[
\begin{pmatrix}a & b \\ c & d \end{pmatrix} \cdot z = \frac{az+b}{cz+d}.
\]
For any integer $k$ and $\gamma=\bigl(\begin{smallmatrix} a & b \\ c & d
\end{smallmatrix}\bigr)\in \SL_2(\R)$, define the slash operator on functions
$f\colon\Hyp\rightarrow\C$ by the rule\footnote{When necessary, we will abuse
notation by using a
superscript on the slash notation to disambiguate which variable it applies
to.  For example, we write $f(\tau,z)\mid_2^z(\begin{smallmatrix} 0 & -1 \\ 1 &
\phantom{-}0
\end{smallmatrix})$ or $(f |_2^z (\begin{smallmatrix} 0 & -1 \\ 1 &
\phantom{-}0
\end{smallmatrix}))(\tau,z)$ for $z^{-2}f(\tau,-1/z)$.}
\begin{equation}\label{slashnotation}
(f |_k \gamma) (z) = (cz+d)^{-k} f\left(\frac{az+b}{cz+d}\right).
\end{equation}
We define the \emph{factor of automorphy} $j(\gamma,z)$ by
\begin{equation}
\label{eq:automorphyfactordef}
j(\gamma,z) = cz+d;
\end{equation}
note that it satisfies the identity $j(\gamma_1\gamma_2,z) =
j(\gamma_1,\gamma_2 z)j(\gamma_2, z)$, which implies that
$f|_k(\gamma_1\gamma_2) = (f |_k \gamma_1)|_k \gamma_2$.

Let $\Gamma$ be a discrete subgroup of $\SL_2(\R)$ such that the quotient
$\Gamma \backslash \Hyp$ has finite volume.  Recall that a holomorphic
\emph{modular form} of weight $k$ for $\Gamma$ is a holomorphic function $f
\colon \Hyp \to \C$ such that $f|_k\gamma = f$ for all $\gamma\in\Gamma$, and
that furthermore satisfies a polynomial boundedness condition at each cusp
(see \cite{CohenStromberg,Iwaniec,Serre} for details). The space of
holomorphic modular forms of weight $k$ for $\Gamma$ will be denoted
${\mathcal M}_k(\Gamma)$. It contains the subspace ${\mathcal S}_k(\Gamma)$
of weight $k$ \emph{cusp forms} for $\Gamma$; these are the modular forms
that vanish at all the cusps of $\Gamma \backslash \Hyp$. On the other hand,
we may relax the definition to allow modular functions that are holomorphic
on $\Hyp$ but merely meromorphic at the cusps (equivalently, they satisfy an
exponential growth bound near the cusps). This defines the space ${\mathcal
M}^!_k(\Gamma)$ of \emph{weakly holomorphic} modular forms, which is
infinite-dimensional whenever $\Gamma$ has a cusp.

When $\Gamma$ is a congruence subgroup of $\SL_2(\Z)$, i.e., one that
contains
\[
\Gamma(N) = \left \{ \gamma \in \SL_2(\Z) :
\gamma \equiv \begin{pmatrix} 1 & 0 \\ 0 & 1 \end{pmatrix} \pmod{N} \right \}
\]
for some $N$, the boundedness condition defining ${\mathcal M}_k(\Gamma)$ is
simply that $|(f|_k\gamma)(z)|$ is bounded as $\Im(z) \rightarrow \infty$,
for every $\gamma \in \SL_2(\Z)$.  Similarly, ${\mathcal S}_k(\Gamma)$ is
defined by the condition that $(f|_k\gamma)(z) \to 0$ as $\Im(z) \to \infty$,
while ${\mathcal M}^!_k(\Gamma)$ is defined by the condition that
$|(f|_k\gamma)(z)|$ is bounded above by a polynomial in $e^{\Im(z)}$. Since
$\SL_2(\Z)$ is generated by
\begin{equation} \label{SandTdef}
T = \begin{pmatrix} 1 & 1 \\ 0 & 1 \end{pmatrix} \quad\text{and}\quad S =
\begin{pmatrix} 0 & -1 \\ 1 & \phantom{-}0 \end{pmatrix},
\end{equation}
we will often indicate the action of $\SL_2(\Z)$ on a modular form for a
congruence subgroup by giving the action of the slash operators corresponding
to $S$ and $T$.

The element $-I \in \SL_2(\R)$ acts trivially on $\Hyp$, and thus the action
of $\SL_2(\R)$ descends to an action of $\PSL_2(\R) = \SL_2(\R)/\{\pm I\}$.
We write $\overline{\Gamma}$ to denote the image of a subgroup $\Gamma$ of
$\SL_2(\R)$ in $\PSL_2(\R)$. The group $\PSL_2(\Z) = \overline{\SL_2(\Z)}$
has an elegant presentation in terms of the generators $S$ and $T$, namely
$\PSL_2(\Z) = \langle S,T \mid S^2 = (ST)^3 = I\rangle$, and its subgroup
$\overline{\Gamma(2)}$ of index~$6$ is freely generated by $T^2$ and $ST^2S$
(with cusps $0$, $1$, and $\infty$).

We next describe the structure of the graded rings
\[
{\mathcal M}^!(\Gamma) = \bigoplus_{k\in \Z}{\mathcal M}^!_k(\Gamma)
\]
and
\[
{\mathcal M}(\Gamma) = \bigoplus_{k\in \Z}{\mathcal M}_k(\Gamma)
\]
for the cases of interest in this paper, which are the congruence subgroups
$\Gamma=\Gamma(N)$ for $N=1,2$. In these cases, the weight $k$ of any
nonzero modular form is necessarily even because $-I\in\Gamma$.

\subsubsection{Modular forms for $\Gamma(1)=\SL_2(\Z)$}

Here all modular forms can be described in terms of the Eisenstein series
\[
\aligned
E_{k}(z) & = \frac{1}{2\zeta(k)}\sum_{\substack{(m,n) \in \Z^2\\(m,n)\neq (0,0)}}(mz+n)^{-k}\\
& = 1-\frac{2k}{B_k}\sum_{n=1}^\infty \sigma_{k-1}(n) e^{2\pi i nz}
\endaligned
\]
for even integers $k \geq 4$, where $B_k$ is the $k$-th Bernoulli number and
$\sigma_\ell(n) = \sum_{d \mid n} d^\ell$ is the $\ell$-th power divisor sum
function. The ring ${\mathcal M}(\SL_2(\Z))$ is the free polynomial ring over
$\C$ with generators
\begin{align*}
E_4(z) &=  1 + 240 \sum_{n=1}^\infty \sigma_3(n) q^n \quad \text{and}  \\
E_6(z) &=  1 - 504 \sum_{n=1}^\infty \sigma_5(n) q^n,
\end{align*}
where we use the customary shorthand $q=e^{2\pi i z}$.   In particular,
\begin{equation}\label{modulardimensionformula}
\dim {\mathcal M}_k(\SL_2(\Z)) = \left \lfloor \frac{k}{12} \right \rfloor + \left\{
\begin{array}{ll}
1  & \text{for }k\equiv 0,4,6,8,10\pmod{12}, \text{ and} \\
0 & \text{for }k\equiv 2\pmod{12}
\end{array}
\right.
\end{equation}
for even integers $k\ge 0$, and the identities $E_8=E_4^2$, $E_{10} = E_4E_6$
and $E_{14} = E_4^2E_6$ hold because the modular forms of weight $8$, $10$,
or $14$ form a one-dimensional space. Let
\[
\Delta(z) = \frac{E_4(z)^3 - E_6(z)^2}{1728} = q \prod_{n=1}^\infty (1-q^n)^{24}
\]
denote Ramanujan's cusp form of weight $12$. The product formula shows that
$\Delta$ does not vanish on $\Hyp$ and satisfies the decay estimate
\[
\Delta(x +iy) = O\big(e^{-2\pi y}\big)
\]
for $y \geq 1$, uniformly in $x \in \R$. In particular, $\Delta^{-1}$ is a
weakly holomorphic modular form for $\SL_2(\Z)$. Furthermore, since $\Delta$
vanishes to first order at the unique cusp of $\SL_2(\Z)$, we can use it to
cancel the pole of any form $f \in {\mathcal M}^!(\SL_2(\Z))$: if $f$ has
weight $k$ and a pole of order $r$ at the cusp, then $\Delta^r f \in
{\mathcal M}_{k+12r}(\SL_2(\Z))$.  It follows that
\[
{\mathcal M}^!(\SL_2(\Z)) = \C[E_4, E_6, \Delta^{-1}].
\]
For example, the modular $j$-invariant defined by
\begin{equation}\label{jfunction}
j(z) =  \frac{E_4(z)^3}{\Delta(z)} = q^{-1} + 744 + 196884q + 21493760q^2 + \cdots
\end{equation}
is in ${\mathcal M}_0^!(\SL_2(\Z))$, and its derivative, which is in
${\mathcal M}_2^!(\SL_2(\Z))$, can be expressed as
\begin{equation}\label{jprime}
j'(z) = 2\pi i q \, \frac{dj}{dq} = -2\pi i \frac{E_{14}(z)}{\Delta(z)},
\end{equation}
since both sides share the same leading asymptotics as $\Im(z) \to \infty$ and
$j' \Delta$ lies in the one-dimensional space ${\mathcal M}_{14}(\SL_2(\Z))$.

An important role in this paper (as well as in \cite{V,CKMRV}) is played by
the \emph{quasimodular} form
\begin{equation}\label{logderivDelta}
E_2(z) = 1 - 24 \sum_{n = 1}^\infty \sigma_1(n) q^n =  \frac{1}{2 \pi i } \frac{\Delta'(z)}{\Delta(z)}\,,
\end{equation}
which just barely fails to be modular:
\begin{equation} \label{E2slash}
E_2(z+1) = E_2(z) \qquad \text{and} \qquad
E_2 \mathopen{}\left( \frac{-1}{z} \right)\mathclose{} = z^2 E_2(z) - \frac{6 i z}{\pi}.
\end{equation}
General quasimodular forms for congruence subgroups are polynomials in $E_2$
with modular form coefficients; they may also be obtained by
differentiating modular forms. More precisely, a holomorphic function
$f\colon\Hyp\to\C$ is a quasimodular form for $\Gamma(N)$ of weight $k$ and
depth at most $p$ if it is an element of $\bigoplus_{j=0}^p E_2^j {\mathcal
M}_{k-2j}(\Gamma(N))$. In other words, depth $p$ corresponds to a degree $p$
polynomial in $E_2$ or to taking the $p$-th derivative of a modular form (see
\cite{Choie}, \cite{KZ}, and \cite[Prop.\ 20]{Z}).\footnote{For a more
intrinsic definition of quasimodular form, which applies to non-congruence
subgroups of $\SL_2(\R)$ as well, see \cite[Section 5.3]{Z}.}

\subsubsection{Modular forms for $\Gamma(2)$}
Recall the \emph{Jacobi theta functions}
\[
\aligned
  \Theta_3(z) = \theta_{00}(z) &= \sum_{n \in \Z} e^{\pi i n^2 z}, \\
  \Theta_4(z) = \theta_{01}(z) &= \sum_{n \in \Z} (-1)^n e^{\pi i n^2 z}, \quad \text{and} \\
  \Theta_2(z) = \theta_{10}(z) &= \sum_{n \in \Z} e^{\pi i \left(n+\frac{1}{2}\right)^2 z}
\endaligned
\]
(with their historical numbering), which arise in the classical theory of
theta functions. We define
\begin{equation}\label{UVWdef}
\aligned
  U(z) &= \theta_{00}(z)^4, \\
  V(z) &= \theta_{10}(z)^4, \quad \text{and} \\
  W(z) &= \theta_{01}(z)^4.
\endaligned
\end{equation}
These functions are modular forms of weight $2$ for $\Gamma(2)$, they satisfy
the \emph{Jacobi identity}
\begin{equation}\label{U=V+W}
U = V + W,
\end{equation}
and ${\mathcal M}(\Gamma(2))$ is the polynomial ring generated by $V$ and
$W$.  As was the case for $\Gamma=\Gamma(1)$, multiplication by powers of
$\Delta$ removes singularities at cusps while increasing the weight. Thus any
element of ${\mathcal M}^!(\Gamma(2))$ is again the quotient of an element of
${\mathcal M}(\Gamma(2))$ by some power of $\Delta$, with the behavior at
cusps determined by the numerator and the power of $\Delta$.  (In fact,
${\mathcal M}^!(\Gamma) = {\mathcal M}(\Gamma)[\Delta^{-1}]$ for any
congruence subgroup $\Gamma$, because $\Delta$ is in ${\mathcal M}(\Gamma)$
and vanishes at all cusps.)

The modular forms $U$, $V$, and $W$ transform under $\SL_2(\Z)$ as follows:
\begin{equation}\label{UVWslash}
\aligned
  U |_2T &= W,  & V |_2T &= -V, & W |_2T &= U,  \\
  U |_2S &= -U, & V |_2S &= -W, & W |_2S &= -V.
\endaligned
\end{equation}
These formulas specify how modular forms for $\Gamma(2)$ transform under the
larger group $\SL_2(\Z)$.  Conversely, every modular form for $\SL_2(\Z)$ is
a modular form for $\Gamma(2)$ and thus can be written in terms of $U$, $V$,
and $W$.  For example,
\begin{equation} \label{thetatoEisenstein}
  \begin{aligned}
  E_4 & = \frac{1}{2}(U^2 + V^2 + W^2), \\
  E_6 & = \frac{1}{2}(U+V)(U+W)(W-V), \quad \text{and}\\
  \Delta &= \frac{1}{256}(UVW)^2.
  \end{aligned}
\end{equation}
It will also be convenient to use the holomorphic square root of $\Delta$
defined by
\[
\sqrt{\Delta} = \frac{1}{16}UVW,
\]
which is a modular form of weight $6$ for $\Gamma(2)$.

Eichler and Zagier \cite[Remark after Theorem 8.4]{EZ} showed as part of a
general result that the  algebra $\mathcal M(\Gamma(2))$ is a
six-dimensional free module over ${\mathcal M}_1 := {\mathcal M}(\Gamma(1))$.
Their proof shows that
\begin{equation}\label{EichlerZagier}
{\mathcal M}(\Gamma(2)) = {\mathcal M_1}\oplus U  {\mathcal M_1}\oplus V
{\mathcal M_1}\oplus U^2  {\mathcal M_1}\oplus V^2  {\mathcal M_1}\oplus UVW  {\mathcal M_1},
\end{equation}
with the only subtlety being to show that the modular form $UVW$ of weight
six does not lie in the direct sum of the first five factors.  For
comparison, $UV$ lies in ${\mathcal M_1}\oplus U^2  {\mathcal M_1}\oplus V^2
{\mathcal M_1}$, because $UV = U^2+V^2-E_4$  by \eqref{U=V+W} and
\eqref{thetatoEisenstein}. As $\mathcal M_1$ is itself the free polynomial
ring in $U^2+V^2+W^2$ and $(U+V)(U+W)(W-V)$, the
decomposition~\eqref{EichlerZagier} can also be deduced from the theory of
symmetric polynomials. This decomposition will be used in Section~\ref{sec:annihilatorcalcs}
in the solution of some functional equations.

\subsubsection{The modular function $\lambda$ and its logarithm}\label{sec:sub:sub:modularlambda}

The modular function $\lambda = V/U$ mapping $\Hyp$ to $\C\setminus\{0,1\}$
is a Hauptmodul for the modular curve $X(2) = \Gamma(2) \backslash (\Hyp \cup
\Proj^1(\Q))$ (that is, a generator of its function field over $\C$). For
example, it is related to the Hauptmodul $j$ for $\Gamma(1)$ from
\eqref{jfunction} by
\[
j = \frac{256(1-\lambda + \lambda^2)^3}{\lambda^2(1-\lambda)^2}.
\]
The function $\lambda$ takes the values $0$, $1$, and $\infty$ at the cusps
$\infty$, $0$, and $-1$, respectively, and its restriction $\lambda(it)$ to
the positive imaginary axis decreases from $1$ to $0$ as $t$ increases from
$0$ to $\infty$.  If we let
\begin{equation}\label{lambdaSdef}
\lambda_S(z) := (\lambda |_0 S)(z) = \lambda(-z^{-1}) = 1-\lambda(z),
\end{equation}
then these functions also satisfy the properties
\begin{equation}\label{lambdatransf}
\lambda(z+1) = \frac{\lambda(z)}{\lambda(z)-1} = - \frac{\lambda(z)}{\lambda_S(z)}
\quad \text{and} \quad
\lambda_S(z+1) = \frac{1}{\lambda_S(z)}
\end{equation}
for $z\in \Hyp$.

The nonvanishing of $\lambda$ and $\lambda_S$ on $\Hyp$ allows us to define
\[
{\mathcal L}(z) = \int_0^z \frac{\lambda'(w)}{\lambda(w)}\, dw \quad  \text{and} \quad
{\mathcal L}_S(z) =-\int_z^\infty \frac{\lambda_S'(w)}{\lambda_S(w)}\,dw,
\]
where the contours are chosen to approach the singularities $0$ or $\infty$
on vertical lines. These functions satisfy
\[
{\mathcal L}(it) = \log(\lambda(it)) \quad \text{and} \quad {\mathcal L}_S(it) = \log(\lambda_S(it)) = \log(1-\lambda(it))
\]
for $t>0$, and as such are holomorphic functions for which $e^{\mathcal
L}=\lambda$ and $e^{\mathcal L_S}=\lambda_S$; however, they are not in
general the principal branches of the logarithms of $\lambda$ or $\lambda_S$.
We note the asymptotics
\begin{equation}\label{lambdaasympt}
\aligned
   \mathcal L(z) &= \pi i z +4 \log (2)-8
   q^{1/2}+O(q) \quad \text{and}\\
 \mathcal L_S(z) &= -16 q^{1/2}-\frac{64 q^{3/2}}{3}+O\mathopen{}\left(q^{5/2}\right)\mathclose{}
\endaligned
\end{equation}
as $q \rightarrow 0$, where $q^{n/2} = e^{2\pi i n z/2}$.

The functions $\mathcal{L}$ and $\mathcal{L}_S$ have the transformation
properties
\begin{equation}\label{loglambdaslash}
\aligned
{\mathcal L} |_0T &={\mathcal L} -  {\mathcal L}_S + \pi i, &
 {\mathcal L}_S |_0T &= - {\mathcal L}_S, \\
   {\mathcal L} |_0S &=  {\mathcal L}_S,  & \ \
 {\mathcal L}_S |_0S &={\mathcal L}.
\endaligned
\end{equation}
Indeed, the last pair of assertions follows directly from the definitions and
holomorphy.  The first two assertions, which read
\[
{\mathcal L}(z+1) = {\mathcal L}(z) - {\mathcal L}_S(z) + \pi i
\quad \text{and} \quad
{\mathcal L}_S(z+1) = -{\mathcal L}_S(z),
\]
are proved by showing that the derivatives of both sides are equal (using the
derivatives of the identities in \eqref{lambdatransf}) and by comparing the
asymptotics \eqref{lambdaasympt} to determine the constant of integration.

\subsection{Elliptic integrals}\label{sub:sec:elliptic}

We normalize the \emph{complete elliptic integral of the first kind} by
\[
K(m) = \int_0^{\pi/2} \frac{d\theta}{\sqrt{1 - m \sin^2 \theta}}
\]
and \emph{of the second kind} by
\[
E(m) = \int_0^{\pi/2} \sqrt{1 - m \sin^2 \theta} \, d \theta.
\]
Note that many references, such as \cite[Chapter~19]{NIST}, define $K$ and
$E$ in terms of the elliptic modulus $k$, so that the complete elliptic
integrals are what we call $k \mapsto K(k^2)$ and $k \mapsto E(k^2)$.  Our
normalization is less principled from the perspective of elliptic
function theory, but it has the advantage of simplifying various expressions
that occur later in our paper.

These elliptic integrals satisfy a plethora of beautiful identities, a few of
which we list below. First, $E$ and $K$ are related by
\begin{equation} \label{Kderiv}
K'(m) = \frac{E(m)}{2m(1-m)} - \frac{K(m)}{2m}.
\end{equation}
(Here, $K'$ denotes the derivative of $K$. In the elliptic function
literature, $K'$ is often used instead to denote the elliptic integral with
respect to the complementary modulus $k'=\sqrt{1-k^2}$.) Legendre proved the
identity
\[
K(m)E(1-m) + E(m) K(1-m) - K(m)K(1-m) = \frac{\pi}{2}
\]
(see \cite[pp. 68--69]{MM}). The two identities above can be combined to
obtain
\begin{equation} \label{Kderividentity}
K(m) K'(1-m) + K'(m) K(1-m) = \frac{\pi}{4 m(1-m)}.
\end{equation}
In other words, the Wronskian of $K$ and $m \mapsto K(1-m)$ has a simple
form.

Elliptic integrals are also related to the modular forms of
Section~\ref{sec:sub:modularforms} via classical identities dating back to
Jacobi. For $z\in\Hyp$ on the imaginary axis, the key identity is the
inversion formula
\[
\Theta_3(z)^2 =2\pi^{-1}K(\lambda(z))
\]
(see \cite[\S21.61 and \S22.301]{WW} or \cite[Theorem~5.8]{Segal}). It
follows that for such $z$,
\begin{equation} \label{thetafromK}
\aligned
U(z) & = 4\pi^{-2}K(\lambda(z))^2, \\
V(z) & = 4\pi^{-2}\lambda(z)K(\lambda(z))^2, \quad \text{and} \\
W(z) & = 4\pi^{-2}\lambda_S(z)K(\lambda(z))^2,
\endaligned
\end{equation}
because $U = \Theta_3^4$, $\lambda=V/U$, and $U=V+W$. Using
\eqref{lambdaSdef} and Jacobi's transformation law
\[
\Theta_3(-1/z)^2 = -iz \Theta_3(z)^2
\]
(which follows immediately from Poisson summation), we obtain
\begin{equation}\label{zintermsofKlambda}
\frac{K(1-\lambda(z))}{K(\lambda(z))} = - iz.
\end{equation}
Differentiation combined with \eqref{Kderividentity} yields the identity
\[
\lambda'(z) = 4i\pi^{-1}\lambda(z) (1-\lambda(z)) K(\lambda(z))^2.
\]
Finally, we can use \eqref{logderivDelta}, \eqref{thetatoEisenstein},
\eqref{Kderiv}, and \eqref{thetafromK} to write
\begin{equation} \label{E2fromK}
E_2(z)  =  4 \pi^{-2}K(\lambda(z))  \Big( 3 E(\lambda(z)) - (2-\lambda(z)) K(\lambda(z)) \Big).
\end{equation}
In Section~\ref{sec:inequality} we will use equations
\eqref{thetatoEisenstein}, \eqref{thetafromK}, \eqref{zintermsofKlambda}, and
\eqref{E2fromK} to write the restriction of elements of ${\mathcal
M}^!(\Gamma(2))$ (and related quasimodular expressions) to the imaginary axis
in terms of elliptic integrals of $\lambda$.

The elliptic integrals $E$ and $K$ are holomorphic in the open unit disk.
Their behavior near $1$ is governed by
\begin{equation}\label{EandKnear1}
E(1-z) = A_1(z)+A_2(z)\log(z)\quad \text{and} \quad K(1-z) = A_3(z)+A_4(z)\log(z),
\end{equation}
where each $A_j$ is a holomorphic function on the open unit disk with real
Taylor coefficients about the origin (see \cite[Section 19.12]{NIST} for
explicit formulas). Furthermore, $A_1$ and $A_3$ have nonnegative
coefficients, while $A_2$ and $A_4$ have nonpositive coefficients.

\subsection{Radial Schwartz functions}

For a smooth function $f$ on $\R^d$, define the \emph{Schwartz seminorms} by
\[
\|f\|_{\alpha,\beta} = \sup_{x \in \R^d} \big| x^\alpha \partial^\beta f(x)\big|,
\]
for $\alpha,\beta \in \Z_{\ge 0}^d$, where we use the multi-index notation
\[
x^\alpha = x_1^{\alpha_1} \dotsb x_d^{\alpha_d} \quad\textup{and}\quad
\partial^\beta = \left(\frac{\partial}{\partial x_1}\right)^{\beta_1} \dotsb  \left(\frac{\partial}{\partial x_d}\right)^{\beta_d}.
\]
By definition, $f$ is a Schwartz function if and only if
$\|f\|_{\alpha,\beta} < \infty$ for all $\alpha$ and $\beta$, and these
seminorms define the Schwartz space topology on $\Schw(\R^d)$.  The radial
Schwartz space $\Schw_{\textup{rad}}(\R^d)$ is the subspace of radial
functions in $\Schw(\R^d)$, with the induced topology.

\begin{lemma} \label{lem:Schwartz+radial}
A Schwartz function $f$ on $\R^d$ is radial if and only if there exists an
even Schwartz function $f_0$ on $\R$ such that $f(x) = f_0\big(|x|\big)$ for
all $x \in \R^d$.  Furthermore, the map $f_0 \mapsto f$ is an isomorphism of
topological vector spaces from $\Schw_{\textup{rad}}(\R^1)$ to
$\Schw_{\textup{rad}}(\R^d)$.
\end{lemma}

Of course $\Schw_{\textup{rad}}(\R^1)$ consists of the even functions in
$\Schw(\R^1)$.  For a proof of Lemma~\ref{lem:Schwartz+radial} (in fact, of a
slightly stronger result), see \cite[Section~3]{GT}.

For our purposes, the significance of Lemma~\ref{lem:Schwartz+radial} is that
we can restrict our attention to radial derivatives when dealing with radial
Schwartz functions. Let $D$ denote the radial derivative, defined by $Df(x) =
f_0'\big(|x|\big)$, and define the \emph{radial seminorms} by
\[
\|f\|^\textup{rad}_{k,\ell} = \sup_{x \in \R^d} |x|^k \big|D^\ell f(x)\big|
\]
for $k,\ell \in \Z_{\ge 0}$.  (Note that $\Schw_{\textup{rad}}(\R^d)$ is not
closed under $D$, because the derivative of an even function is odd, but
$D^\ell f(x)$ is nevertheless well defined.)  Then
Lemma~\ref{lem:Schwartz+radial} tells us that a smooth, radial function $f$
is a Schwartz function if and only if $\|f\|^\textup{rad}_{k,\ell} < \infty$
for all $k$ and $\ell$, and these seminorms characterize the topology of
$\Schw_{\textup{rad}}(\R^d)$.

We will also need the first part of the following lemma, which we prove using
the techniques from \cite[Section~6]{RV}. See also Lemma~2 and Remark~4 in
\cite{J-M} for a more general density result.

\begin{lemma} \label{lemma:dense}
The complex Gaussians $x \mapsto e^{\pi i \tau|x|^2}$ with $\tau \in \Hyp$
span a dense subspace of $\Schw_{\textup{rad}}(\R^d)$.  In fact, for any
$y>0$, the same is true if we use only complex Gaussians with $\Im(\tau)=y$.
\end{lemma}

\begin{proof}
Compactly supported functions are dense in $\Schw_{\textup{rad}}(\R^d)$, as
is easily shown by multiplying by a suitable bump function.  Thus it will
suffice to show that compactly supported, smooth, radial functions can be
approximated arbitrarily well by linear combinations of complex Gaussians
with $\Im(\tau)=y$.

Removing a factor of $e^{\pi y |x|^2}$ shows that every compactly supported,
smooth, radial $f$ on $\R^d$ can be written as
\[
f(x) = g(|x|^2)e^{-\pi y |x|^2},
\]
where $g$ is a smooth, compactly supported function on $\R$.  Let
$\widehat{g}$ be its one-dimensional Fourier transform
\[
\widehat{g}(t) = \int_{\R} g(x) e^{-2\pi i t x} \, dx.
\]
Then
\[
g(|x|^2) = \int_{\R} \widehat{g}(t) e^{2\pi i t |x|^2} \, dt = \lim_{T \to \infty} \int_{-T}^T \widehat{g}(t) e^{2\pi i t |x|^2} \, dt
\]
by Fourier inversion, and hence
\[
f(x) = \lim_{T \to \infty} \int_{-T}^T \widehat{g}(t) e^{\pi i(2t+iy) |x|^2} \, dt.
\]
The functions
\begin{equation} \label{eq:integral-approx}
x \mapsto \int_{-T}^T \widehat{g}(t) e^{\pi i(2t+iy) |x|^2} \, dt
\end{equation}
are Schwartz functions that converge to $f$ in $\Schw_{\textup{rad}}(\R^d)$
as $T \to \infty$, because $\widehat{g}$ is rapidly decreasing and we can
therefore control the radial Schwartz seminorms of the error term
\[
x \mapsto \int_{\R\setminus[-T,T]} \widehat{g}(t) e^{\pi i(2t+iy) |x|^2} \, dt.
\]
Furthermore, for each $T$, equally spaced Riemann sums for
\eqref{eq:integral-approx} converge to this function in
$\Schw_{\textup{rad}}(\R^d)$, by the usual error estimate in terms of the
derivative.  These Riemann sums are linear combinations of complex Gaussians
with $\tau = 2t+iy$ for different values of $t$ in $\R$, as desired.
\end{proof}

\section{Functional equations and the group algebra $\C[\PSL_2(\Z)]$}
\label{sec:interpolation}

\subsection{From interpolation to functional equations and back}
\label{sec:sub:31}

As mentioned in \eqref{eq:FFourier} and \eqref{eq:FtFourier}, we consider the
generating functions
\begin{equation} \label{eq2:FFourier}
F(\tau,x)=\sum_{n \ge n_0} a_n(x)\,e^{2\pi i n \tau}+2\pi i \tau\sum_{n \ge n_0}\sqrt{2n}\, b_n(x)\,e^{2\pi i n \tau}
\end{equation}
and
\begin{equation} \label{eq2:FtFourier}
\widetilde{F}(\tau,x) = \sum_{n \ge n_0} \widetilde{a}_n(x)\,e^{2\pi i n \tau}+2\pi i \tau\sum_{n \ge n_0}\sqrt{2n}\,\, \widetilde{b}_n(x)\,e^{2\pi i n \tau}
\end{equation}
for the interpolation basis, where $x \in \R^d$ and $\tau$ is in the upper
half-plane $\Hyp = \{z \in \C : \Im(z) > 0\}$.  In equations
\eqref{eq:interpolation-identity}--\eqref{eq:2nddiffFt}, we derived
functional equations for $F$ and $\widetilde{F}$ from the existence of an
interpolation basis.  We now show that the converse holds as well: the
existence of a well-behaved solution to the functional equations
\eqref{eq:interpolation-identity}--\eqref{eq:2nddiffFt} implies an
interpolation theorem, regardless of the dimension.

\begin{theorem}\label{thm:FEimpliesIF}
Suppose there exist smooth functions
$F,\widetilde{F}\colon\Hyp\times\R^d\to\C$ such that
\begin{enumerate}
\item $F(\tau,x)$ and $\widetilde{F}(\tau,x)$ are holomorphic in $\tau$,
\item $F(\tau,x)$ and $\widetilde{F}(\tau,x)$ are radial in $x$,
\item\label{new part 3 of theorem F} for all nonnegative integers $k$ and
    $\ell$, the radial derivative $D_x$ with respect to $x$ satisfies the
    uniform bounds
\[
|x|^k \big|D_x^\ell F(\tau,x)\big| <
\alpha_{k,\ell}\Im(\tau)^{-\beta_{k,\ell}}+\gamma_{k,\ell}|\tau|^{\delta_{k,\ell}}
\]
and
\[
|x|^k \big|D_x^\ell\widetilde{F}(\tau,x)\big| <
\alpha_{k,\ell}\Im(\tau)^{-\beta_{k,\ell}}+\gamma_{k,\ell}|\tau|^{\delta_{k,\ell}}
\]
for some nonnegative constants $\alpha_{k,\ell}$, $\beta_{k,\ell}$,
$\gamma_{k,\ell}$, and $\delta_{k,\ell}$,
\item\label{new part 4 of theorem F} in the special case $(k,\ell)=(0,0)$,
\[
\max\mathopen{}\big(|F(\tau,x)|,|\widetilde{F}(\tau,x)|\big)\mathclose{} \le \alpha_{0,0}\Im(\tau)^{-\beta_{0,0}}
\]
for $-1\le \Re(\tau)\le 1$ and $x\in \R^d$, with $\beta_{0,0}>0$, and
\item\label{new part 5 of theorem F} $F$ and $\widetilde{F}$ satisfy the
    functional equations
    \eqref{eq:interpolation-identity}--\eqref{eq:2nddiffFt}, i.e.,
\begin{align*}
F(\tau+2,x)-2F(\tau+1,x)+F(\tau,x)&=0,\\
\widetilde{F}(\tau+2,x)-2\widetilde{F}(\tau+1,x)+\widetilde{F}(\tau,x)&=0, \textup{ and}\\
F(\tau,x)+(i/\tau)^{d/2}\widetilde{F}(-1/\tau,x)&=e^{\pi i \tau |x|^2 }.
\end{align*}
\end{enumerate}
Then $F$ and $\widetilde{F}$ have expansions of the form \eqref{eq2:FFourier}
and \eqref{eq2:FtFourier} with $n_0=1$, for some radial Schwartz functions
$a_n,b_n,\widetilde{a}_n,\widetilde{b}_n$. Moreover, for every radial
Schwartz function $f\colon\R^d\to\R$, the interpolation formula
\[
\begin{split}
f(x) &= \sum_{n=1}^\infty f\big(\sqrt{2n}\big)\,a_n(x)+\sum_{n=1}^\infty f'\big(\sqrt{2n}\big)\,b_n(x)\\
&\quad\phantom{}+\sum_{n=1}^\infty \widehat{f}\big(\sqrt{2n}\big)\,\widetilde{a}_n(x)+\sum_{n=1}^\infty \widehat{f}\,'\big(\sqrt{2n}\big)\,\widetilde{b}_n(x),
\end{split}	
\]
holds, and the right side converges absolutely.  Finally, for fixed $k$ and
$\ell$, the radial seminorms
\begin{align*}
\sup_{x \in \R^d} |x|^k \big|a_n^{(\ell)}(x)\big|, \qquad
&\sup_{x \in \R^d} |x|^k \big|b_n^{(\ell)}(x)\big|,\\
\sup_{x \in \R^d} |x|^k \big|\widetilde{a}_n^{(\ell)}(x)\big|, \qquad
&\sup_{x \in \R^d} |x|^k \big|\widetilde{b}_n^{(\ell)}(x)\big|
\end{align*}
all grow at most polynomially in $n$.

Furthermore, $a_1 = \widetilde{a}_1 = b_1 = \widetilde{b}_1=0$ if and only if
$F(\tau,x)$ and $\widetilde{F}(\tau,x)$ are $o\big(e^{-2\pi \Im(\tau)}\big)$
as $\Im(\tau)\rightarrow \infty$ in the strip $-1\le \Re(\tau) \le 1$ with
$x$ fixed.
\end{theorem}

This last statement concerns starting the interpolation formula at $n_0=2$,
which Theorem~\ref{theorem:interpolation} asserts is the case for $d=24$. The
separate condition \eqref{new part 4 of theorem F} is important for ruling
out a contribution from $n=0$ in the interpolation formula; the restriction
to the strip $-1\le \Re(\tau)\le 1$ is because generic solutions to the
recurrences in part~\eqref{new part 5 of theorem F} grow linearly in
$\Re(\tau)$ (see \eqref{periodicplustauperiodic}).

\begin{proof}
We begin by obtaining the expansion of $F$. The difference
\[
F(\tau+1,x)-F(\tau,x)
\]
is holomorphic in $\tau$ and invariant under $\tau \mapsto \tau+1$, by the
functional equation
\[
F(\tau+2,x)-2F(\tau+1,x)+F(\tau,x)=0.
\]
Thus, for each $x$ there is a holomorphic function $g_x$ on the punctured
disk $\{z \in \C : 0 < |z| < 1\}$ such that
\[
F(\tau+1,x)-F(\tau,x) = g_x(e^{2\pi i \tau}),
\]
because $\tau \mapsto e^{2\pi i \tau}$ is a covering map from $\Hyp$ to the
punctured disk. Furthermore, it follows from part \eqref{new part 4 of
theorem F} of the hypotheses that
\[
\lim_{z \to 0} g_x(z)=0,
\]
and thus $g_x$ extends to a holomorphic function that vanishes at $0$. By
taking the Taylor series of $g_x$ about $0$, we obtain coefficients $b_n(x)$
for $n \ge 1$ such that
\begin{equation} \label{eq:bcoeff}
F(\tau+1,x)-F(\tau,x) = 2\pi i \sum_{n \ge 1}\sqrt{2n}\, b_n(x)\,e^{2\pi i n \tau}.
\end{equation}
To obtain $a_n(x)$, we remove the $b_n(x)$ terms and instead look at
\[
F(\tau,x) - \tau\big(F(\tau+1,x)-F(\tau,x)\big),
\]
which is again holomorphic in $\tau$ and invariant under $\tau \mapsto
\tau+1$.  The parenthetical expression decays exponentially as
$\Im(\tau)\rightarrow \infty$ by \eqref{eq:bcoeff}, so the bound in
part~\eqref{new part 4 of theorem F} again yields coefficients $a_n(x)$ for
$n \ge 1$ such that
\begin{equation} \label{eq:acoeff}
F(\tau,x) - \tau\big(F(\tau+1,x)-F(\tau,x)\big) = \sum_{n \ge 1} a_n(x)\,e^{2\pi i n \tau}.
\end{equation}
Thus,
\begin{equation}\label{periodicplustauperiodic}
F(\tau,x)=\sum_{n \ge 1} a_n(x)\,e^{2\pi i n \tau}+2\pi i \tau\sum_{n \ge 1}\sqrt{2n}\, b_n(x)\,e^{2\pi i n \tau}.
\end{equation}
Because of the symmetry of the hypotheses, the case of $\widetilde{F}$ works
exactly the same way, with coefficients $\widetilde{a}_n(x)$ and
$\widetilde{b}_n(x)$.  The assertion at the end of the theorem statement
about when $a_1 = \widetilde{a}_1 = b_1 = \widetilde{b}_1=0$ is then an
immediate consequence of formula~\eqref{periodicplustauperiodic} and its
counterpart for $\widetilde{F}$.

To check that the coefficients are radial Schwartz functions, we note that
for any $y>0$,
\begin{equation} \label{eq:anformula}
a_n(x) = \int_{-1+iy}^{iy} \left(F(\tau,x)
- \tau\big(F(\tau+1,x)-F(\tau,x)\big)\right) e^{-2\pi i n \tau} \, d\tau
\end{equation}
and
\begin{equation} \label{eq:bnformula}
b_n(x) = \frac{1}{2\pi i \sqrt{2n}} \int_{-1+iy}^{iy}
\big(F(\tau+1,x)-F(\tau,x)\big) e^{-2\pi i n \tau} \, d\tau
\end{equation}
by orthogonality using \eqref{eq:acoeff} and \eqref{eq:bcoeff}. We can take
radial derivatives in $x$ under the integral sign, because all the
derivatives are continuous. If we do so and apply part \eqref{new part 3 of
theorem F} of the hypotheses, we find that the radial seminorms
\[
\sup_{x \in \R^d} |x|^k \big|a_n^{(\ell)}(x)\big| \quad\textup{and}\quad \sup_{x \in \R^d} |x|^k \big|b_n^{(\ell)}(x)\big|
\]
are all finite, for any $k$, $\ell$, and $n$.  Thus, $a_n$ and $b_n$ are
Schwartz functions (see Lemma~\ref{lem:Schwartz+radial}). Furthermore, if we
take $y=1/n$ and integrate over the straight line from $-1+iy$ to $iy$, we
find that these radial seminorms of $a_n$ and $b_n$ grow at most polynomially
in $n$ for each $k$ and $\ell$. By symmetry, the same holds for
$\widetilde{a}_n$ and $\widetilde{b}_n$ as well.

These estimates imply that the sum
\[
\begin{split}
&\sum_{n=1}^\infty f\big(\sqrt{2n}\big)\,a_n(x)+\sum_{n=1}^\infty f'\big(\sqrt{2n}\big)\,b_n(x)\\
&\phantom{}+\sum_{n=1}^\infty \widehat{f}\big(\sqrt{2n}\big)\,\widetilde{a}_n(x)+\sum_{n=1}^\infty \widehat{f}\,'\big(\sqrt{2n}\big)\,\widetilde{b}_n(x),
\end{split}	
\]
converges absolutely whenever $f$ is a radial Schwartz function, and that
this formula defines a continuous linear functional on
$\Schw_{\textup{rad}}(\R^d)$.

All that remains is to prove the interpolation formula. Fix $x_0 \in \R^d$,
and define the functional $\Lambda$ on $\Schw_{\textup{rad}}(\R^d)$ by
\[
\begin{split}
\Lambda(f) &= \sum_{n=1}^\infty f\big(\sqrt{2n}\big)\,a_n(x_0)+\sum_{n=1}^\infty f'\big(\sqrt{2n}\big)\,b_n(x_0)\\
&\quad\phantom{}+\sum_{n=1}^\infty \widehat{f}\big(\sqrt{2n}\big)\,\widetilde{a}_n(x_0)+\sum_{n=1}^\infty \widehat{f}\,'\big(\sqrt{2n}\big)\,\widetilde{b}_n(x_0)\\
&\quad\phantom{}-f(x_0),
\end{split}	
\]
so that the interpolation formula for $x_0$ is equivalent to $\Lambda=0$.
Because $\Lambda$ is continuous, it suffices to prove that $\Lambda(f)$
vanishes when $f$ is a complex Gaussian, i.e., $f(x) = e^{\pi i \tau|x|^2}$
with $\tau \in \Hyp$, by Lemma~\ref{lemma:dense}.  This condition amounts to
the function equation
\[
F(\tau,x_0)+(i/\tau)^{d/2}\widetilde{F}(-1/\tau,x_0)=e^{\pi i \tau |x_0|^2 },
\]
because $\widehat{f}(x) = (i/\tau)^{d/2} e^{\pi i (-1/\tau) |x|^2}$.  Thus,
the interpolation formula holds for all radial Schwartz functions, as
desired.
\end{proof}

Theorem~\ref{thm:FEimpliesIF} reduces Theorem~\ref{theorem:interpolation} to
constructing $F$ and $\widetilde{F}$, but the only hint it gives for how to
do so is the functional equations they must satisfy.  The rest of
Sections~\ref{sec:interpolation} and~\ref{sec:analysis} consists of a
detailed study of these functional equations, in terms of the right action of
$\PSL_2(\Z)$ via the slash operator $|^\tau_{d/2}$ in the $\tau$ variable
(assuming $d/2$ is an even integer). Using the standard generators $S$ and
$T$ from \eqref{SandTdef}, the equation
\[
F(\tau,x)+(i/\tau)^{d/2}\widetilde{F}(-1/\tau,x)=e^{\pi i \tau |x|^2 }
\]
expresses $F$ in terms of $\widetilde{F} |^\tau_{d/2} S$ and vice versa
(because $S^2=I$).  It therefore suffices to construct $F$, from which we can
obtain $\widetilde{F}$.  The remaining functional equations are best stated
in terms of the linear extension of the slash operator action
\eqref{slashnotation} of $\PSL_2(\Z)$ to the group algebra
\[
R = \C[\PSL_2(\Z)]
\]
of finite formal linear combinations of elements of
$\PSL_2(\Z)$. The equation
\[
F(\tau+2,x)-2F(\tau+1,x)+F(\tau,x) = 0
\]
says $F$ is annihilated by the element $(T-I)^2 = T^2 - 2T + I$ of $R$, while
\[
\widetilde{F}(\tau+2,x)-2\widetilde{F}(\tau+1,x)+\widetilde{F}(\tau,x)=0
\]
specifies the action of $S(T-I)^2$ on $F$ (see \eqref{rerestated1.6to1.8}
below). Thus, the functional equations specify the action of the right ideal
${\mathcal I} = (T-I)^2 \cdot R + S (T-I)^2 \cdot R$ on $F$. This ideal does
not amount to all of $R$, and in fact $\dim_\C (R/{\mathcal I})=6$, as we
will see in Proposition~\ref{prop:ideals2}. (For simplicity we write
$R/\mathcal{I}$ rather than $\mathcal{I} \backslash R$, despite the fact that
$\mathcal{I}$ is a right ideal.) To make further progress, we must understand
the structure of $\mathcal I$ and the action of $\PSL_2(\Z)$ on the
six-dimensional vector space $R/{\mathcal I}$.

\subsection{A six-dimensional representation of $\slz$}\label{sec:sub:6dimlrepn}

Many of our arguments use facts about a particular six-dimensional
representation $\sigma$ of $\slz$, which we collect here for later reference.
We will see in \eqref{eq:sigmaVmeaning} that this representation describes
the action of $\PSL_2(\Z)$ on $R/{\mathcal I}$.

Recall that $\PSL_2(\Z)$ has the subgroup $\overline{\Gamma(2)}$ of
index~$6$, which is freely generated by $T^2 = \bigl(\begin{smallmatrix} 1 &
2 \\ 0 & 1
\end{smallmatrix}\bigr)$ and $ST^2S = \bigl(\begin{smallmatrix} -1 & 0 \\ 2 &
-1
\end{smallmatrix}\bigr)$.  The following
lemma gives some standard bounds on the length of a word in these generators
by the size of the matrix, in a way which will be useful for later
applications such as Proposition~\ref{prop:propagationandextension}. It
follows from work of Eichler \cite{Ei}, but we give a direct proof. The idea
of column domination that appears in part~(1) dates back at least to Markov's
1954 book on algorithms \cite[Chapter~VI \S10 2.5]{Markov}.

Let $\|\gamma\|_{\textup{Frob}}=(\operatorname{Tr}(\gamma\gamma^t))^{1/2}$
denote the Frobenius norm of $\gamma\in\SL_2(\R)$, i.e.,
\[
\left\|\begin{pmatrix}
a &b \\
c & d
\end{pmatrix}\right\|_{\textup{Frob}}^2 = a^2+b^2+c^2+d^2.
\]
We apply this norm also to elements of $\PSL_2(\R)$, because
$\|\gamma\|_{\textup{Frob}} = \|{-\gamma}\|_{\textup{Frob}}$. Recall that it
is submultiplicative: $\|\gamma \gamma'\|_{\textup{Frob}} \le \|\gamma\|_{\textup{Frob}}  \| \gamma'\|_{\textup{Frob}}$
for all $\gamma$ and $\gamma'$.

\begin{lemma}\label{lem:columndomination}
Let $\gamma_1=T^2$ and $\gamma_2=ST^2S$, so that every element $\gamma\in
\overline{\Gamma(2)} \subseteq \PSL_2(\Z)$ has a unique expression as a
finite reduced word $\gamma_1^{e_1}\gamma_2^{f_1}\gamma_1^{e_2}\cdots$, with
each $e_i,f_i\in\Z\setminus\{0\}$ except perhaps $e_1=0$.
\begin{enumerate}
\item (Column domination property) The second column of $\gamma$ has
    strictly greater Euclidean norm than the first column if and only if
    $\gamma$'s reduced word ends in a nonzero power of $\gamma_1=T^2$.
\item The Frobenius norm of $\gamma$ satisfies
\[
|e_1|+|f_1|+|e_2|+\cdots \le \|\gamma\|_{\textup{Frob}}^2\le
(2+4e_1^2)(2+4f_1^2)(2+4e_2^2)\cdots.
\]
\item The initial subwords
\[
\gamma_1^{\operatorname{sgn}(e_1)},\gamma_1^{2\operatorname{sgn}(e_1)},\ldots,\gamma_1^{e_1},
\gamma_1^{e_1}\gamma_2^{\operatorname{sgn}(f_1)},\dots,\gamma_1^{e_1}\gamma_2^{f_1},\dots,\gamma
\]
of the reduced word of $\gamma$ have strictly increasing Frobenius norms.
\end{enumerate}
\end{lemma}

Of course the column vectors of an element of $\PSL_2(\Z)$ are defined only
modulo multiplication by $\pm1$, but that suffices for their norms to be well
defined. The bounds in part (2) are not sharp, but they will suffice for our
purposes.

\begin{proof}
Note that conjugating by $S$ maps $\left(\begin{smallmatrix}
a & b \\
c & d
\end{smallmatrix}
\right)$ to $\left(\begin{smallmatrix}
\phantom{-}d & \,-c \\
-b & \,\phantom{-}a
\end{smallmatrix}
\right)$, which interchanges the column norms as well as the generators
$\gamma_1 = T^2$ and $\gamma_2 = ST^2S$.  Because of this symmetry, part~(1)
implies that the first column of $\gamma$ has strictly greater Euclidean norm
than the second column if and only if $\gamma$'s reduced word ends in a
nonzero power of $\gamma_2=ST^2S$, and that only the identity element of
$\overline{\Gamma(2)}$ has columns of the same Euclidean norm.  We will prove
these three statements together by induction on the total number of factors
$\gamma_1^{e_i}$ or $\gamma_2^{f_i}$ in $\gamma$'s reduced word. The base
case of powers $\gamma_1^{e_1}$ and $\gamma_2^{f_1}$ is straightforward. For
the inductive step, by symmetry we reduce to the case where $\gamma$ has the
form
\[
\gamma=\begin{pmatrix} a & b \\
c & d
\end{pmatrix}\begin{pmatrix} 1 & 2e \\
0 & 1
\end{pmatrix}=
\begin{pmatrix}
  a & b+2ea \\
  c & d+2ec
\end{pmatrix},
\]
where $\left(\begin{smallmatrix}
a & b \\
c & d
\end{smallmatrix}
\right)\in \PSL_2(\Z)$ is not the identity element, $e$ is a nonzero integer,
and $a^2+c^2>b^2+d^2$ by the inductive assumption. The Euclidean norm squared
of $\gamma$'s second column is
\[
h(e) :=
b^2+d^2+4e^2(a^2+c^2)+4 e(ab+cd),
\]
and we must show that $h(e) > a^2+c^2$ when $e \ne 0$.  Because $h(e)$ is
quadratic in $e$ and $h(0) < a^2+c^2$, it suffices to show that $h(\pm 1) >
a^2+c^2$.  Indeed,
\begin{align*}
h(\pm 1) & \ge  b^2+d^2+4(a^2+c^2)- 4  \sqrt{a^2+c^2}\sqrt{b^2+d^2}\\
& = \big(2\sqrt{a^2+c^2}-\sqrt{b^2+d^2}\big)^2 > a^2+c^2,
\end{align*}
as desired, where the first inequality follows from the Cauchy-Schwarz
inequality.  This proves part~(1).

The upper bound in part~(2) follows from the submultiplicativity of the
Frobenius norm.  The lower bound follows from part~(3), because there are
$|e_1|+|f_1|+|e_2|+\cdots$ initial subwords.  Finally, part (3) follows from
part~(1) and the fact that $h(e)$ increases for $e\ge 1$ and decreases for $e
\le -1$ (it is a quadratic function of $e$ whose minimum occurs between
$e=-1$ and $e=1$).
\end{proof}

We next introduce some representations of $\slz$.  First, we define the
three-dimensional representation $\rho_3\colon\slz\rightarrow \GL_3(\Z)$ by
the formula
\begin{equation}\label{rho3def}
  \rho_3\begin{pmatrix}
         a &b \\
         c & d
       \end{pmatrix}=\begin{pmatrix}
                 \phantom{-}a^2 & \phantom{-}2ab & -b^2  \\
                 \phantom{-}ac & \phantom{-}ad+bc & -bd  \\
                 -c^2 & -2cd & \phantom{-}d^2
               \end{pmatrix} ;
\end{equation}
it is the restriction of a three-dimensional representation\footnote{There is
a unique irreducible, continuous representation  of $\SL_2(\R)$ of each
finite dimension, up to isomorphism.} of $\SL_2(\R)$ to $\slz$.  We define
$\rho_2\colon\slz\rightarrow \GL_2(\Z)$ by its action on the generators,
\[
  \rho_2(T) = \rho_2\begin{pmatrix}
         1 & 1 \\
         0 & 1
       \end{pmatrix}=\begin{pmatrix}
         -1 & 1 \\
         \phantom{-}0 & 1
       \end{pmatrix} \ \text{and} \
  \rho_2(S) = \rho_2\begin{pmatrix}
         0 & -1 \\
         1 & \phantom{-}0
       \end{pmatrix}=\begin{pmatrix}
         \phantom{-}0 & -1 \\
         -1 & \phantom{-}0
       \end{pmatrix};
\]
its image is a dihedral group of order 6 and its kernel is $\Gamma(2)$, so
$\rho_2$ is just a faithful representation of the dihedral group
$\SL_2(\Z)/\Gamma(2)$. Finally, define the function
$\vec{v}\colon\slz\rightarrow \Z^2$ by $\vec{v}(S)=(0, 0)$ and
$\vec{v}(T)=(1, -1)$ (thought of as row vectors), and then in general by the
cocycle formula
\begin{equation}\label{vecvcocyclelaw}
\vec{v}(\gamma \gamma')=\vec{v}(\gamma)\rho_2(\gamma')+\vec{v}(\gamma');
\end{equation}
to check that this cocycle is well defined, we can define it via
\eqref{vecvcocyclelaw} on the free group generated by $S$ and $T$, and then
check that it annihilates $S^2$ and $(ST)^3$, so that it factors though
$\PSL_2(\Z)$. Then
\[
\rho(\gamma)=\begin{pmatrix}
\rho_3(\gamma) & 0& 0\\
0& 1 & \vec{v}(\gamma) \\
0 & 0 & \rho_2(\gamma) \\
\end{pmatrix}
\]
defines a six-dimensional representation of $\slz$.

Many of our later calculations use a conjugate $\sigma$ of $\rho$ defined by
\begin{equation}\label{sigmaVrhoconjugate}
\sigma(\gamma) = g_{\rho\sigma}^{-1}\rho(\gamma)g_{\rho\sigma},
\end{equation}
where
\[
g_{\rho\sigma}=\left(
\begin{smallmatrix}
\phantom{-}1 &\,\, \phantom{-}0 &\,\, \phantom{-}1 &\,\, -1 &\,\, \phantom{-}0 &\,\, -1 \\
 \phantom{-}1 &\,\, \phantom{-}0 &\,\, \phantom{-}0 &\,\, -1 &\,\, \phantom{-}0 &\,\, \phantom{-}0 \\
 \phantom{-}1 &\,\, -1 &\,\, \phantom{-}0 &\,\, -1 &\,\, \phantom{-}1 &\,\, \phantom{-}0 \\
 \phantom{-}3 &\,\, -1 &\,\, -1 &\,\, \phantom{-}3 &\,\, -1 &\,\, -1 \\
 -2 &\,\, \phantom{-}0 &\,\, \phantom{-}2 &\,\, -2 &\,\, \phantom{-}0 &\,\, \phantom{-}2 \\
 \phantom{-}2 &\,\, -2 &\,\, \phantom{-}0 &\,\, \phantom{-}2 &\,\, -2 &\,\, \phantom{-}0 \\
\end{smallmatrix}
\right),
\]
which is characterized by its values
\[
\sigma(S)=
\left(\begin{smallmatrix}
0&0&0&1&0&0\\
0&0&1&0&0&0\\
0&1&0&0&0&0\\
1&0&0&0&0&0\\
0&0&0&0&0&1\\
0&0&0&0&1&0
\end{smallmatrix}\right)\quad\text{and}\quad
\sigma(T)=
\left(\begin{smallmatrix}
\phantom{-}0&\,\,\phantom{-}1&\,\,\phantom{-}0&\,\,\phantom{-}0&\,\,\phantom{-}0&\,\,\phantom{-}0\\
-1&\,\,\phantom{-}2&\,\,\phantom{-}0&\,\,\phantom{-}0&\,\,\phantom{-}0&\,\,\phantom{-}0\\
\phantom{-}2&\,\,\phantom{-}0&\,\,\phantom{-}0&\,\,\phantom{-}0&\,\,\phantom{-}0&\,\,-1\\
\phantom{-}0&\,\,\phantom{-}0&\,\,\phantom{-}0&\,\,\phantom{-}0&\,\,\phantom{-}1&\,\,\phantom{-}0\\
\phantom{-}0&\,\,\phantom{-}0&\,\,\phantom{-}0&\,\,-1&\,\,\phantom{-}2&\,\,\phantom{-}0\\
\phantom{-}0&\,\,\phantom{-}0&\,\,-1&\,\,\phantom{-}2&\,\,\phantom{-}0&\,\,\phantom{-}0
\end{smallmatrix}\right)
\]
on the generators \eqref{SandTdef}.  See the paragraph after the proof of
Proposition~\ref{prop:ideals2} for more discussion of the role of $\sigma$ in
this paper and its relationship with $\rho$.

The next result shows that the (integral) matrix entries of $\rho(\gamma)$
and $\sigma(\gamma)$ do not grow quickly in terms of those of $\gamma$.

\begin{lemma}\label{lem:polynomiallybounded}
There exist absolute constants $C,N>0$ such that each matrix entry of
$\rho\begin{pmatrix}
a & b \\
c & d
\end{pmatrix}$ and of $\sigma\begin{pmatrix}
a & b \\
c & d
\end{pmatrix}$ is bounded by $C(a^2+b^2+c^2+d^2)^N$ in absolute value.
\end{lemma}

The boundedness assertion in this lemma depends on the realization of the
abstract group $\slz$ in integer matrices.  In particular, it implies that
the restrictions of $\rho$ or $\sigma$ to free subgroups of $\slz$ satisfy
the same bound, a fact that is false for general representations of these
subgroups.  For example, the entries of $\sigma(T)^n$ grow only polynomially
in $n$, while exponential growth can occur for other representations of the
subgroup $\langle T \rangle$.

\begin{proof}
The assertions for these two conjugate representations are equivalent. The
representation $\rho_3$ satisfies this boundedness condition by virtue of its
explicit algebraic formula in \eqref{rho3def}, while $\rho_2$ has a finite
image.  Thus it suffices to verify that $\vec{v}(\gamma)$ with
$\gamma=\left(\begin{smallmatrix}
a & b \\
c & d
\end{smallmatrix}\right)$ satisfies the claimed bound.
Furthermore, formula \eqref{vecvcocyclelaw} with $\gamma\in \Gamma(2)$ and
$\gamma'$ one of the six coset representatives for $\Gamma(2)$ shows that it
suffices to prove this last bound for $\gamma\in \Gamma(2)$. Because
$\Gamma(2)$ is the kernel of $\rho_2$, formula \eqref{vecvcocyclelaw}
restricted to $\Gamma(2)$ shows that $\vec{v}$ is a homomorphism from
$\Gamma(2)$ to $\Z\oplus \Z$. Writing $\gamma$ as
$\gamma_1^{e_1}\gamma_2^{f_1}\gamma_1^{e_2}\cdots$ as in the statement of
Lemma~\ref{lem:columndomination}, we see that the entries of
$\vec{v}(\gamma)$ are bounded in absolute value by a constant multiple of
$|e_1|+|f_1|+|e_2|+\cdots$, and the result follows from part~(2) of
Lemma~\ref{lem:columndomination}.
\end{proof}

\subsection{The group algebra $R=\C[\PSL_2(\Z)]$}\label{sec:sub:groupring}

In constructing a solution to the identities
\eqref{eq:interpolation-identity}--\eqref{eq:2nddiffFt}, it is convenient to
use the slash operator action of $R=\C[\PSL_2(\Z)]$. Then
\eqref{eq:interpolation-identity}--\eqref{eq:2nddiffFt} state that
\begin{equation}\label{restated1.6to1.8}
F+i^{d/2} \widetilde{F}|^\tau_{d/2}S = e^{\pi i \tau |x|^2} \quad\text{and}\quad
F|^\tau_{d/2}(T-I)^2=\widetilde{F}|^\tau_{d/2}(T-I)^2=0,
\end{equation}
where $I$ denotes the identity element while $S$ and $T$ are defined by
\eqref{SandTdef}.

This subsection is devoted to studying some properties of $R$ that are used
later in the paper, in particular quotients of the translation action of
$\PSL_2(\Z)$ on $R$. Recall that $\PSL_2(\Z)$ is the free product of the
subgroups $\{I,S\}$ and $\{I,ST,STST\}$. If we write
\[
x=S \quad\text{and}\quad y=ST,
\]
then every element $\gamma$ of $\Gamma=\PSL_2(\Z)=\langle x,y \mid
x^2=y^3=1\rangle$ has a unique reduced expression as a product $w_1w_2\cdots
w_\ell$, where $\ell=\ell(\gamma)$ is the length of the product,  each  $w_j$
is either $x$, $y$, or $y^2$, and the only allowable consecutive pairs $w_j$
and $w_{j+1}$ are
\[
(w_j,w_{j+1})=(x,y),\ (x,y^2),\ (y,x), \text{~or~} (y^2,x).
\]
We extend the notion of length to $R=\C[\Gamma]$ by defining
$\ell(\sum_{\gamma\in\Gamma}c_\gamma \gamma)$ to be the maximum of all
$\ell(\gamma)$ for which $c_\gamma\neq 0$ (otherwise, $\ell(0)=-\infty$).

The order two element $x=S$ acts by left multiplication on $R$, which can be
diagonalized into $\pm 1$ eigenspaces using the decomposition
\begin{equation}\label{diagonalizex}
r=\frac{I+S}{2}r+\frac{I-S}{2}r
\end{equation}
for $r \in R$. In particular,
\begin{equation}\label{diagonalizexcor}
\{r\in R : (S\pm I)r=0\}=(S\mp I)R.
\end{equation}
Similarly we obtain
\begin{equation}\label{diagonalizeycor}
\{r\in R : (y-1)r=0\}=(y^2+y+1)R
\end{equation}
from the idempotent decomposition
\begin{equation*}\label{diagonalizey}
r=\sum_{j=0}^2 \frac{1+e^{2\pi  ij/3} y+e^{4\pi  ij/3}y^2}{3}r
\end{equation*}
into three distinct eigenspaces for left multiplication by $y$.

The equation $(T-I)v=w$ is a discretization of the derivative from
single-variable calculus, and it can be solved using a discretization of the
integral as follows.

\begin{lemma} \label{lemma:T-Iv=w}
Let $w = \sum_{\gamma\in\Gamma}c_\gamma \gamma$ with $c_\gamma \in \C$.  Then
there exists a solution $v \in R$ to
\[
(T-I)v=(xy-1)v=w
\]
if and only if $\sum_{n\in\Z}c_{(xy)^n\gamma}=0$ for all $\gamma$, in which case the unique
solution in $R$ is $v=\sum_{\gamma\in \Gamma}d_\gamma \gamma$ with
$d_{\gamma}=\sum_{n> 0}c_{(xy)^n\gamma}$.
\end{lemma}

\begin{proof}
Suppose $v=\sum_{\gamma\in \Gamma}d_\gamma \gamma$.  Then $(xy-1)v=w$ means
$c_\gamma = d_{(xy)^{-1} \gamma} - d_\gamma$ for all $\gamma$, and hence
$\sum_{n\in\Z}c_{(xy)^n\gamma}=0$ via telescoping.  Conversely, if
$\sum_{n\in\Z}c_{(xy)^n\gamma}=0$ for all $\gamma$, then only finitely many of the numbers
$d_{\gamma}:=\sum_{n> 0}c_{(xy)^n\gamma}$ are nonzero, and they satisfy
$d_{(xy)^{-1} \gamma} - d_\gamma = c_\gamma$, as desired.  To see that the
solution is unique, note that $(xy-1)v=0$ implies $d_\gamma =
d_{(xy)^{-1}\gamma}$ for all $\gamma$, which can happen only when $v=0$
because the coefficients have finite support.
\end{proof}

\begin{corollary} \label{cor:noT-1solcrit}
Suppose $w = \sum_\gamma c_\gamma \gamma$ and there exists $\gamma \in
\Gamma$ such that $c_{(xy)^n\gamma}\neq 0$ for exactly one $n \in \Z$.  Then
$w \notin (T-I)R$.
\end{corollary}

\begin{lemma} \label{lem:T-IvS+w}
The set of all $v \in R$ for which there exist $w \in R$ satisfying
\[
(T-I)v=(S+I)w
\]
is the right ideal $(y^2-y+1)(x+1)R$.  In other words,
\[
\{v \in R : (xy-1)v \in (x+1)R\} = (y^2-y+1)(x+1)R.
\]
\end{lemma}

\begin{proof}
Because $xy-1 = (x+1)y - (y+1)$, we see that $(xy-1)v \in (x+1)R$ if and only
if $(y+1)v \in (x+1)R$.  However, $y+1$ is invertible in the ring $R$, with
multiplicative inverse $(y^2-y+1)/2$, because $y^3=1$.  Thus, $(y+1)v \in
(x+1)R$ is equivalent to $v \in (y^2-y+1)(x+1)R$, as desired.
\end{proof}

\begin{lemma} \label{lem:T-IvS-w}
The set of all $v \in R$ for which there exist $w \in R$ satisfying
\[
(T-I)v=(S-I)w
\]
is the right ideal $(y^2+y+1)R$.
\end{lemma}

\begin{proof}
The identity $(xy-1)(y^2+y+1)=(x-1)(y^2+y+1)$ shows that any element of that
ideal provides a solution.  Conversely, if $(xy-1)v=(x-1)w$, then multiplying
by $x$ shows that
\[
(1-x)w=(y-x)v=(y-1)v-(x-1)v
\]
and hence $(y-1)v\in(x-1)R$. We will show that $(x-1)R \cap (y-1)R=\{0\}$,
from which it follows that $(y-1)v=0$ and thus $v\in (y^2+y+1)R$ by
\eqref{diagonalizeycor}, as needs to be shown.  If $(y-1)v\in (x-1)R$, we may
write $yv-v=xu-u$, for some $u=\sum_\gamma c_\gamma \gamma$ for which
$c_\gamma=0$ if $\gamma$'s reduced word begins with $x$ (such $\gamma$ can be
replaced by $-x\gamma$, which does not start with $x$, because
$(x-1)(-x)=x-1$). Multiplying both sides by $y^2+y+1$ on the left annihilates
$(y-1)v$ and yields $y^2xu+yxu=y^2u+y u -xu+ u$.  If $u$ is nonzero, then the
right side of this equation has length at most $\ell(u)+1$, while the left
side has terms having length $\ell(u)+2$ which cannot cancel each other.
Thus $u=0$, proving that $(x-I)R\cap (y-I)R=\{0\}$.
\end{proof}

\subsubsection{Some ideals of $R = \C[\PSL_2(\Z)]$}\label{sec:sub:sub:freeideals}

Define the right ideals
\begin{equation}\label{freeideals}
\aligned
{\mathcal I} & = (T-I)^2\cdot R + S(T-I)^2\cdot R, \\
{\mathcal I}_+ & = (S+I)\cdot R + (T-I)^2\cdot R, \\
{\mathcal I}_{-} & =  (S-I)\cdot R + (T-I)^2\cdot R, \\
\widetilde{\mathcal I}_+ & = (T-2I+T^{-1}-2S)\cdot R + (I-STS)\cdot R, \text{ and} \\
\widetilde{\mathcal I}_- & = (T-2I+T^{-1}+2S)\cdot R + (I+STS)\cdot R \\
\endaligned
\end{equation}
of $R$.  (We treat $\pm1$ and $\pm$ synonymously in subscripts.) As mentioned
at the end of Section~\ref{sec:sub:31}, $\mathcal{I}$ consists of the
elements of $R$ whose action on the generating function $F$ is determined by
the functional equations from Theorem~\ref{thm:FEimpliesIF}.  The ideals
${\mathcal I}_{\pm}$ play the same role when we decompose $F$ into
eigenfunctions for $S$ (see also part~(3) of Proposition~\ref{prop:ideals2}),
while the ideals $\widetilde{\mathcal I}_\pm$ are characterized by
Proposition~\ref{prop:ideals3} and will be used to describe the residues of
the kernels constructed in Theorem~\ref{thm: K}.  Note in particular that
$\mathcal{I} \subseteq \mathcal{I}_{\pm}$, because $S(T-I)^2 = (S\pm
I)(T-I)^2 \mp (T-I)^2$, and $\mathcal{I}_+ \cap \mathcal{I}_- = \mathcal{I}$,
because $Sr \equiv \pm r \pmod{\mathcal{I}}$ for $r \in \mathcal{I}_\pm$, so
$r \equiv -r \pmod{\mathcal{I}}$ for $r \in \mathcal{I}_+ \cap
\mathcal{I}_-$.

\begin{proposition}\label{prop:freeideals1}
The sums defining the ideals  $\mathcal I$, $\mathcal I_{+}$, and $\mathcal
I_{-}$ are all direct sums (i.e., these ideals are free modules of rank $2$
over $R$).
\end{proposition}

\begin{proof}
To deal with $\mathcal I$, suppose that  $(T-I)^2r=S(T-I)^2r'$ for some
$r,r'\in R$.   Multiply both sides by $S$ and combine to obtain the identity
$(I-\varepsilon S)(T-I)^2(r+\varepsilon r')=0$ for $\varepsilon=\pm1$. By
\eqref{diagonalizexcor}, this shows $(T-I)^2(r+\varepsilon r')\in
(S+\varepsilon I)R$. From this we see that the assertion for $\mathcal I$
follows from those for $\mathcal I_\pm$.

First consider $\mathcal I_+$ and $u,w\in R$ for which $(T-I)^2u=(S+I)w$. By
Lemma~\ref{lem:T-IvS+w}, $(T-I)u\in (y^2-y+1)(x+1)R$, where $x=S$ and $y=ST$
as above. We will prove that there is no $u$ for which $(T-I)u$ is a nonzero
member of this ideal, using Corollary~\ref{cor:noT-1solcrit}.  Suppose
$(T-I)u=(y^2-y+1)(x+1)r$ for some $r\in R$. Without loss of generality, we
may assume $r=r_0+yr_1+y^2r_2$, where $r_0\in \C$ and the reduced words
occurring in $r_1$ or $r_2$ all start with $x$ or are the identity. (The
nontrivial part of this assertion is that we can take $r_0 \in \C$, which
holds because we can use the identity $(x+1)x=x+1$ to incorporate any other
terms from $r_0$ to $\C$, $r_1$, or $r_2$.) Then $(y^2-y+1)(x+1)r$ equals
\begin{multline*}
y^2 x (r_0+yr_1+y^2 r_2)
-y x (r_0+yr_1+y^2 r_2)
+ x (r_0+yr_1+y^2 r_2)\\
\phantom{}+y^2  (r_0+yr_1+y^2 r_2)
-y  (r_0+yr_1+y^2 r_2)
+  (r_0+yr_1+y^2 r_2).
\end{multline*}
Assume now that $r_1$ and $r_2$ are not both zero and that $\ell(r_2)\ge
\ell(r_1)$, so that $\ell((T-I)u)\le \ell(r_2)+3$.  Choose $\gamma \in
\PSL_2(\Z)$ occurring in $r_2$ with $\ell(\gamma)=\ell(r_2)$.  The group
elements $(xy)^n y xy^2\gamma=(y^2x)^{-n} y xy^2 \gamma$ have length
$2n+2+\ell(r_2)$ (if $n\ge 1$) or $-2n+3+\ell(r_2)$ (if $n\le 0$), and so
cannot occur in $(T-I)u$ unless $n=0$, when they occur in exactly one place,
namely in $-yxy^2r_2$ (they do not occur in the three terms in the second
line because those have length at most $\ell(r_2)+1$).
Corollary~\ref{cor:noT-1solcrit} then proves there is no solution in this
case.  The same argument applies if $\ell(r_1)>\ell(r_2)$: if $\gamma$ occurs
in $r_1$ with $\ell(\gamma)=\ell(r_1)$, then the group elements $(xy)^n yxy
\gamma$ occur only in $-yxyr_1$ and with $n=0$.  Thus $r_1=r_2=0$, and unless
$r_0=0$ we may renormalize to set $r_0=1$ and obtain
$(T-I)u=y^2x-yx+x+y^2-y+1$.  Now $yx$ is the only term in this expression of
the form $(xy)^n yx$, and we again obtain a contradiction from
Corollary~\ref{cor:noT-1solcrit}.

The  case  of $\mathcal I_{-}$ is slightly simpler: here $(T-I)^2u=(S-I)w$
implies $(T-I)u\in (y^2+y+1)R$ by Lemma~\ref{lem:T-IvS-w}. Suppose
$(T-I)u=(y^2+y+1)r$ with $r \in R$.  This time there is no loss in generality
in assuming $r=r_0+x r_1$, where $r_0\in \C$ and $r_1$ does not begin with
$x$, because $y^2+y+1$ annihilates both $y-1$ and $y^2-1$. Now
\[
(y^2+y+1)r =
r_0y^2+r_0y+r_0+y^2xr_1+yxr_1+xr_1.
\]
The terms occurring in $(xy)^n yxr_1$ occur in this expression only if $n=0$,
and thus Corollary~\ref{cor:noT-1solcrit} applies unless possibly $r_1=0$. If
$r_1=0$, then the term $(xy)^ny$ occurs only for $n=0$, and so we conclude
that $(T-I)u=0$, as desired.
\end{proof}

\begin{proposition}\label{prop:ideals2} The ideals $\mathcal I$, $\mathcal
I_+$, and $\mathcal I_-$ of $R$ have the following properties:
\begin{enumerate}
\item $\dim_\C R/\mathcal I =6$ and
\begin{equation}\label{Midef}
M_1 = I,\  M_2=T,\ M_3=TS,\ M_4=S,\ M_5=ST,\ M_6=STS
\end{equation}
are a basis of $R/\mathcal I$,

\item $\dim_\C R/\mathcal I_\pm =3$  and $I,T,TS$ are a basis of
    $R/\mathcal I_{\pm}$, and

\item $\mathcal I_\varepsilon=\{r\in R : (S-\varepsilon I)r\in \mathcal
    I\}$ for $\varepsilon = \pm 1$.
\end{enumerate}
\end{proposition}

\begin{proof}
Consider the six-dimensional representation $\sigma$ of $\PSL_2(\Z)$ defined
in \eqref{sigmaVrhoconjugate}. By linearity $\sigma$ further extends to an
algebra homomorphism from $R=\C[\PSL_2(\Z)]$ to $\C^{6 \times 6}$, whose
first row vanishes on $\mathcal I$ because the first rows of both
$\sigma(S(T-I)^2)$ and $\sigma((T-I)^2)$ vanish.

We begin by proving that $M_1,\dots,M_6$ span $R/\mathcal I$. If $\vec M$
denotes the column vector $(M_1,\dots,M_6)\in R^6$, then the calculations
\begin{equation}\label{explicitsigmaVcalc}
\aligned
\vec{M}\cdot S&=\left(\begin{smallmatrix}S\\TS\\T\\I\\STS\\ST\end{smallmatrix}\right)
=
\sigma(S)\vec{M} \quad\text{and}
\\
\vec{M}\cdot T&=
\left(\begin{smallmatrix}T\\T^2\\TST\\ST\\ST^2\\STST\end{smallmatrix}\right)
=\left(\begin{smallmatrix}T\\-I+2T+(T-I)^2\\2-STS+S(T-I)^2T^{-1}S\\ST\\2ST-S+S(T-I)^2\\
2S-TS+(T-I)^2T^{-1}S\end{smallmatrix}\right)\\
&=\sigma(T)\vec{M}+(T-I)^2\cdot \left(\begin{smallmatrix}0\\I\\0\\0\\0
\\T^{-1}S\end{smallmatrix}\right)+S(T-I)^2\cdot\left(\begin{smallmatrix}0\\0\\T^{-1}S\\0\\ I \\0\end{smallmatrix}\right)
\endaligned
\end{equation}
show that
\begin{equation} \label{eq:sigmaVmeaning}
\vec{M}\cdot \gamma\in \sigma(\gamma)\vec{M}+{\mathcal I}^6
\end{equation}
for every $\gamma\in \PSL_2(\Z)$, because we inductively obtain
\[\vec{M}\cdot \gamma\gamma' \in \big(\sigma(\gamma)\vec{M}+{\mathcal I}^6\big) \cdot \gamma' \in \sigma(\gamma) \sigma(\gamma') \vec{M}+{\mathcal I}^6.\]
(In fact, \eqref{eq:sigmaVmeaning} is the motivation for our definition of
$\sigma$.) It follows that $\gamma$, which is the first entry of
$\vec{M}\cdot \gamma$, is a linear combination of $M_1,\dots,M_6$ plus an
element of $\mathcal I$.

To show that $M_1,\dots,M_6$ are linearly independent in $R/\mathcal I$, we
begin by checking that $\sigma(M_i)$ has first row entries that are all $0$
except for a $1$ in the $i$-th position.  Because the first row of
$\sigma(r)$ vanishes for all $r \in \mathcal I$, no nontrivial linear
combination of $M_1,\dots,M_6$ can lie in $\mathcal I$, which completes the
proof of part~(1).  (Note also that examining the first row of $\sigma$ gives
a convenient algorithm for reducing elements of $\PSL_2(\Z)$ modulo
$\mathcal{I}$.)

Next we prove part~(2).  First, we observe that left multiplication by $S$
preserves each of $\mathcal{I}$ and $\mathcal{I}_\pm$; for $\mathcal{I}$ this
preservation follows immediately from the definition of $\mathcal{I}$, and
for $\mathcal{I}_\pm$ it follows from the calculations in the paragraph
containing \eqref{freeideals}. Because $I$, $T$, $TS$, $S$, $ST$, and $STS$
span $R/\mathcal{I}$, and $S$ acts on the left by $\mp1$ modulo
$\mathcal{I}_\pm$, we see that $I$, $T$, and $TS$ span $R/\mathcal{I}_\pm$.
Now let $\pi \colon R \to (R/\mathcal{I}_+) \oplus (R/\mathcal{I}_-)$ be the
direct sum of the projections modulo these ideals.  Because $\mathcal{I}_+
\cap \mathcal{I}_- = \mathcal{I}$, the kernel of $\pi$ is $\mathcal{I}$, and
thus $R/\mathcal{I}$ maps injectively to $(R/\mathcal{I}_+) \oplus
(R/\mathcal{I}_-)$.  We have seen that $R/\mathcal{I}_\pm$ are both at most
three-dimensional, while $R/\mathcal{I}$ is six-dimensional by part~(1).
Thus, $\dim_\C R/\mathcal{I}_\pm = 3$, and $I$, $T$, and $TS$ form a basis.

All that remains is to prove part~(3). Since $(S-\varepsilon I)(S+\varepsilon
I)=0$ and $(S-\varepsilon I)(T-I)^2r=S(T-I)^2r + (T-I)^2 (-\varepsilon r)$
lies in $\mathcal I$ for all $r\in R$, it is clear that left multiplication
by $S-\varepsilon I$ maps $\mathcal I_\varepsilon$ to $\mathcal I$.  To show
the reverse inclusion, suppose that
\[
(S-\varepsilon I)r=(T-I)^2 r_1+S(T-I)^2 r_2 \in \mathcal I
\]
for some $r,r_1,r_2\in R$, and hence after multiplying both sides by $S$,
\[
-\varepsilon(S-\varepsilon I)r=(T-I)^2 r_2+S(T-I)^2 r_1.
\]
Combining, we see that  $2(S-\varepsilon I)r= (S-\varepsilon
I)(T-I)^2(r_2-\varepsilon r_1)$ and hence $(S-\varepsilon I)\big(2r -
(T-I)^2(r_2-\varepsilon r_1)\big)=0$. By \eqref{diagonalizexcor},
$2r-(T-I)^2(r_2-\varepsilon r_1) \in (S+\varepsilon I)R$, which implies $r\in
\mathcal I_\varepsilon$.
\end{proof}

Note that \eqref{eq:sigmaVmeaning} gives us a natural interpretation of
$\sigma$ in terms of the right action of $\PSL_2(\Z)$ on the six-dimensional
vector space $R/\mathcal I$.  From this perspective, the conjugacy between
$\sigma$ and $\rho$ in Section~\ref{sec:sub:6dimlrepn} means $R/\mathcal{I}$
must decompose into the direct sum of two three-dimensional right
$R$-modules. These submodules are the subspaces $\mathcal{I}_\pm/\mathcal{I}$
of $R/\mathcal{I}$, which are spanned by $I \pm S$, $(I \pm S)T$, and $(I \pm
S)TS$.  Here, $\mathcal{I}_-/\mathcal{I}$ corresponds to $\rho_3$, while
$\mathcal{I}_+/\mathcal{I}$ has a two-dimensional submodule corresponding to
$\rho_2$, namely $(\mathcal{I}_+ \cap \mathcal{I}_\text{aug})/\mathcal{I}$,
where $\mathcal{I}_\text{aug}$ is the augmentation ideal of $R$.  (Recall
that the augmentation ideal is the two-sided ideal generated by $\gamma-I$
with $\gamma \in \PSL_2(\Z)$.  It follows immediately from the definitions of
$\mathcal{I}_\pm$ that $\mathcal{I}_- \subseteq \mathcal{I}_\text{aug}$,
while $\mathcal{I}_+ \not\subseteq \mathcal{I}_\text{aug}$.) Choosing a basis
compatible with the above structure led us to the representation $\rho$ and
the intertwining operator $g_{\rho\sigma}$ introduced in
Section~\ref{sec:sub:6dimlrepn}. Note that the reason why
$\mathcal{I}_+/\mathcal{I}$ has a two-dimensional submodule rather than a
two-dimensional quotient is that $\sigma$ is a left representation, while we
have a right action of $R$ on $R/\mathcal I$. This discrepancy is
unfortunate, but avoiding it would require breaking convention by either
using right representations or replacing the slash operator with a left
action.

Since $\mathcal I=(T-I)^2 \cdot R+ S(T-I)^2 \cdot R$ is free as a right
$R$-module, for each $r\in R$ there exist unique column vectors
$\vec{N}_1(r),\vec{N}_2(r)\in R^6$ such that
\begin{equation}\label{sigmaandN}
\vec{M}\cdot r = \sigma(r)\vec{M}+(T-I)^2 \cdot\vec{N}_1(r) + S(T-I)^2 \cdot\vec{N}_2(r).
\end{equation}
It follows that
\begin{equation}\label{cocyclelaw}
\vec{N}_i(r_1r_2) = \sigma(r_1)\vec{N}_i(r_2)+\vec{N}_i(r_1)\cdot r_2
\end{equation}
for $i=1,2$ and $r_1,r_2\in R$, with the cocycle relation \eqref{cocyclelaw}
describing the composition law for multiple applications of
\eqref{sigmaandN}.  Repeated applications of \eqref{cocyclelaw} result in the
more general formula
\begin{equation}\label{cocyclelawwithmore}
\begin{split}
\vec{N}_i(r_1r_2\cdots r_n) &= \sigma(r_1\cdots r_{n-1})\vec{N}_i(r_n)\\
&\quad \phantom{}
+ \sigma(r_1\cdots r_{n-2})\vec{N}_i(r_{n-1})\cdot r_n
\\
&\quad \phantom {} + \sigma(r_1\cdots r_{n-3})\vec{N}_i(r_{n-2})\cdot r_{n-1}r_n \\
&\quad \phantom {} +
\cdots +
\vec{N}_i(r_{1})\cdot r_2\cdots r_n
\end{split}
\end{equation}
for more than two factors.

\begin{lemma}\label{lem:cocyclebounds}
There exist positive constants $C$ and $N$ such that for all
$\gamma\in\PSL_2(\Z)$, the entries of $\vec{N}_i(\gamma)$ have the form
$\sum_{\delta\in \PSL_2(\Z)}n_\delta \delta$, with
$\sum_{\delta}|n_\delta|\le C\|\gamma\|_{\textup{Frob}}^N$ and
$\|\delta\|_{\textup{Frob}}\le C\|\gamma\|_{\textup{Frob}}$ whenever
$n_\delta\neq 0$.
\end{lemma}

\begin{proof}
Formula \eqref{cocyclelaw} reduces the claim to $\gamma\in
\overline{\Gamma(2)}$, as can be seen by taking $r_2\in \overline{\Gamma(2)}$
and $r_1$ one of the six coset representatives of $\overline{\Gamma(2)}$.
Lemma~\ref{lem:polynomiallybounded} shows the existence of constants $C_0,N_0>0$
such that the matrix entries  of $\sigma(\gamma)$ are bounded by
$C_0\|\gamma\|_{\textup{Frob}}^{N_0}$ in absolute value.  Factor
$\gamma=\gamma_1^{e_1}\gamma_2^{f_1}\gamma_1^{e_2}\cdots$, with
$\gamma_1=T^2$ and $\gamma_2=ST^2S$ as in Lemma~\ref{lem:columndomination},
and refine the factorization to $\gamma = r_1 \dots r_n$ with $r_i \in
\{\gamma_1^{\pm 1}, \gamma_2^{\pm 1}\}$ and $n = |e_1|+|f_1|+|e_2|+\cdots \le
\|\gamma\|_{\textup{Frob}}^2$ by part~(2) of
Lemma~\ref{lem:columndomination}.  By part~(3) of
Lemma~\ref{lem:columndomination}, $\|r_1r_2\cdots r_i\|_{\textup{Frob}}$ is
an increasing function of $i$, while $\|r_ir_{i+1}\cdots
r_n\|_{\textup{Frob}}$ is decreasing. Then \eqref{cocyclelawwithmore}
expresses $\vec{N}_i(\gamma)$ as a combination of $n$ terms, each of which
satisfies the asserted bounds.
\end{proof}

Like all group rings, $R$ is equipped with the anti-involution $\iota$ that
sends $\sum_\gamma c_\gamma \gamma$ to $\sum_\gamma c_\gamma \gamma^{-1}$.

\begin{proposition}\label{prop:ideals3}
Define linear functionals $\phi_\pm\colon R/{\mathcal I}_\pm\to \C$ on the
basis vectors from Proposition~\ref{prop:ideals2} by setting
\begin{equation}\label{phiepsdef}
\phi_\pm(I)=0, \quad \phi_\pm(T)=1, \quad \text{and}\quad \phi_\pm(TS)=0.
\end{equation}
Then
\[
\widetilde{\mathcal I}_\pm = \{r\in R : \phi_\pm(r'\cdot \iota(r))=0 \textup{ for all $r'\in R$}\},
\]
and $I,S,T$ are a basis of $R/\widetilde{\mathcal I}_\pm$.
\end{proposition}

The motivation for the maps $\phi_\pm$ is that they describe the residues in
part~(3) of Theorem~\ref{thm: K}, and therefore they play a key role in the
contour shifts in Section~\ref{sec:proofs}.

\begin{proof}
Let $\varepsilon = \pm 1$, let $r_1=T-2I+T^{-1}-2\varepsilon S$ and
$r_2=I-\varepsilon STS =I-\varepsilon T^{-1}ST^{-1}$ be the generators of
$\widetilde{\mathcal{I}}_\varepsilon$ from its definition \eqref{freeideals},
and let
\[
J_\varepsilon = \{r\in R : \phi_\varepsilon(r'\cdot \iota(r))=0 \textup{ for all $r'\in R$}\},
\]
which we will show equals  $\widetilde{\mathcal{I}}_\varepsilon$.

To prove that $\widetilde{\mathcal I}_\varepsilon \subseteq J_\varepsilon$,
note first that $J_\varepsilon$ is a right ideal. It will therefore suffice
to show that $\phi_\varepsilon$ annihilates $\iota(r_i)$, $T \iota(r_i)$, and
$TS \iota(r_i)$ for $i=1,2$, because $I$, $T$, and $TS$ span $R/\mathcal
I_\varepsilon$ by part~(2) of Proposition~\ref{prop:ideals2}. We check in
succession that
\begin{align*}
\phi_\varepsilon(S) &= -\varepsilon \phi_\varepsilon(I) = 0,\\
\phi_\varepsilon(ST) &= -\varepsilon \phi_\varepsilon(T) = -\varepsilon,\\
\phi_\varepsilon\big(T^2\big) &= \phi_\varepsilon(2T-I) = 2\phi_\varepsilon(T) = 2,\\
\phi_\varepsilon\big(T^{-1}\big) &= \phi_\varepsilon(2I-T) = -\phi_\varepsilon(T) = -1,\\
\phi_\varepsilon\big(T^{-1}S\big) &= \phi_\varepsilon((2I-T)S) = 0,\\
\phi_\varepsilon(TST) &= \phi_\varepsilon\big(ST^{-1}S\big) = -\varepsilon\phi_\varepsilon\big(T^{-1}S\big) = 0,\\
\phi_\varepsilon\big(ST^{-1}\big) &= -\varepsilon \phi_\varepsilon\big(T^{-1}\big) = \varepsilon,\\
\phi_\varepsilon\big(T^{-1}ST^{-1}\big) &= \phi_\varepsilon(STS) = -\varepsilon\phi_\varepsilon(TS) = 0,\\
\phi_\varepsilon\big(TST^{-1}\big) &= \phi_\varepsilon\big(\big(2I-T^{-1}\big)ST^{-1}\big) = 2\varepsilon, \quad\text{and}\\
\phi_\varepsilon\big(T^2ST\big) &= \phi_\varepsilon((2T-I)ST) = \varepsilon.
\end{align*}
It is then straightforward to verify by direct calculation than
$\phi_\varepsilon$ evaluates to zero on $\iota(r_i)$, $T\iota(r_i)$, and $TS
\iota(r_i)$, because $\iota(r_1)=r_1$ and $\iota(r_2)=I-\e ST^{-1}S=I-\e
TST$. Thus, $\widetilde{\mathcal I}_\varepsilon \subseteq J_\varepsilon$.

To complete the proof, we will show that $\dim
(R/\widetilde{\mathcal{I}}_\varepsilon) \le \dim(R/J_\varepsilon)$. We first
note that the map $\psi\colon R \to \Hom_\C(R/{\mathcal I}_\e, \C)$ taking
$r$ to $r' \mapsto \phi_\e(r' \cdot\iota(r))$ has kernel $J_\e$ by definition
and is surjective, because the linear functionals $\psi(I)$, $\psi(S)$, and
$\psi(T)$ satisfy
\[
\aligned
\psi(I)(I) &= 0, \quad& \psi(I)(T) &= 1, \quad& \psi(I)(TS) &= 0,\\
\psi(S)(I) &= 0, \quad& \psi(S)(T) &= 0, \quad& \psi(S)(TS) &= 1,\\
\psi(T)(I) &= -1, \quad& \psi(T)(T) &= 0, \quad& \psi(T)(TS) &= 2\e\\
\endaligned
\]
and hence span the three-dimensional space of linear functionals. Therefore
$R/J_\e$ is a three-dimensional vector space, and it suffices to show that
$I$, $S$, and $T$ span $R/\widetilde{\mathcal I}_\e$.

To prove that $I$, $S$, and $T$ span $R/\widetilde{\mathcal I}_\e$, we begin
by reducing $ST$, $ST^{-1}$, $TS$, $T^{-1}$, and $T^2$ to linear combinations
of $I$, $S$, and $T$ modulo $\widetilde{\mathcal I}_\e$ via
\begin{align*}
ST &= \e S - \e r_2S\\
&\equiv \e S \pmod{\widetilde{\mathcal I}_\e},\\
ST^{-1} &= \e S + r_2 ST^{-1}\\
& \equiv \e S \pmod{\widetilde{\mathcal I}_\e},\\
TS &= 2\e I + 2S - T^{-1}S + r_1 S\\
& \equiv 2\e I + 2S - T^{-1}S \pmod{\widetilde{\mathcal
I}_\e}\\
& = 2\e I + 2S - \e T + \e r_2 T\\
& \equiv 2\e I + 2S - \e T \pmod{\widetilde{\mathcal
I}_\e},\\
T^{-1} &= \e TS - \e r_2 TS\\
&\equiv \e TS \equiv 2I + 2\e S - T \pmod{\widetilde{\mathcal I}_\e}, \quad\text{and}\\
T^2 &= 2T - I + 2\e ST + r_1T\\
&\equiv 2T - I + 2\e ST \pmod{\widetilde{\mathcal
I}_\e}\\
&\equiv 2T - I + 2S \pmod{\widetilde{\mathcal I}_\e}.
\end{align*}
Using these relations, we can reduce any $\gamma \in \overline{\Gamma}$ (and
therefore any $r \in R$) to a linear combination of $I$, $S$, and $T$ modulo $\widetilde{\mathcal I}_\e$, by
starting from the left of the reduced word representing $\gamma$.
\end{proof}

\section{Solutions of functional equations and modular form kernels}
\label{sec:analysis}

\subsection{The action of $R$ on holomorphic functions and integral kernels}
\label{sec:sub:RactingonP}

We begin by recalling a commonly used class of functions on the upper
half-plane $\Hyp$, which appears in the literature dating at least as far
back as 1974 in work of Knopp \cite{Knopp}: those satisfying a bound of the
form
\begin{equation}\label{classPbound}
|F(\tau)| \le\alpha\big(\Im(\tau)^{-\beta}+|\tau|^\gamma\big)
\end{equation}
for some $\alpha,\beta,\gamma \ge 0$. Such a bound has already appeared in
part~(3) of Theorem~\ref{thm:FEimpliesIF}, and it is satisfied by any power
series in $e^{\pi i \tau}$ whose coefficients grow more slowly than some
polynomial (e.g., the classical modular forms $E_k$ and theta functions from
Section~\ref{sec:sub:modularforms}).  Conversely, the boundedness condition
in the definition of modular form is automatic for functions satisfying
\eqref{classPbound}.

Recall also that a function $F$ has \emph{moderate growth} on the symmetric
space $\Hyp$ if
\begin{equation}\label{HCmoderategrowth}
|F(g\cdot i)| \le C \| g\|^N
\end{equation}
for some $C,N \ge 0$, where $g\cdot i=\frac{ai+b}{ci+d}$ denotes the action
of $g=\left(\begin{smallmatrix}
a & b \\
c & d
\end{smallmatrix}\right)\in \SL_2(\R)$ on $i\in \Hyp$ (with $i^2=-1$)
and $\|\cdot\|$ is some fixed matrix norm on $\SL_2(\R)$, such as the
Frobenius norm. This notion is independent of the choice of matrix norm.

In fact, moderate growth is equivalent to \eqref{classPbound}. First, note
that the notion of moderate growth does not depend on the choice of $g$: if
$g \cdot i = g' \cdot i$ with $g,g' \in \SL_2(\R)$, then $g^{-1} g'  \in
\SO_2(\R)$ and hence $\|g\|_{\textup{Frob}} = \|g'\|_{\textup{Frob}}$. Now to
see the equivalence between the two bounds, write $\tau=x+iy$ and let
$g_\tau=y^{-1/2}\left(
\begin{smallmatrix}
y & x \\
0 & 1
\end{smallmatrix}
\right)$.  Then $g_\tau\cdot i =\tau$ and
\begin{align*}
\|g_\tau\|^2_{\textup{Frob}}&= y^{-1} +y^{-1}(x^2+y^2)\\
& \le y^{-1} + \big(y^{-2} + (x^2+y^2)^2\big)/2\\
&= O\big(y^{-2} + (x^2+y^2)^2\big).
\end{align*}
Thus, moderate growth implies \eqref{classPbound}.  For the other direction,
we must bound $y^{-1}$ and $x^2+y^2$ by polynomials in
$\|g_\tau\|^2_{\textup{Frob}}$.  To do so, we note that
\[
y^{-1}\le\frac{1+x^2+y^2}{y}=\|g_\tau\|^2_{\textup{Frob}} \quad \text{and} \quad
y^2\le
\left(\frac{1+x^2+y^2}{y}\right)^2 =\|g_\tau\|^4_{\textup{Frob}},
\]
and hence
\[
x^2+y^2=y\frac{x^2+y^2}{y}\le
\frac{1}{2}\left(y^2+\frac{(x^2+y^2)^2}{y^2}\right) \le \|g_\tau\|^4_{\textup{Frob}},
\]
as desired.

When $\mathcal{S}$ is a subset of $\Hyp$,
we accordingly use the term \emph{moderate growth on $\mathcal S$} to
describe functions satisfying the bound \eqref{classPbound} for $\tau\in
\mathcal S$. Following Knopp's notation, let
\begin{equation}\label{classP}
\mathcal P = \left\{
F\colon\Hyp\to \C : \text{$F$ is holomorphic and satisfies \eqref{classPbound} for all~}\tau\in\Hyp
\right\}
\end{equation}
denote the space of holomorphic functions having moderate growth on the full
upper half-plane.  In terms of the unit disk model of the hyperbolic plane,
$\mathcal P$ corresponds to holomorphic functions on $\{z \in \C : |z|<1\}$
that are bounded in absolute value by $C(1-|z|)^{-N}$ for some constants
$C,N\ge 0$.

\begin{lemma} \label{lem:HCmoderategrowthPslashed}
Let $F$ be a holomorphic function on $\Hyp$, and $\mathcal S \subseteq \Hyp$.
If
\[
|F(g\cdot i)| \le C   \| g\|_{\textup{Frob}}^N
\]
for all $g \in \SL_2(\R)$ such that $g \cdot i \in \mathcal S$, then
\[
|(F|_k\gamma)(g\cdot i)| \le C \|\gamma\|^{N+|k|}_{\textup{Frob}} \| g\|^{N+2|k|}_{\textup{Frob}}
\]
for all $\gamma,g \in \SL_2(\R)$ such that $g \cdot i \in \gamma^{-1}
\mathcal S$.
\end{lemma}

In particular, $\mathcal P$ is preserved by the slash operation
\eqref{slashnotation}.  It is consequently a representation space for
$\SL_2(\Z)$, and for $\PSL_2(\Z)$ and thus $R=\C[\PSL_2(\Z)]$ when $k$ is
even (so that $\pm\gamma$ act the same for $\gamma \in \SL_2(\Z)$; see
\cite[Section~1.1.6]{CohenStromberg}).

\begin{proof}
The factor of automorphy $j(g,z)=cz+d$ for a matrix
$g=\left(\begin{smallmatrix}
a & b \\
c & d
\end{smallmatrix}\right)\in \SL_2(\R)$ and point $z\in\Hyp$
satisfies the bounds $\|g\|_{\textup{Frob}}^{-1} \le |j(g,i)|\le
\|g\|_{\textup{Frob}}$. The upper bound is trivial, while the lower bound
follows from $\|g\|_{\textup{Frob}}^2 > a^2+b^2 \ge (c^2+d^2)^{-1}$, where
the last inequality $(a^2+b^2)(c^2+d^2)\ge 1= (ad-bc)^2$ is itself a
consequence of the Cauchy-Schwarz inequality. Note that the same inequality
also shows that $\|g\|_{\textup{Frob}}\ge\sqrt{2}$. Let $g_\tau$ be some
element of $\SL_2(\R)$ for which $g_\tau\cdot i=\tau$.  The identity
$j(g_1g_2,z)=j(g_1,g_2z)j(g_2,z)$ shows that $j(\gamma,\tau)=j(\gamma
g_\tau,i)/j(g_\tau,i)$, and thus
\begin{equation}\label{jvsnorms}
|j(\gamma,\tau)|,|j(\gamma,\tau)|^{-1} \le
\|g_\tau\|_{\textup{Frob}} \|\gamma g_\tau\|_{\textup{Frob}}\le
\|\gamma\|_{\textup{Frob}} \|g_\tau\|_{\textup{Frob}}^2
\end{equation}
for all $\gamma\in \SL_2(\R)$ and $\tau\in\Hyp$.  Now the desired bound
follows from the definition
\[
(F|_k\gamma)(z) = j(\gamma,z)^{-k} F(\gamma \cdot z)
\]
of the slash operation.
\end{proof}

We now algebraically reformulate the system \eqref{restated1.6to1.8} in terms
of this action on $\mathcal P$.  For simplicity we assume $d$ is a multiple
of $4$, so that $\PSL_2(\Z)$ can act with weight $d/2$.  Then if we let
$\widetilde F=i^{-d/2}(e^{\pi i \tau |x|^2 }-F)|^\tau_{d/2}S$, system
\eqref{restated1.6to1.8} becomes
\begin{equation}\label{rerestated1.6to1.8}
    F|_{d/2}^\tau(T-I)^2   = 0 \quad  \text{and} \quad  F|_{d/2}^\tau S(T-I)^2  = e^{\pi i \tau |x|^2}|_{d/2}^\tau S(T-I)^2.
\end{equation}
Since $\mathcal I=(T-I)^2\cdot R+ S(T-I)^2\cdot R$, these equations govern
the slash operator action of $\mathcal I$ on $F$ in the $\tau$ variable.
After we solve for $F$ in this section,
bounds on $F$ (e.g., to show membership in $\mathcal P$) will be shown in
Section~\ref{sec:proofs}. Let
\begin{equation}\label{domainD} \mathcal{D}=\left\{ z\in\Hyp
\;:\;\Re(z)\in(-1,1),\;\left|z-\frac12\right|>\frac12,\;\left|z+\frac12\right|>\frac12\right\}
\end{equation}
and
\begin{equation}\label{domainF}
\mathcal{F}=\{ z\in\Hyp \;:\;\Re(z)\in(0,1),\;|z|>1,\;|z-1|>1\}
\end{equation}
be the fundamental domains for $\Gamma(2)$ and $\slz$, respectively, which
are shown in Figure~\ref{fig:domainD} and satisfy
\begin{equation}\label{DclosureFclosure}
\overline{\mathcal{D}} = \overline{\mathcal{F} \cup T^{-1} \mathcal{F}
\cup STS \mathcal{F}
\cup S \mathcal{F}
\cup ST^{-1} \mathcal{F}
\cup TS \mathcal{F}}.
\end{equation}
The following proposition shows that functions with symmetry properties
generalizing \eqref{rerestated1.6to1.8} are determined by their behavior on
the closure $\overline{\mathcal D}$ of $\mathcal D$ in $\Hyp$.

\begin{figure}
\begin{tikzpicture}[scale=4]
\fill[black!30!white] (-1,0) rectangle (1,2);
\fill[gray!40!white] (1,0) arc (0:180:1);
\fill[gray!40!white] (1,0)--(1,1) arc (90:180:1);
\fill[gray!40!white] (0,0) arc (0:90:1)--(-1,0);
\fill[gray!40!white] (-1,0) rectangle (0,2);
\fill[gray!40!white] (-1,0) rectangle (0,2);
\fill[white]  (1,0) arc (0:180:0.5);
\fill[white]  (0,0) arc (0:180:0.5);
\draw[gray] (0,0)--(0,2);
\draw[gray] (-1,0)--(-1,2);
\draw[gray] (1,0)--(1,2);
\draw[gray] (0.5,{sqrt(3/4)}) arc (60:120:1);
\draw[gray] (1,1) arc (90:120:1);
\draw[gray] (-1,1) arc (90:60:1);
\draw[gray] (0,0) arc (0:180:0.5);
\draw[gray] (1,0) arc (0:180:0.5);
\draw[gray] (0.5,0.5) -- (0.5,{sqrt(3/4)});
\draw[gray] (-0.5,0.5) -- (-0.5,{sqrt(3/4)});
\draw[gray] (1.5,0.5) arc (90:180:0.5);
\draw (1.25,1) node[right] {\scriptsize{$T^2\mathcal{D}$}};
\draw[gray] (-1.5,0.5) arc (90:0:0.5);
\draw (-1.25,1) node[left] {\scriptsize{$T^{-2}\mathcal{D}$}};
\draw[gray] (1,0) arc (0:180:0.25);
\draw[gray] (0.5,0) arc (0:180:{1/12});
\draw[gray] ({1/3},0) arc (0:180:{1/6});
\draw (0.5,0.35) node {\scriptsize{$ST^{-2}S\mathcal{D}$}};
\draw[gray] (-0.5,0) arc (0:180:0.25);
\draw[gray] ({-1/3},0) arc (0:180:{1/12});
\draw[gray] (0,0) arc (0:180:{1/6});
\draw (-0.5,0.35) node {\scriptsize{$ST^2S\mathcal{D}$}};
\filldraw[black] (0,0) circle (0.01125);
\draw  (0,0) node[below=2pt] {\scriptsize{$0$}};
\filldraw[black] (-1,0) circle (0.01125);
\draw  (-1,0)node[below=2pt] {\scriptsize{$-1$}};
\filldraw[black] (1,0) circle (0.01125);
\draw (1,0) node[below=2pt] {\scriptsize{$1$}} ;
\filldraw[black] ({1/3},{4/3}) circle (0.01125); \draw  ({1/3},{4/3})
node[above=1pt] {\scriptsize{$\tau$}};
\filldraw[black] ({-2/3},{4/3}) circle (0.01125); \draw  ({-2/3},{4/3})
node[above=1pt] {\scriptsize{$T^{-1}\tau=\tau-1$}};
\filldraw[black] ({-3/17},{12/17}) circle (0.01125); \draw  ({-3/17},{12/17})
node[above=1pt] {\scriptsize{$S\tau=-\frac{1}{\tau}$}};
\filldraw[black] ({14/17},{12/17}) circle (0.01125); \draw  (0.76,{12/17})
node[above=1pt] {\scriptsize{$TS\tau=\frac{\tau-1}{\tau}$}};
\filldraw[black] (0.3,0.6) circle (0.01125); \draw  (0.26,0.6)
node[above=1pt] {\scriptsize{$ST^{-1}\tau=\frac{1}{1-\tau}$}};
\filldraw[black] (-0.7,0.6) circle (0.01125); \draw  (-0.73,0.6)
node[above=1pt] {\scriptsize{$STS\tau=\frac{\tau}{1-\tau}$}};
\end{tikzpicture}
\caption{The fundamental domain  $\mathcal{D}$ for $\Gamma(2)$ defined in
\eqref{domainD} is shaded, with the fundamental domain $\mathcal{F}$ for $\slz$ defined in
\eqref{domainF} shaded darker (note that $\mathcal{F} \subseteq \mathcal{D}$).
The six marked points in the interior of $\mathcal{D}$ are the images of a point $\tau \in \mathcal{F}$.
\label{fig:domainD}}
\end{figure}

\begin{proposition}\label{prop:propagationandextension}
Let $k$ be an even integer, and let $h_1,h_2\colon \Hyp\rightarrow\C$ be
continuous functions.  Then the following hold:
\begin{enumerate}
\item (Analytic continuation)  Suppose $h_1$ and $h_2$ are holomorphic. Let
    $\mathcal O\subseteq \Hyp$ denote an open neighborhood of
    $\overline{\mathcal D}$, and let $f\colon\mathcal O\rightarrow\C$ be a
    holomorphic function satisfying the transformation laws
\begin{equation}\label{propagation1}
  f|_k(T-I)^2 = h_1 \quad \text{and}  \quad f|_k S(T-I)^2=h_2
\end{equation}
whenever both sides are defined (that is, on $\mathcal O\cap T^{-1}\mathcal O
\cap T^{-2}\mathcal O$ for the first equation, and on $S\mathcal O\cap
T^{-1}S\mathcal O \cap T^{-2}S\mathcal O$  for the second equation).  Then
$f$ extends to a holomorphic function on $\Hyp$ satisfying
\eqref{propagation1}.

\item (Propagation of the moderate growth bound) Suppose
    $f\colon\Hyp\rightarrow\C$ is a continuous function that satisfies the
    transformation laws \eqref{propagation1} on $\Hyp$ and grows moderately
    on $\mathcal D$; i.e., there exist nonnegative constants $C_f$ and
    $N_f$ such that
\begin{equation}\label{extensionmodgrowthassump}
  |f(g\cdot i)| \le C_f\|g\|_{\textup{Frob}}^{N_f}
\end{equation}
for $g\cdot i\in\mathcal D$ (see~\eqref{HCmoderategrowth}).  Suppose also
that $h_1$ and $h_2$ have moderate growth, and let $C_{h_1}$, $C_{h_2}$,
$N_{h_1}$, and $N_{h_2}$ be nonnegative constants such that
\begin{equation}\label{hjmodgrowth}
  |h_1(g\cdot i)| \le C_{h_1}\|g\|_{\textup{Frob}}^{N_{h_1}}\quad\text{and} \quad  |h_2(g\cdot i)| \le C_{h_2}\|g\|_{\textup{Frob}}^{N_{h_2}}
\end{equation}
for $g\in\SL_2(\R)$. Then $f$ has the following moderate growth bound on
all of $\Hyp$: for some constants $C,N \ge 0$ depending only on $N_f$,
$N_{h_1}$, $N_{h_2}$, and $k$,
\begin{equation}\label{extendedmoderategrowth}
  |f(g\cdot i)| \le C(C_f+C_{h_1}+C_{h_2})\|g\|_{\textup{Frob}}^N
\end{equation}
for $g\in\SL_2(\R)$. In particular, if $f$, $h_1$, and $h_2$ are all
holomorphic, then $f\in\mathcal P$.
\end{enumerate}
\end{proposition}

The fact that the bound \eqref{extendedmoderategrowth} depends linearly on $C_f$,
$C_{h_1}$, and $C_{h_2}$ will be used to obtain uniform moderate growth bounds
in Section~\ref{sec:proofs}.

\begin{proof}
For expositional reasons we begin with the proof of part (2).  When we refer
to a ``constant'' in this proof, we mean that it can depend only on $N_f$,
$N_{h_1}$, $N_{h_2}$, and $k$, and not on $C_f$, $C_{h_1}$, or $C_{h_2}$.
By increasing the exponents if necessary, we
may assume $N_{h_1}=N_{h_2}=N_f$.  Since $f$ is continuous the moderate
growth bound \eqref{extensionmodgrowthassump} holds over the closure
$\overline{\mathcal D}$. Recall the matrices $M_i\in \slz$ from
\eqref{Midef}:
\[
(M_1,M_2,M_3,M_4,M_5,M_6)=(I,T,TS,S,ST,STS)
.
\]
It follows from the assumptions and \eqref{DclosureFclosure} that
\eqref{extensionmodgrowthassump} holds on $M_1\mathcal F=\mathcal F$,
$M_3\mathcal F=TS\mathcal F$,  $M_4\mathcal F=S\mathcal F$, and $M_6\mathcal
F = STS\mathcal F$ (i.e., whenever $g \cdot i$ is in these sets).  Our first
step is to check such an inequality on $M_2\mathcal F = T\mathcal F$ and
$M_5\mathcal F = ST\mathcal F$.

We can analyze the growth of $f$ on $T\mathcal{F}$ as follows, using $f$'s
moderate growth on $\mathcal F\cup T^{-1}\mathcal F\subseteq \mathcal D$ and
the functional equations \eqref{propagation1}. By
Lemma~\ref{lem:HCmoderategrowthPslashed}, the moderate growth of $f$ on
$T^{-1}\mathcal{F}$ implies that $f|_k T^{-1}$ has moderate growth on
$\mathcal{F}$, with suitably adjusted constants as in the lemma, and the
moderate growth of $h_1$ from \eqref{hjmodgrowth} implies that $h_1|_k
T^{-1}$ also has moderate growth.  Now we write $f|_k T =
f|_k\big(2I-T^{-1}\big) + h_1 |_k T^{-1}$ by \eqref{propagation1}, to deduce
that $f|_k T$ has moderate growth on $\mathcal{F}$; hence $f$ has moderate
growth on $T \mathcal{F}$ by Lemma~\ref{lem:HCmoderategrowthPslashed} again.
Written out more explicitly, $|f(g\cdot i)|\le (c_1 C_f+c_2 C_{h_1})
\|g\|_{\textup{Frob}}^{N_f+4|k|}$ on $M_2\mathcal F = T\mathcal F$, for some
constants $c_1,c_2>0$. Likewise, the moderate growth on $S\mathcal F$ and
$ST^{-1}\mathcal F$ implies a bound of the form $|f(g\cdot i)|\le (c_1'
C_f+c_2' C_{h_2}) \|g\|_{\textup{Frob}}^{N_f+4|k|}$ on $M_5\mathcal F =
(ST^2)T^{-1}\mathcal F=ST\mathcal F$, for some constants $c_1',c_2'>0$. Thus,
by Lemma~\ref{lem:HCmoderategrowthPslashed}, there exist constants
$C_1,N_1\ge 0$ such that
\begin{equation}\label{propMjslash}
|(f|_k{M_j})(g\cdot i)|\le C_1(C_f+C_{h_1}+C_{h_2})\|g\|_{\textup{Frob}}^{N_1}
\end{equation}
for $g\cdot i \in \overline{\mathcal F}$ and each $j\le 6$.

For $\tau\in \Hyp$ choose $\gamma=\left(
\begin{smallmatrix}
a & b \\
c & d
\end{smallmatrix}\right)\in\slz$
such that $w:=\gamma^{-1}\cdot \tau\in \overline{\mathcal F}$. Consider
\eqref{sigmaandN} with $r=\gamma$, so that $\gamma=I\cdot\gamma$ is the first
entry of $\sigma(\gamma)\vec{M}+(T-I)^2\cdot
\vec{N}_1(\gamma)+S(T-I)^2\cdot\vec{N}_2(\gamma)$.  We next bound
$(f|_k\gamma)(w)$ by expanding it according to this last decomposition. Write
the matrix entries of $\sigma(\gamma)$ as $\sigma(\gamma)_{ij}$, and the
first vector entries of $\vec{N}_1(\gamma)$ and $\vec{N}_2(\gamma)$ as
$\sum_{\delta\in\PSL_2(\Z)}n_\delta \delta$ and
$\sum_{\delta\in\PSL_2(\Z)}n'_\delta \delta$, respectively (bounds on these
quantities are given in Lemmas~\ref{lem:polynomiallybounded}
and~\ref{lem:cocyclebounds}).  Let $g_\tau$ and $g_w=\gamma^{-1}g_\tau$ be
matrices in $\SL_2(\R)$ which map $i$ to $\tau$ and $w$, respectively. Then
\begin{equation}\label{propagation2}
\begin{split}
(c w+d )^{-k} f(\tau) & =(f|_k\gamma)(g_w\cdot i) \\
& = \sum_{j\le 6} \sigma(\gamma)_{1j} (f|_k M_j)(g_w\cdot i)\\
&\quad\phantom{}+\sum_{\delta\in\PSL_2(\Z)}
\big(n_\delta (h_1|_k \delta)(g_w\cdot i) + n'_{\delta}(h_2|_k \delta)(g_w\cdot i)\big).
\end{split}
\end{equation}
Invoking \eqref{extensionmodgrowthassump}  and \eqref{hjmodgrowth}, along
with Lemma~\ref{lem:HCmoderategrowthPslashed} and \eqref{propMjslash}, yields
the bound
\[
\begin{split}
  |(c  w+d )^{-k} f(\tau)|  \le  C_2(C_f+C_{h_1}+C_{h_2})&
\Bigg(\sum_{j\le 6} |\sigma(\gamma)_{1j}| \,\|g_w\|^{N_2}_{\textup{Frob}}\\
&\ \phantom{} +\sum_{\delta}(|n_\delta| +|n_\delta'|)\|\delta\|^{N_2}_{\textup{Frob}} \|g_w\|^{N_2}_{\textup{Frob}}\Bigg)
\end{split}
\]
for some constants $C_2,N_2\ge 0$.  The estimates from
Lemmas~\ref{lem:polynomiallybounded} and~\ref{lem:cocyclebounds} show that
the sums are at most
$C_3\|\gamma\|_{\textup{Frob}}^{N_3}\|g_w\|_{\textup{Frob}}^{N_3}$ for some
constants $C_3,N_3\ge0$.  The factor of automorphy $c w + d=j(\gamma,w)$ has
a bound of the same form in \eqref{jvsnorms}, and thus
\[
|f(\tau)|\le
C_4(C_f+C_{h_1}+C_{h_2})\|\gamma\|_{\textup{Frob}}^{N_4}\|g_w\|_{\textup{Frob}}^{N_4}
\]
for some constants $C_4,N_4\ge 0$.

We now claim that for $z$ in the standard ``keyhole'' fundamental domain $\{z
\in \Hyp: |z|\ge 1, -\frac 12\le \Re(z)\le \frac 12\}$,
\begin{equation}\label{reductiontheory1}
 \|g_z\|_{\textup{Frob}} \le \|\delta g_z\|_{\textup{Frob}}
\end{equation}
for all $\delta\in\slz$, where $g_z$ is some matrix in $\SL_2(\R)$ with
$g_z\cdot i=z$ (this defines $g_z$ uniquely up to right multiplication by an
element of $\SO_2(\R)$, so the quantities on both sides of  inequality are
independent of such a choice).  Indeed, writing $z=x+iy$ and
$\delta=\left(\begin{smallmatrix}
p & q \\
r & s
\end{smallmatrix}\right)$,
we may take $g_z=y^{-1/2}\left(
\begin{smallmatrix}
y& x \\
0 & 1
\end{smallmatrix}\right)$,
compute the norms squared, and reduce the claim to showing that
\[
(x^2+y^2-1)(p^2+r^2-1)+p^2+q^2+r^2+s^2+2 x (pq+rs)-2 \ge 0
\]
for all $\delta\in \slz$.  Using the defining inequalities on $z$, the
expression on the left is at least $p^2-|pq|+q^2+r^2-|rs|+s^2-2$, which is
nonnegative because the rows of $\delta$ are nonzero and the integral
quadratic form $m^2-mn+n^2$ takes positive integer values on integer pairs
$(m,n)\neq (0,0)$.

Half of the fundamental domain $\mathcal F$ lies in the keyhole fundamental
domain, while the other half lies in its translate by $T$. Therefore
\eqref{reductiontheory1} and the submultiplicativity of the Frobenius norm
imply that there exists a positive constant $C''$ such that
\[
 \|g_w\|_{\textup{Frob}} \le C''\|\delta g_w\|_{\textup{Frob}}
\]
for all $\delta\in\slz$. (Specifically, if $w$ is not in the keyhole
fundamental domain, then $g_w = T g_z$ with $z$ in it. In that case $\|\delta
g_w\|_{\textup{Frob}} = \|\delta T g_z\|_{\textup{Frob}} \ge
\|g_z\|_{\textup{Frob}} \ge \|T\|_{\textup{Frob}}^{-1}
\|g_w\|_{\textup{Frob}}$, so we can take $C'' = \|T\|_{\textup{Frob}} =
\sqrt{3}$.) Now we can use the inequality
\[\|\gamma\|_{\textup{Frob}}\le \|g_\tau\|_{\textup{Frob}}\|g_w^{-1}\|_{\textup{Frob}}
=\|g_\tau\|_{\textup{Frob}} \|g_w\|_{\textup{Frob}} \le C''\|g_\tau\|_{\textup{Frob}}^2
\]
(the second step using the fact
$\|g_w\|_{\textup{Frob}}=\|g_w^{-1}\|_{\textup{Frob}}$ for $g_w\in
\SL_2(\R)$) to deduce that
\begin{align*}
|f(\tau)| &\le  C_4(C_f+C_{h_1}+C_{h_2})\|\gamma\|_{\textup{Frob}}^{N_4}\|g_w\|_{\textup{Frob}}^{N_4} \\
&\le C_4(C_f+C_{h_1}+C_{h_2})(C'')^{N_4} \|g_\tau\|_{\textup{Frob}}^{3N_4},
\end{align*}
which completes the proof of part~(2).

We now turn to the proof of part (1), which shares similar ingredients but
instead works with a different basis for $R/{\mathcal I}$, namely the entries
of the column vector
\[
\vec{M}'=  (M_1',M_2',M_3',M_4',M_5',M_6') = (I,T^{-1},STS,S,ST^{-1},TS)
\]
coming from \eqref{DclosureFclosure} (i.e., specifying the translates of
$\mathcal F$ that tile $\mathcal D$).  Checking that $\vec{M}'$ consists of a
basis amounts to  observing that $T^{-1} \equiv 2I-T \pmod{\mathcal I}$ and
$ST^{-1} \equiv 2S-ST \pmod{\mathcal I}$.

As was the case in \eqref{sigmaandN}, there exist a representation
$\sigma'\colon\PSL_2(\Z)\rightarrow \GL_6(\Z)$ and maps
$\vec{N}'_i\colon\PSL_2(\Z)\rightarrow R^6$ such that
\begin{equation}\label{sigma'andN'}
\vec{M}'\cdot \gamma = \sigma'(\gamma)\vec{M}'+(T-I)^2\cdot\vec{N}'_1(\gamma)+S(T-I)^2\cdot\vec{N}'_2(\gamma)
\end{equation}
for all $\gamma\in \PSL_2(\Z)$, and such that these maps satisfy the
analogous cocycle relation to \eqref{cocyclelaw}.  In particular, $\vec{M}'$
is an integral change of basis from $\vec{M}$ modulo $\mathcal I^6$, and thus
$\sigma'$ is the corresponding conjugate of $\sigma$.

By shrinking $\mathcal O$ if necessary, we assume that its  only
$\Gamma(2)$-translates which intersect it are $T^{2}\mathcal O$,
$T^{-2}\mathcal O$, $ST^{2}S\mathcal O$, and $ST^{-2}S\mathcal O$ (coming
from the boundaries of $\mathcal D$). There exists an open neighborhood
$\mathcal O_{\mathcal F}$ of $\overline{\mathcal F}$ such that $\bigcup_{j\le
6}M_j'\mathcal O_{\mathcal F}\subseteq \mathcal O$; in particular, $f|_kM_j'$
is defined on $\mathcal O_\mathcal F$ for each $j\le 6$.  By shrinking
$\mathcal O_{\mathcal F}$ if necessary, we may assume that $\gamma \mathcal
O_{\mathcal F}$ intersects $\mathcal O_{\mathcal F}$ with $\gamma \in
\PSL_2(\Z)$ only when $\gamma \overline{\mathcal F}$ and $\overline{\mathcal
F}$ share a boundary point. Accordingly, let
\begin{align*}
\Neighbors &=\{\neighbor\in \PSL_2(\Z) : \neighbor
\mathcal O_{\mathcal F}\cap \mathcal O_{\mathcal F}\neq \emptyset\}\\
&=\{S,T,T^{-1},ST^{-1},TS,TST^{-1}\}.
\end{align*}
Consider a pair $\tau,\tau'\in \mathcal O_\mathcal F$ such that
$\tau'=\neighbor\tau$ for some $\neighbor\in \Neighbors$.  We claim for each
$i\le 6$ that
\begin{equation}\label{smallextensionclaim}
\aligned
(f|_k M'_i\neighbor)(\tau) &= \sum_{j\le 6}\sigma'_{ij}(\neighbor) (f|_kM'_j)(\tau)\\
& \quad \phantom{} + (h_1|_kN'_{1i}(\neighbor))(\tau) + (h_2|_kN'_{2i}(\neighbor))(\tau),
\endaligned
\end{equation}
where $\sigma'_{ij}(\neighbor)$ denote the matrix entries of
$\sigma'(\neighbor)$, and $N'_{1i}(\neighbor)$ and $N'_{2i}(\neighbor)$
denote the vector entries of $\vec{N}'_1(\neighbor)$ and
$\vec{N}'_2(\neighbor)$, respectively. This claim would follow immediately
from \eqref{sigma'andN'} if we knew we could apply the functional equations
\eqref{propagation1}, and so all we must do is to verify the hypotheses of
\eqref{propagation1}, namely that $\gamma\tau \in \mathcal O\cap
T^{-1}\mathcal O \cap T^{-2}\mathcal O$ for every term $\gamma$ that occurs
in $N'_{1i}(\neighbor)$, and $\gamma\tau \in S\mathcal O\cap T^{-1}S\mathcal
O \cap T^{-2}S\mathcal O$ for every term $\gamma$ in $N'_{2i}(\neighbor)$. To
begin, we write \eqref{sigma'andN'} as
\begin{equation}\label{smallextension1}
M'_i\neighbor - \sum_{j\le 6}\sigma'_{ij}(\neighbor)M'_j = (T-I)^2 N'_{1i}(\neighbor) + S(T-I)^2 N'_{2i}(\neighbor).
\end{equation}
Similarly to \eqref{explicitsigmaVcalc}, it is straightforward to check for
each possible choice of $i$ and $\neighbor$ that either $N'_{1i}(\neighbor) =
N'_{2i}(\neighbor) = 0$, in which case the hypotheses of \eqref{propagation1}
hold vacuously and \eqref{smallextensionclaim} follows, or one of
$N'_{1i}(\neighbor)$ and $N'_{2i}(\neighbor)$ is zero and the other is a
group element $\gamma \in \PSL_2(\Z)$.  In other words, the right side of
\eqref{smallextension1} is of the form $(T-I)^2\gamma$ or $S(T-I)^2\gamma$
with $\gamma\in\PSL_2(\Z)$ whenever is it nonzero. Since
$M'_i\neighbor\tau=M'_i\tau' \in \mathcal{O}$ and $M_j'\tau \in \mathcal O$ for all $i$ and $j$, all the
terms on the left side of \eqref{smallextension1} map $\tau$ to points in
$\mathcal O$; therefore $\tau$ is mapped to $\mathcal{O}$ by all of
$T^2\gamma, T\gamma, \gamma$ or $ST^2\gamma, ST\gamma, S\gamma$ (depending on
which form the right side has). This assertion is the hypothesis needed for
\eqref{propagation1} to apply at the point $\gamma\tau$, and
\eqref{smallextensionclaim} follows by applying the slash operator to $f$.

Having shown \eqref{smallextensionclaim}, we now extend $f$ to arbitrary
$w\in\Hyp$ by imitating \eqref{propagation2} (but with slightly different
notation).  Namely, we write $w$ as $\gamma\tau$ with $\tau\in \mathcal
O_\mathcal F$ and $\gamma\in \PSL_2(\Z)$, and we define $f(w)$ by
\[
\begin{split}
j(\gamma,\tau)^{-k} f(w) &= (f|_k\gamma)(\tau)\\
&=
\sum_{j\le 6}\sigma'_{1j}(\gamma)(f|_kM'_j)(\tau) + (h_1|_kN'_{11}(\gamma))(\tau)+ (h_2|_kN'_{21}(\gamma))(\tau),
\end{split}
\]
where as usual $j(\gamma,\tau)$ is the factor of automorphy from
\eqref{eq:automorphyfactordef}.  That $f(w)$ is well defined follows from
\eqref{smallextensionclaim} and the cocycle relation for $\sigma'$,
$\vec{N}'_1$, and $\vec{N}'_2$ analogous to \eqref{cocyclelaw}. Specifically,
suppose that $w = \gamma \tau = \gamma' \tau'$, with $\tau, \tau' \in
\mathcal{O}_{\mathcal F}$ and $\gamma,\gamma' \in \PSL_2(\Z)$. Then $\tau' =
\neighbor \tau$ for some $\neighbor \in \Neighbors$, and hence $\gamma =
\gamma' \neighbor$. Starting with the definition of $j(\gamma,\tau)^{-k}
f(w)$ given above, we expand the right side using $\sigma'(\gamma) =
\sigma'(\gamma')\sigma'(\neighbor)$ and the cocycle relations to obtain
\begin{align*}
& \quad \sum_{i \leq 6} \sigma'_{1i}(\gamma') \left(\sum_{j \leq 6} \sigma'_{ij}(\neighbor) (f |_k M'_j)(\tau) + (h_1 |_k N'_{1i} (\neighbor)) (\tau) + (h_2 |_k N'_{2i} (\neighbor)) (\tau)\right)
\\
 & \quad \phantom{} + (h_1 |_k N'_{11}(\gamma') \neighbor) (\tau)
 +  (h_2 |_k N'_{21}(\gamma') \neighbor) (\tau).
\end{align*}
Applying \eqref{smallextensionclaim} to the expression in parentheses shows
that $j(\gamma,\tau)^{-k} f(w)$ equals
\[
\sum_{i \leq 6} \sigma'_{1i}(\gamma') (f|_k M'_i\neighbor)(\tau) +
(h_1 |_k N'_{11}(\gamma') \neighbor) (\tau)
 +  (h_2 |_k N'_{21}(\gamma') \neighbor) (\tau),
\]
or equivalently
\[
j(\neighbor,\tau)^{-k}\left( \sum_i \sigma'_{1i}(\gamma') (f |_k M'_i) + (h_1 |_k N'_{1i}(\gamma'))  +  (h_2 |_k N'_{2i}(\gamma')) \right) (\neighbor \tau)
\]
by the definition of the slash operator $|_k\neighbor$.
Finally, using $\neighbor \tau = \tau'$ and $j(\gamma,\tau) =
j(\gamma',\tau') j(\neighbor,\tau)$ yields
\[
j(\gamma',\tau')^{-k} f(w) = \sum_i \sigma'_{1i}(\gamma') (f |_k M'_i)(\tau')
+ (h_1 |_k N'_{1i}(\gamma')) (\tau') +  (h_2 |_k N'_{2i}(\gamma'))  (\tau'),
\]
which is the definition of $f(w)$ as we would obtain by using $\tau'$ and
$\gamma'$. Therefore $f(w)$ is well defined.

We have now defined a holomorphic function on $\Hyp$ agreeing with $f$ on the
open neighborhood $\bigcup_{i\le 6}M'_j\mathcal O_\mathcal F$ of $\mathcal D$
(and thus also the original neighborhood $\mathcal{O}$). Since the
holomorphic identity \eqref{propagation1} holds for the original function, it
must hold for the extension as well.
\end{proof}

Equation \eqref{rerestated1.6to1.8} recasts the interpolation formula
\eqref{eqn:IF} from Theorem~\ref{theorem:interpolation} in terms of
properties of the function $F\colon\Hyp\times \R^d\to\R$. To construct $F$,
we make the Ansatz that $\widetilde{F}|^\tau_{d/2}S$ can (essentially) be
written as a Laplace transform, which is equivalent to a contour integral
construction introduced by Viazovska in her work on sphere packing \cite{V}
and motivated by cycle integrals of modular forms appearing in \cite{DIT}.
Specifically, we will construct an integral kernel $\K$ on $\Hyp \times \Hyp$
such that
\begin{equation}\label{FfromKlaplace}
F(\tau,x)=e^{\pi i \tau |x|^2}+4 \sin\mathopen{}\big(\pi |x|^2/2\big)^2\mathclose{}\,\int_0^\infty \K(\tau,it)\,e^{-\pi |x|^2 t}\,dt,
\end{equation}
at least for $|x|$ sufficiently large and $\tau$ inside the fundamental
domain $\mathcal F$ for the action of $\slz$ on $\Hyp$ defined in
\eqref{domainF}.  Such a formula requires initially restricting $\tau$,
because the kernel $\K(\tau,z)$ will have poles; see part (1) of
Theorem~\ref{thm: K}.  This is the reason we have taken the fundamental
domain $\mathcal{F}$ to be different from the usual keyhole fundamental
domain for $\slz\backslash\Hyp$, so that in particular $\mathcal{F}$ does not
intersect the imaginary axis $z=it$.

It is natural here to decompose the proposed kernel into eigenfunctions of
$|^\tau_{d/2}S$ as $\K=\frac 12(\K_+ + \K_-)$ using \eqref{diagonalizex},
where
\begin{equation}\label{Kpmdef}
\K_\pm(\tau,z):=\K(\tau,z)\mid^{\tau}_{d/2}(I\mp S).
\end{equation}
Note that $\K_\pm$ has eigenvalue $\mp1$ under $|^\tau_{d/2}S$, not $\pm1$.
However, the notation is consistent with our use of signs elsewhere, such as
in part~(2) of the next theorem.  This labeling reflects the fact that
$\K_\pm$ contributes to the decomposition of $F$ into eigenfunctions of the
Fourier transform with eigenvalue $\pm 1$.  In terms of the relationship
$\widetilde F=(e^{\pi i \tau |x|^2 }-F)|^\tau_{d/2}S$ between $\widetilde{F}$
and $F$ when $d$ is a multiple of $8$, the right side contains $-F
|^\tau_{d/2} S$, which introduces an extra minus sign.

The next result, which is proved in Section~\ref{sec:sub:explicitkernels},
shows the existence of a kernel $\K$ that will be used to construct the
solution $F$ of the functional equations \eqref{rerestated1.6to1.8} via
\eqref{FfromKlaplace} and later extensions of this formula.

\begin{theorem}\label{thm: K} For dimensions $d=8$ and $24$
there exist unique meromorphic functions $\K=\K^{(d)}$ and
$\K_\pm=\K_\pm^{(d)}$ for $d=8$ and $24$ (related by \eqref{Kpmdef}) on
$\Hyp\times\Hyp$ satisfying the following properties.
\begin{enumerate}
\item For  fixed $z\in\Hyp$ the poles of $\K(\tau,z)$ and $\K_\pm(\tau,z)$
    as functions of $\tau$ are all simple and  contained in the $\SL_2(\Z)$-orbit of
    $z$.
\item The kernel $\K$ satisfies the functional equations
\[
\K(\tau,z)\mid^\tau_{d/2}(T-I)^2=0 \quad\text{and}\quad
\K(\tau,z)\mid^\tau_{d/2}S(T-I)^2=0;
\]
that is, $\K|^\tau_{d/2}r=0$ for all $r\in \mathcal I$.  Also, $\K_+$ and
$\K_-$ satisfy the functional equations
\[
\K_\pm(\tau,z)\mid^\tau_{d/2}(T-I)^2=0\quad\text{and}\quad
\K_\pm(\tau,z)\mid^\tau_{d/2}(S\pm I)=0;
\]
that is, $\K_\pm|^\tau_{d/2}r=0$ for all $r\in\mathcal I_\pm$.
\item For $z\in\Hyp$ and $r\in R$, the residues of $\K$ and $\K_\pm$ as
    functions of $\tau$ satisfy
\[
\operatorname{Res}_{\tau=z}\big(\K|^\tau_{d/2}r\big)=-\frac{1}{2\pi}\phi(r)
\]
and
\[
\operatorname{Res}_{\tau=z}\big(\K_\pm|^\tau_{d/2}r\big)=-\frac{1}{2\pi}\phi_\pm(r),
\]
where $\phi\colon R/\mathcal I\to\C$ is the linear map defined by
\begin{flalign*}
&&\phi(I)&=0, & \phi(T)&=1, & \phi(TS)&=0,&&\\
&&\phi(S)&=0, & \phi(ST)&=0, & \phi(STS)&=0&&
\end{flalign*}
and $\phi_\pm\colon R/\mathcal I_\pm\to\C$ is the linear map from
\eqref{phiepsdef} defined by
\begin{flalign*}
&&\phi_\pm(I)&=0, &\phi_\pm(T)&=1, & \phi_\pm(TS)&=0. &&
\end{flalign*}
(Note that $\phi$ and $\phi_\pm$ are defined modulo $\mathcal{I}$ and
$\mathcal{I}_\pm$ thanks to part (2) above, so it suffices to define them
on the bases from Proposition~\ref{prop:ideals2}.)
\item The functions
\[
\Delta(\tau)\Delta(z)(j(\tau)-j(z))\K^{(8)}_\pm(\tau,z)
\]
and
\[
\Delta(\tau)\Delta(z)^2(j(\tau)-j(z))\K^{(24)}_\pm(\tau,z)
\]
are in the class $\mathcal P$ both as functions of $\tau$ and $z$.
Furthermore, for $z$ fixed the kernels satisfy the bounds
\begin{equation}\label{Kpmboundsatinf}
  \K_\pm^{(8)}(\tau,z) = O\big(|\tau e^{2\pi i \tau}|\big) \quad \text{and} \quad
  \K_\pm^{(24)}(\tau,z) = O\big(|\tau e^{4\pi i \tau}|\big)
\end{equation}
as $\Im(\tau)\to\infty$.
\end{enumerate}
\end{theorem}

It is not difficult to check that the functional equations for $\K_\pm$ in
part~(2) are equivalent to those for $\K$, and the same is true for the
residue calculations in part~(3).  We have stated both cases for
completeness.

Our first step in proving Theorem~\ref{thm: K} is to solve the functional
equations satisfied by $\K_\pm$. Part~(2) of the theorem asserts that
$\K_\pm(\tau,z) \mid^\tau_{d/2} r= 0$ for all $r \in \mathcal{I}_\pm$, and we
will see in Proposition~\ref{prop:Kpmslashtilde} that $\K_\pm(\tau,z)
\mid^z_{2-d/2} r= 0$ for all $r \in \widetilde{\mathcal{I}}_\pm$.
Furthermore, part~(4) reduces the problem to the case of functions in
$\mathcal{P}$, with the factors of $\Delta(\tau)$ and $\Delta(z)$ changing
the weights of the actions in $\tau$ and $z$.

\subsection{Functions in $\mathcal P$ annihilated by $\mathcal I_\pm$ and $\widetilde{\mathcal I}_\pm$}
\label{sec:annihilatorcalcs}

Given a right ideal $J$ of $R=\C[\PSL_2(\Z)]$ and an even integer $k$, let
\[
\operatorname{Ann}_k(J,\mathcal P) = \{ f\in \mathcal P : f|_k r=0 \text{~for all~} r\in J\}.
\]
The following four propositions describe $\operatorname{Ann}_k(J,\mathcal P)$
for $J=\mathcal I_{\pm}=(S\pm I)\cdot R+(T-I)^2\cdot R$ and
$\widetilde{\mathcal I}_{\pm}=(T-2I+T^{-1}\mp 2S)\cdot R +(I\mp STS)\cdot R$
from \eqref{freeideals}.

\begin{proposition}\label{prop: HFE plus}  Let $k$ be an even integer.
The space $\operatorname{Ann}_k(\mathcal I_{+},\mathcal P)$, i.e., the
solutions $f\in \mathcal P$ to the system
\begin{equation}\label{QCeven1}
f|_k(T-I)^2 =0 \quad \text{and} \quad f|_k(S+I)=0,
\end{equation}
is equal to
\begin{equation}\label{eqn: HFE plus space}
\varphi_2\,\mathcal{M}_{k-2}(\SL_2(\Z))+\varphi_0\,\mathcal{M}_k(\SL_2(\Z))+\varphi_{-2}\,{\mathcal M}_{k+2}(\SL_2(\Z)),
\end{equation}
where
\[
\varphi_2(\tau)=\tau E_2(\tau)^2-\frac{6i}{\pi}E_2(\tau),\quad
\varphi_0(\tau)=\tau E_2(\tau)-\frac{3i}{\pi},\quad \text{and} \quad
\varphi_{-2}(\tau)=\tau.
\]
In particular, the dimension of the space of solutions equals
$\max\mathopen{}\big(0,\lceil \frac{ k}{4}\rceil+1\big)\mathclose{}$.
\end{proposition}

\begin{proof}
Suppose that $f\in \mathcal P$ is a solution to \eqref{QCeven1}. The vector
space spanned by $f$, $f|_k T$, and $f|_k TS$ is preserved by the action of
$R$, because $I,T,TS$ are a basis  of $R/\mathcal{I}_+$ by part~(2) of
Proposition~\ref{prop:ideals2}, and the remainder of the proof will consist
of a careful study of this action.

Set
\[
\begin{split}
g_0 &:=f|_k(T-I) = f|_kT - f,\\
g_1&:=f, \text{ and} \\
g_2 &:=f|_k(T-I)S = f|_kTS + f,
\end{split}
\]
from which it follows that
\begin{equation}\label{QCevenactions}
\aligned
 \begin{pmatrix}g_0|_k T\\g_1|_k T\\g_2|_k T\end{pmatrix} &=
\begin{pmatrix}1&0&0\\1&1&0\\1&2&1\end{pmatrix}
\begin{pmatrix}g_0\\g_1\\g_2\end{pmatrix} \qquad \text{and} \\
 \begin{pmatrix}g_0|_k S\\g_1|_k S\\g_2|_k S\end{pmatrix} &=
\begin{pmatrix}0&\phantom{-}0&1\\ 0&-1&0\\1&\phantom{-}0&0\end{pmatrix}
\begin{pmatrix}g_0\\g_1\\g_2\end{pmatrix}
.
\endaligned
\end{equation}
Define $h_0,h_1,h_2$ by
\begin{equation}\label{eqn: define h012}
\begin{pmatrix} h_0(\tau) \\ h_1(\tau) \\ h_2(\tau) \end{pmatrix} = \begin{pmatrix} 1 & 0 & 0 \\ -2 \tau & 2 & 0 \\ \tau^2 & -2 \tau & 1 \end{pmatrix} \begin{pmatrix}g_0(\tau)\\g_1(\tau)\\g_2(\tau)\end{pmatrix}.
\end{equation}
Denote the above matrix by $M(\tau)$, and the column vectors by $H(\tau)$ and
$G(\tau)$, so that $H(\tau) = M(\tau)G(\tau)$, and denote the matrices from
\eqref{QCevenactions} for the actions of $|_kT$ and $|_kS$ on $G$ by $T_G$
and $S_G$, respectively. Then
\[
(H |_k T)(\tau) = M(\tau+1) T_G M(\tau)^{-1} H(\tau) = \begin{pmatrix} 1 & 0 & 0 \\ 0 & 1 & 0 \\ 0 & 0 & 1  \end{pmatrix} H(\tau) = H(\tau)
\]
and
\[
(H |_k S)(\tau)  = M(-1/\tau) S_G M(\tau)^{-1} H(\tau) = \begin{pmatrix} \tau^2 &
\tau & 1 \\ 0 & 1 & 2/\tau \\ 0 & 0 & 1/\tau^2  \end{pmatrix} H(\tau).
\]
In other words,
\begin{equation}\label{prop311aligned}
\aligned
h_2 |_{k-2} T & = h_2, & h_2 |_{k-2} S &= h_2, \\
  h_1|_k T& =h_1, &  (h_1|_k S)(\tau)&=h_1(\tau)+2\tau^{-1}h_2(\tau), \\
   h_0|_{k+2} T& =h_0, & \text{and} \quad (h_0|_{k+2} S)(\tau)&=h_0(\tau)+\tau^{-1}h_1(\tau)+\tau^{-2}h_2(\tau).
\endaligned
\end{equation}
Therefore, $h_2\in {\mathcal M}_{k-2}(\SL_2(\Z))$ because of this invariance
and since it has moderate growth (it is an element of $\mathcal P$). It
follows from \eqref{E2slash} and the transformation properties of $h_1$ that
$h_1-\frac{\pi i}{3}h_2E_2 \in {\mathcal M}_{k}(\slz)$; in particular, $h_1$
is a quasimodular form of weight $k$ and depth at most $1$ for $\slz$.
Similarly, $h_0-\frac{\pi i}{6}h_1 E_2-\frac{\pi^2}{36}h_2 E_2^2\in {\mathcal
M}_{k+2}(\slz)$, and $h_0$ is therefore a quasimodular form of weight $k+2$
and depth at most $2$ for $\slz$.

Thus we have shown the existence of modular forms $f_0\in \mathcal
M_{k+2}(\SL_2(\Z))$, $f_1\in \mathcal M_{k}(\SL_2(\Z))$, and  $f_2\in
\mathcal M_{k-2}(\SL_2(\Z))$ such that $h_0=f_0+f_1 E_2+f_2 E_2^2$. Expanding
and comparing with the last transformation law for $h_0|_{k+2}S$ in
\eqref{prop311aligned}, we deduce from the periodicity of $h_0$, $h_1$, and
$h_2$ that $h_1=-\frac{6i}{\pi}(f_1+2f_2 E_2)$ and
$h_2=-\frac{36}{\pi^2}f_2$, and thus
\[
f=g_1=\tau h_0+\frac12h_1=\varphi_{-2}f_0+\varphi_0 f_1+\varphi_2 f_2.
\]
Hence $f$ lies in \eqref{eqn: HFE plus space}, and it is straightforward to
verify that all elements of \eqref{eqn: HFE plus space} satisfy the
conditions in \eqref{QCeven1}.  Finally, the dimension assertion follows
directly from \eqref{modulardimensionformula}.
\end{proof}

\begin{proposition}\label{prop: HFE minus}
Let $k$ be an even integer.
The space $\operatorname{Ann}_k(\mathcal I_{-},\mathcal P)$, i.e., the space
of solutions $f\in \mathcal P$ to the system
\begin{equation}\label{QCodd1}
f|_k(T-I)^2 =0 \quad \text{and} \quad f|_k(S-I)=0,
\end{equation}
is equal to
\begin{equation}\label{eqn: HFE minus space}
\psi_4\,{\mathcal M}_{k-4}(\SL_2(\Z))+\psi_2\,{\mathcal M}_{k-2}(\SL_2(\Z))+\psi_{0}\,{\mathcal M}_{k}(\SL_2(\Z)),
\end{equation}
where
\[
\psi_4=\xi_4\cdot {\mathcal L} +(\xi_4|_4 S)\cdot {\mathcal L}_S,\quad
\psi_2=\xi_2\cdot  {\mathcal L} +(\xi_2|_2 S)\cdot  {\mathcal L}_S,\quad
\text{and} \quad \psi_{0}=1,
\]
with
\[
\xi_4=U^2+W^2-2V^2 \quad\text{and}\quad
\xi_2=U+W
\]
defined in terms of the theta functions in \eqref{UVWdef}. In particular, the
dimension of the space of solutions equals $\max\mathopen{}\big(0,\lceil
\frac{k-2}{4}\rceil+1\big)\mathclose{}$.
\end{proposition}

\begin{proof}
As with the proof of Proposition~\ref{prop: HFE plus}, it is straightforward
to use \eqref{UVWslash} and \eqref{loglambdaslash} to verify that all
elements of \eqref{eqn: HFE minus space} satisfy \eqref{QCodd1}, as well as
to deduce the dimension formula from \eqref{modulardimensionformula}.  Thus
we will show that any solution   $f\in \mathcal P$   to \eqref{QCodd1} lies
in \eqref{eqn: HFE minus space}. Set $h_0:=f$, $h_1:=f|_k(T-I)S$, and
$h_2:=f|_k(T-I)$, from which it follows that
\[
 \begin{pmatrix}h_0|_k T\\h_1|_k T\\h_2|_k T\end{pmatrix}=
\begin{pmatrix}1&\phantom{-}0&\phantom{-}1\\0&-1&-1\\0&\phantom{-}0&\phantom{-}1\end{pmatrix}
\begin{pmatrix}h_0\\h_1\\h_2\end{pmatrix} \ \ \text{and}
\ \ \begin{pmatrix}h_0|_k S\\h_1|_k S\\h_2|_k S\end{pmatrix} =
\begin{pmatrix}1&0&0\\0&0&1\\0&1&0\end{pmatrix}
\begin{pmatrix}h_0\\h_1\\h_2\end{pmatrix}
.
\]
From this we see that $h_2|_kT$ and $h_2|_k(ST^2S)$ equal $h_2$. Since $f\in
\mathcal P$, $h_2(\tau)=f(\tau+1)-f(\tau)$ grows at most polynomially as
$\Im(\tau)\to\infty$, and the same is true for $h_2|_k S = h_1$.
Hence $h_2\in{\mathcal M}_{k}(\Gamma_0(2))$, where
$\Gamma_0(2)=\langle T,ST^2S\rangle$ is the subgroup of matrices in $\slz$
whose bottom-left entries are even. Furthermore,  $h_2|_k(I+S+ST)=0$ and
$h_0$ satisfies the system
\begin{equation}\label{eqn: h0FE}
h_0|_k(S-I)=0\quad \text{and} \quad
h_0|_k (T-I)=h_2.
\end{equation}
A solution to  \eqref{eqn: h0FE} is given by the function
\[
g_0:=\frac{1}{\pi i}\left(h_2\cdot {\mathcal L}+h_1\cdot {\mathcal L}_S\right),
\]
as can be seen by inserting the transformation laws \eqref{loglambdaslash},
and thus $h_0-g_0\in {\mathcal M}_k(\SL_2(\Z))$, since it satisfies the
homogeneous version of \eqref{eqn: h0FE} and inherits  polynomial growth from
$\mathcal P$ and \eqref{lambdaasympt}.

Thus
\[
f=h_0=h\cdot\mathcal L+(h|_k S)\cdot\mathcal L_S+g,
\]
where  $h=\frac{1}{\pi i}h_2\in {\mathcal M}_k(\Gamma_0(2))$   satisfies
$h|_k(I+S+ST)=0$,  and $g\in \mathcal{M}_k(\SL_2(\Z))$.  We now invoke
\eqref{EichlerZagier} and write $h$ uniquely as
\[
h=f_1+Uf_2+Vf_3+U^2f_4+V^2f_5+UVWf_6,
\]
where $f_1\in {\mathcal M}_{k}(\slz)$, $f_2,f_3\in {\mathcal M}_{k-2}(\slz)$,
$f_4,f_5\in {\mathcal M}_{k-4}(\slz)$, and $f_6\in {\mathcal M}_{k-6}(\slz)$.
Using \eqref{UVWslash}, we can compute the set of all solutions to the system
$h|_k(I+S+ST)=h|_k (T-I)=0$ in $\mathcal M_k(\Gamma(2))$. For instance, the
latter condition alone forces $f_6 = f_4 = 0$ and $f_2 = -2f_3$. We find that
the space of solutions is $\xi_4\mathcal M _{k-4}(\slz)+\xi_2 \mathcal
M_{k-2}(\slz)$, which then implies $f$ lies in \eqref{eqn: HFE minus space}.
\end{proof}

For small values of $k$ the  solution spaces in Propositions~\ref{prop: HFE
plus} and~\ref{prop: HFE minus} are small enough to rule out certain
asymptotic behaviors.  For example, the following lemma, which is used below
to show various uniqueness statements, can be proved by direct computation of
asymptotics as $\Im(\tau)\to\infty$ using $q$-expansions (i.e., expansions in
possibly fractional powers of $q = e^{2\pi i z}$ or $e^{2\pi i \tau}$, where
we allow polynomials in $z$ or $\tau$ as coefficients).

\begin{lemma}\label{lem:noPsolutions} No nonzero
$f\in \operatorname{Ann}_4(\mathcal I_\pm,\mathcal P)$ satisfies the bound
$f(\tau)=o(1)$ as $\Im(\tau)\rightarrow\infty$ with $-1 \le \Re(\tau) \le 1$.
Likewise, no nonzero $f\in \operatorname{Ann}_{12}(\mathcal I_\pm,\mathcal
P)$ satisfies the bound $f(\tau)=o\big(e^{-2\pi \Im(\tau)}\big)$ as
$\Im(\tau)\rightarrow\infty$ with $-1 \le \Re(\tau) \le 1$.
\end{lemma}

\begin{proposition}\label{prop: HFE plus dual}
Let $k$ be an even integer. The space
$\operatorname{Ann}_k(\widetilde{\mathcal I}_{+},\mathcal P)$, i.e., the
solutions $f\in \mathcal P$ to the system
\begin{equation}\label{QCeven3}
f|_k(T-2I+T^{-1}-2S) =0 \quad \text{and} \quad f|_k(I-STS) =0,
\end{equation}
is equal to
\begin{equation}\label{eqn: HFE plus dual space}
\widetilde{\varphi}_2\,\mathcal{M}_{k-2}(\slz)+
\widetilde{\varphi}_0\,\mathcal{M}_k(\slz)+
\widetilde{\varphi}_{-2}\,\mathcal{M}_{k+2}(\slz),
\end{equation}
where
\[
\widetilde{\varphi}_2(z)=z^2 ((E_2|_2S)(z))^2, \quad
\widetilde{\varphi}_0(z)=z^2(E_2|_2S)(z),\quad \text{and} \quad
\widetilde{\varphi}_{-2}(z)=z^2.
\]
In particular, the dimension of the space of solutions equals
$\max\mathopen{}\big(0,\lceil \frac{ k}{4}\rceil+1\big)\mathclose{}$.
\end{proposition}

\begin{proof}
Given a solution $f\in\mathcal P$ to \eqref{QCeven3},  let
\[
g_0=f|_k S,\quad
g_1=-\textstyle{\frac{1}{2}}f|_k(I+S-T),\quad \text{and} \quad
g_2=f.
\]
Using the fact that $I-STS=S(I-T)S$, one verifies that $g_0,g_1,g_2$ satisfy
the transformation law \eqref{QCevenactions}.  As was shown in the proof of
Proposition~\ref{prop: HFE plus}, $g_0$ is consequently a quasimodular form
of weight $k+2$ and depth at most $2$ for $\slz$. Thus, $g_0 \in E_2^2
\mathcal{M}_{k-2}(\slz) + E_2 \mathcal{M}_{k}(\slz) +
\mathcal{M}_{k+2}(\slz)$. It follows that $f=g_0|_k S$ lies in \eqref{eqn:
HFE plus dual space}. The other aspects of the proof are straightforward to
verify as above.
\end{proof}

\begin{proposition}\label{prop: HFE minus dual}
Let $k$ be an even integer. The space
$\operatorname{Ann}_k(\widetilde{\mathcal I}_{-},\mathcal P)$, i.e., the
solutions $f\in \mathcal P$ to the system
\begin{equation}\label{QCodd3}
f|_k(T-2I+T^{-1}+2S) =0 \quad \text{and} \quad f|_k(I+STS) =0,
\end{equation}
is equal to
\begin{equation}\label{eqn: HFE minus dual space}
\widetilde{\psi}_4\,\mathcal{M}_{k-4}(\slz)+\widetilde{\psi}_2\,\mathcal{M}_{k-2}(\slz)+\widetilde{\psi}_{0}\,
\mathcal{M}_{k}(\slz),
\end{equation}
where
\[
\widetilde{\psi}_4=U^2-V^2,
\quad \widetilde{\psi}_2=W, \quad \text{and} \quad
\widetilde{\psi}_{0}=\mathcal{L}.
\]
In particular, the dimension of the space of solutions equals
$\max\mathopen{}\big(0,\lceil \frac{ k-2}{4}\rceil+1\big)\mathclose{}$.
\end{proposition}

\begin{proof}
For the same reasons as before, we again restrict our attention to showing
that solutions $f\in \mathcal P $ to \eqref{QCodd3} lie in \eqref{eqn: HFE
minus dual space}.  Let $g=f|_k(S+T-I)$. We check that $g|_kS = g|_kT = g$,
and so $g \in {\mathcal M}_k(\slz)$ since it has moderate growth. Because
\[
\big(g\cdot \mathcal{L}\big) |_k (S+T-I) = \pi i g,
\]
the function $h:=f-\frac{1}{\pi i}\,g\cdot \mathcal{L}$ has moderate growth
and satisfies the homogeneous equations
\[
h|_k(S+T-I)=0\quad\text{and}\quad h|_k S(T+I)=0,
\]
with the latter equation being a restatement of the second equation in
\eqref{QCodd3} since $\mathcal{L}_S|_0 (T+I) = 0$. Then $h$ is a modular form
of weight $k$ for $\Gamma(2)$, because
\[
T^2-I=(S+T-I)(T+I)-S(T+I)
\]
and
\[
ST^2S-I=(ST+S)(TS-S).
\]
We complete the proof by arguing, as in the proof of Proposition~\ref{prop:
HFE minus}, and again using \eqref{EichlerZagier}, that these conditions
force
\[
h\in \widetilde{\psi}_4{\mathcal M}_{k-4}(\slz) +
\widetilde{\psi}_2{\mathcal M}_{k-2}(\slz). \qedhere
\]
\end{proof}

\subsection{Uniqueness of $\K_\pm$ and their transformation properties}\label{sec:sub:uniqueness}

The characterizations of $\operatorname{Ann}_k({\mathcal I}_\pm,\mathcal P)$
in Section~\ref{sec:annihilatorcalcs} can now be used to establish the
uniqueness assertion in Theorem~\ref{thm: K}.  Indeed, suppose $\K(\tau,z)$
and $\K'(\tau,z)$ are two kernels satisfying conditions (1)--(4). Then for
fixed $z\in\Hyp$, the function $\tau\mapsto
\K_{\pm}(\tau,z)-\K'_{\pm}(\tau,z)$ is annihilated by $|^\tau_{d/2}r$ for all
$r\in \mathcal I_{\pm}$ by part~(2), and it is holomorphic at all
$\tau\in\Hyp$ by parts~(1) and~(3).  Furthermore, it is in $\mathcal{P}$ by
part~(4) combined with the following lemma (recall that $\mathcal{I}
\subseteq \mathcal{I}_\pm$). By Lemma~\ref{lem:noPsolutions} the growth
condition \eqref{Kpmboundsatinf} forces $\K_{\pm}(\tau,z)-\K'_{\pm}(\tau,z)$
to vanish identically, and hence uniqueness follows.

\begin{lemma} \label{lem:deltajtz}
Suppose $f\colon\Hyp\to\C$ is holomorphic, $k$ is an even integer, $f |_k r =
0$ for all $r \in \mathcal{I}$, and $\tau \mapsto \Delta(\tau)(j(\tau)-j(z))
f(\tau)$ is in $\mathcal{P}$ for some fixed $z\in\Hyp$.  Then $f \in
\mathcal{P}$.
\end{lemma}

\begin{proof}
The function $f$ has moderate growth on $\mathcal{D}$, because
\[
\big(\Delta(\tau)(j(\tau)-j(z))\big)^{-1}
\]
is bounded as $\tau$ approaches any cusp from inside $\mathcal
D$ (specifically, $\Delta$ is a cusp form and $j$ has a pole at infinity).
Now Part~(2) of Proposition~\ref{prop:propagationandextension} with
$h_1=h_2=0$ shows that $f \in \mathcal{P}$, as desired.
\end{proof}

Next we show that the kernels $K_{\pm}(\tau,z)$ also satisfy modular
functional equations in the variable $z$.  In order to do this, first we
generalize the residue statements of part~(3) of Theorem~\ref{thm: K} to an
action in the variable $z$, in addition to $\tau$. Suppose first that $f$ is
a meromorphic function on $\Hyp$, with at most simple poles. Then for any
$\alpha \in \slz$ and $k \in \Z$,
\begin{equation}\label{residue1}
\operatorname{Res}_{\tau=\tau_0}(f|_k \alpha)(\tau) =j(\alpha,\tau_0)^{2-k}\operatorname{Res}_{\tau=\alpha\tau_0}f(\tau)
\end{equation}
in terms of the factor of automorphy from \eqref{eq:automorphyfactordef}.
Therefore
\[
\operatorname{Res}_{\tau=z} \big(\K|_{d/2}^\tau \alpha |_{2-d/2}^z \alpha\big)(\tau,z) =  \operatorname{Res}_{\tau=\alpha z}\K(\tau,\alpha z),
\]
since both sides are equal to $\operatorname{Res}_{\tau=z}
j(\alpha,z)^{d/2-2}\big(\K|^\tau_{d/2}\alpha\big)(\tau,\alpha z)$. This
allows us to compute
\begin{equation} \label{eq:residuecomp}
\aligned
\operatorname{Res}_{\tau=z} \big(\K|_{d/2}^\tau \alpha |_{2-d/2}^z \beta\big)(\tau,z) & =
\operatorname{Res}_{\tau=z} \big(\K|_{d/2}^\tau \alpha\beta^{-1} |_{d/2}^\tau \beta |_{2-d/2}^z \beta\big)(\tau,z)\\
& =
\operatorname{Res}_{\tau=\beta z} \big(\K|_{d/2}^\tau \alpha\beta^{-1}\big)(\tau,\beta z)
\endaligned
\end{equation}
for $\alpha,\beta\in\slz$. Of course these identities also hold with $\K$
replaced by $\K_\pm$.  Combining them with \eqref{residue1}, we see that part
(3) of Theorem~\ref{thm: K} generalizes to
\begin{equation}\label{genpart4}
\operatorname{Res}_{\tau=\gamma z}\K_\pm|_{d/2}^\tau \alpha|_{2-d/2}^z \beta=-\frac{j(\gamma,z)^{d/2-2}}{2\pi}\phi_\pm(\alpha\gamma\beta^{-1})
\end{equation}
for $\alpha,\beta,\gamma\in\PSL_2(\Z)$, because
\[
\begin{split}
\frac{\operatorname{Res}_{\tau = \gamma z} \K_\pm |^\tau_{d/2} \alpha |^z_{2-d/2} \beta}{j(\gamma,z)^{d/2-2}}
 &=  \operatorname{Res}_{\tau=z} \K_\pm |^\tau_{d/2} \alpha |^z_{2-d/2} \beta |^\tau_{d/2} \gamma \qquad \text{(by \eqref{residue1})}\\
 &= \operatorname{Res}_{\tau=z} \K_\pm |^\tau_{d/2} \alpha \gamma |^z_{2-d/2} \beta\\
 &= \operatorname{Res}_{\tau = \beta z} \big(\K_\pm |^\tau_{d/2} \alpha \gamma \beta^{-1}\big) (\tau,\beta z) \qquad \text{(by \eqref{eq:residuecomp})}\\
& = -\frac{\phi_\pm(\alpha \gamma \beta^{-1})}{2\pi}.
\end{split}
\]
Furthermore, for all $r\in
R=\C[\PSL_2(\Z)]$ and $\alpha\in\PSL_2(\Z)$, we have the residue formula
\begin{equation}\label{genpart5}
\operatorname{Res}_{\tau=\alpha z} \big(\K_\pm |_{2-d/2}^z r\big)(\tau,z) = -\frac{j(\alpha,z)^{d/2-2}}{2\pi}\phi_\pm(\alpha\cdot \iota(r));
\end{equation}
indeed, by linearity it suffices to verify this formula in the case that
$r=\iota(r)^{-1}\in\PSL_2(\Z)$, in which case it follows from
\eqref{genpart4}.

\begin{proposition}\label{prop:Kpmslashtilde}
Let $\K_\pm$ be the kernels whose existence and uniqueness are guaranteed by
Theorem~\ref{thm: K}.  Then
\[
\K_\pm(\tau,z)\mid^z_{2-d/2}r=0
\]
for all $r\in\widetilde{\mathcal I}_{\pm}$.
\end{proposition}

\begin{proof}
Let $r\in\widetilde{\mathcal I}_\pm$ and consider the function
$g_r:=\K_\pm|_{2-d/2}^z r$ on $\Hyp\times\Hyp$.  Part (1) of
Theorem~\ref{thm: K} implies that all possible poles of $g_r(\tau,z)$ in
$\tau$ lie at points $\tau=\alpha z$, where $\alpha\in\slz$.  The residues at
such points are computed by formula \eqref{genpart5}, and actually vanish
since $\phi_\pm(\alpha\cdot \iota(r))=0$ by Proposition~\ref{prop:ideals3}.
Thus $\tau \mapsto g_r(\tau,z)$ is holomorphic, and it is in $\mathcal{P}$ by
Lemma~\ref{lem:deltajtz} and part~(4) of Theorem~\ref{thm: K}. Thus it
vanishes by Lemma~\ref{lem:noPsolutions} and the bounds
\eqref{Kpmboundsatinf}.
\end{proof}

\subsection{Proof of Theorem~\ref{thm: K} and kernel asymptotics}\label{sec:sub:explicitkernels}
We can now write down the kernels in Theorem~\ref{thm: K} explicitly. We
claim that they are given by
\[
\aligned
\K^{(d)}_+(\tau,z) &  =
\begin{pmatrix}
   \varphi_{-2}(\tau) \\
 \varphi_0(\tau) \\
 \varphi_2(\tau)
\end{pmatrix}^t
\cdot
\Upsilon_+^{(d)}(\tau,z)
\cdot
\begin{pmatrix}
  \widetilde\varphi_{-2}(z) \\
  \widetilde\varphi_{0}(z) \\
  \widetilde\varphi_{2}(z)
\end{pmatrix}
\qquad \text{and}
\\
\K^{(d)}_-(\tau,z) & =
\begin{pmatrix}
 \psi_{0}(\tau) \\
 \psi_2(\tau) \\
 \psi_4(\tau)
\end{pmatrix}^t
\cdot
\Upsilon_-^{(d)}(\tau,z)
\cdot
\begin{pmatrix}
 \widetilde\psi_{0}(z) \\
 \widetilde\psi_{2}(z) \\
 \widetilde\psi_{4}(z)
\end{pmatrix}
\endaligned
\]
in terms of the bases defined in Propositions~\ref{prop: HFE plus}
through~\ref{prop: HFE minus dual}, for certain coefficient matrices
$\Upsilon_\pm^{(d)}(\tau,z)$. The entries of these matrices can be written
with denominators given by $\Delta(\tau) \Delta(z) (j(\tau)-j(z))$ when $d=8$
and $\Delta(\tau) \Delta(z)^2 (j(\tau)-j(z))$ when $d=24$, in accordance with
part~(4) of Theorem~\ref{thm: K}, and numerators that are modular forms for
$\SL_2(\Z)$ in $\tau$ and $z$. Specifically, the matrix $\Upsilon^{(d)}_+$
will have rows of weight $d/2+2$, $d/2$, and $d/2-2$ in $\tau$ and columns of
weight $4-d/2$, $2-d/2$, and $-d/2$ in $z$, while $\Upsilon^{(d)}_-$ will
have rows of weight $d/2$, $d/2-2$, and $d/2-4$ in $\tau$ and columns of
weight $2-d/2$, $-d/2$, and $-2-d/2$ in $z$. (Note that although these
weights can be negative, the numerator of each entry has positive weights
when we use the denominator specified above.) By Propositions~\ref{prop: HFE plus}
through~\ref{prop: HFE minus dual}, these weights ensure that $\K^{(d)}_\pm
|^\tau_{d/2} r=0$ for all $r \in \mathcal{I}_\pm$ and $\K^{(d)}_\pm
|^z_{2-d/2} r=0$ for all $r \in \widetilde{\mathcal{I}}_\pm$, as required by
Theorem~\ref{thm: K} and Proposition~\ref{prop:Kpmslashtilde}.

The entries of the matrices $\Upsilon^{(d)}_\pm$ are by no means obvious, but
they are uniquely determined by the list of required properties for
$\K^{(d)}_\pm$ in Theorem~\ref{thm: K}. One can solve for them explicitly,
with the following result. Let
\[
(f_{-2}, f_0,f_2,\dots, f_{14}) = (E_{10}/\Delta, 1, E_{14}/\Delta, E_4, E_6, E_8, E_{10}, \Delta, E_{14}),
\]
let $\Pi_{k_1,k_2,k_3}$ be the diagonal matrix
$\mathop{\textup{diag}}(f_{k_1}, f_{k_2}, f_{k_3})$, and let us abbreviate
$j_\tau = j(\tau)$, $j_z = j(z)$, $j_{\tau,z} = j(\tau)-j(z)$, and
$\twelvecubed=1728$. Then we define
\[
\Upsilon_+^{(8)}(\tau,z)=\frac{ \Pi_{6,4,2}(\tau) \begin{pmatrix}
j_{\tau,z} & 0 & 1 \\
-2j_{\tau,z} & -2 & 0 \\
1 & 0 & 0
\end{pmatrix} \Pi_{12,10,8}(z) }{36\pi^{-2} i \Delta(z) (j(\tau)-j(z))}
,
\]
\[
\Upsilon_+^{(24)}(\tau,z) = \frac{ \Pi_{14,12,10}(\tau) \begin{pmatrix}
6 & 0 &  \twelvecubed j_{\tau,z}^{-1} - 6 \\
-12 j_\tau + 5 \twelvecubed & -2 \twelvecubed j_{\tau,z}^{-1} & 12 j_\tau-7 \twelvecubed \\
\twelvecubed j_{\tau,z}^{-1}+6 & 0 & -6
\end{pmatrix} \Pi_{4,2,0}(z) }{36\twelvecubed \pi^{-2} i \Delta(z)}
 ,
\]
\[
\Upsilon_{-}^{(8)}(\tau,z)
= \frac{ \Pi_{4,2,0}(\tau) \begin{pmatrix}
-2 \twelvecubed & 0 & 0 \\
0 & 1 & -1 \\
0 & \twelvecubed-j_\tau & j_\tau
\end{pmatrix} \Pi_{10,8,6}(z)}{2\twelvecubed\pi \Delta(z) (j(\tau)-j(z))}
,
\]
and
\[
  \Upsilon_{-}^{(24)}(\tau,z) =  \frac{ \Pi_{12,10,8}(\tau) \begin{pmatrix}
  - 2\twelvecubed  & -2\twelvecubed j_{\tau,z} & 0 \\
  0 & j_\tau + 2j_{\tau,z} & -1 \\
  0 & \twelvecubed - 2j_{\tau,z} - j_\tau & 1
\end{pmatrix} \Pi_{2,0,-2}(z) }{2\twelvecubed\pi\Delta(z)(j(\tau)-j(z))}.
\]
Note that the particular form in which we have written these matrices has no
special significance beyond being fairly compact (and it does not make the
denominators readily apparent).

All that remains in the proof of Theorem~\ref{thm: K} is to verify that these
formulas work. Poles occur in the entries of $\Upsilon_\pm^{(d)}(\tau,z)$
only from dividing by $j(\tau)-j(z)$. Thus, the poles of these matrix entries
in $\tau$ are contained in the $\SL_2(\Z)$-orbit of $z$, and they are all
simple poles unless $j'(z)=0$. Because $j' = -2\pi i E_4^2E_6/\Delta$ by
\eqref{jprime}, that can happen only if $E_4(z)=0$ or $E_6(z)=0$.  The
Eisenstein series $E_4$ and $E_6$ have single roots on the $\SL_2(\Z)$-orbits
of $e^{2\pi i/3}$ and $i$, respectively, and no other roots by
\cite[Proposition~5.6.5]{CohenStromberg}; the function $j = 1728
E_4^3/(E_4^3-E_6^2)$ takes the values $0$ and $1728$ on these orbits.  Using
these facts, one can check that the matrices $\Upsilon_\pm^{(d)}(\tau,z)$
never have poles in $\tau$ of order greater than one.

We can calculate residues by using the identity
\[
\lim_{\tau \to z} \frac{f(\tau)(\tau-z)}{j(\tau)-j(z)} = \frac{f(z)}{j'(z)} = \frac{i f(z) \Delta(z)}{2\pi E_{14}(z)}
\]
for a holomorphic function $f$ on a neighborhood of $z$, where we interpret
the right side by continuity if $f(z)=j'(z)=0$. We find that for both $d=8$
and $24$,
\[
\aligned
 \operatorname{Res}_{\tau=z}\Upsilon^{(d)}_+(\tau,z) & = \frac{\pi}{72}
\left(\begin{smallmatrix}
0 & \phantom{-}0 & \phantom{-}1 \\
0 & -2 & \phantom{-}0 \\
1 & \phantom{-}0 &  \phantom{-}0
\end{smallmatrix}\right) \qquad \text{and} \\
\operatorname{Res}_{\tau=z}\Upsilon^{(d)}_-(\tau,z) & = \frac{i}{6912\pi^2 \Delta(z)}
\left(\begin{smallmatrix}
-3456\Delta(z) & 0 & 0 \\
0 & E_4(z)^2 & -E_6(z) \\
0 & -E_6(z) &  E_4(z)
\end{smallmatrix}\right).
\endaligned
\]
For matrices $\alpha,\beta\in \slz$,
\[
(\K^{(d)}_+|_{d/2}^\tau \alpha |_{2-d/2}^z \beta)(\tau,z)  =
\begin{pmatrix}
 (\varphi_{-2}|_{-2}\alpha)(\tau) \\
 (\varphi_0|_0 \alpha)(\tau) \\
 (\varphi_2|_2 \alpha)(\tau)
\end{pmatrix}^t
\cdot
\Upsilon_+^{(d)}(\tau,z)
\cdot
\begin{pmatrix}
  (\widetilde\varphi_{-2}|_{-2}\beta)(z) \\
  (\widetilde\varphi_{0}|_0\beta)(z) \\
  (\widetilde\varphi_{2}|_2\beta)(z)
\end{pmatrix}
\]
and
\[
(\K^{(d)}_- |_{d/2}^\tau \alpha |_{2-d/2}^z \beta)(\tau,z)  =
\begin{pmatrix}
 (\psi_{0}|_0 \alpha)(\tau) \\
 (\psi_2|_2 \alpha)(\tau)  \\
 (\psi_4|_4 \alpha)(\tau)
\end{pmatrix}^t
\cdot
\Upsilon_-^{(d)}(\tau,z)
\cdot
\begin{pmatrix}
  (\widetilde\psi_{0}|_0\beta)(z) \\
  (\widetilde\psi_{2}|_2\beta)(z) \\
  (\widetilde\psi_{4}|_4\beta)(z)
\end{pmatrix},
\]
using the $\slz$-automorphy properties of the matrix entries of
$\Upsilon^{(d)}_\pm(\tau,z)$.

Theorem~\ref{thm: K} can now be straightforwardly verified from these
formulas. The uniqueness was already shown in
Section~\ref{sec:sub:uniqueness}, we have seen that property~(1) holds, and
property~(2) follows from Propositions~\ref{prop: HFE plus} and~\ref{prop:
HFE minus} and the weights of the coefficients of $\varphi_k$ and $\psi_k$.
The residue property~(3) follows from computing $\operatorname{Res}_{\tau =
z} \K_+^{(d)} |^\tau_{d/2} \alpha$ as
\[
\frac{\pi}{72}
\begin{pmatrix}
 (\varphi_{-2}|_{-2}\alpha)(z) \\
 (\varphi_0|_0 \alpha)(z) \\
 (\varphi_2|_2 \alpha)(z)
\end{pmatrix}^t
\cdot
\begin{pmatrix}
0 & \phantom{-}0 & \phantom{-}1 \\
0 & -2 & \phantom{-}0 \\
1 & \phantom{-}0 &  \phantom{-}0
\end{pmatrix}
\cdot
\begin{pmatrix}
  \widetilde\varphi_{-2}(z) \\
  \widetilde\varphi_{0}(z) \\
  \widetilde\varphi_{2}(z)
\end{pmatrix}
\]
and $\operatorname{Res}_{\tau = z} \K_-^{(d)} |^\tau_{d/2} \alpha$ as
\[
\frac{i}{6912\pi^2 \Delta(z)}
\begin{pmatrix}
 (\psi_{0}|_0 \alpha)(z) \\
 (\psi_2|_2 \alpha)(z)  \\
 (\psi_4|_4 \alpha)(z)
\end{pmatrix}^t
\cdot
\begin{pmatrix}
-3456\Delta(z) & 0 & 0 \\
0 & E_4(z)^2 & -E_6(z) \\
0 & -E_6(z) &  E_4(z)
\end{pmatrix}
\cdot
\begin{pmatrix}
  \widetilde\psi_{0}(z) \\
  \widetilde\psi_{2}(z) \\
  \widetilde\psi_{4}(z)
\end{pmatrix}
\]
for $\alpha \in \{I,T,TS\}$.  Finally, the membership in $\mathcal{P}$
asserted in property~(4) follows from the formulas by using
$j\Delta\in\mathcal{P}$, and \eqref{Kpmboundsatinf} follows from a
$q$-expansion calculation. This completes the proof of Theorem~\ref{thm: K}.

For the detailed analysis of the generating functions $F$ and $\widetilde{F}$
in Section~\ref{sec:proofs}, we will need to understand the non-decaying
asymptotics of $\K^{(d)}_\pm(\tau,it)$ as $t \to \infty$ with $t \in \R$.  To
do so, we expand $\K^{(d)}_\pm(\tau,z)$ as a series in powers of $e^{\pi i
z}$, whose coefficients are functions of $\tau$ and polynomials in $z$ of
degree at most $2$, and we define $\mathcal{G}^{(d)}_\pm(\tau,z)$ to be the
sum of the $e^{n \pi i z}$ terms for $n \le 0$.  Then
$\K^{(d)}_\pm(\tau,it)=\mathcal G^{(d)}_\pm(\tau,it)+O(t^2 e^{-\pi t})$ as $t
\to\infty$.

Explicit calculation shows that these functions can be expanded in $z$ as
\[
\mathcal{G}^{(d)}_\pm(\tau,z) = \sum_{k=-1}^0 \sum_{j=0}^1 z^j e^{2\pi i k z} \mathcal{G}^{(d)}_{k,j,\pm}(\tau),
\]
where the coefficient functions $\mathcal{G}^{(d)}_{k,j,\pm}(\tau)$ are
polynomials in $\tau$, $E_2(\tau)$, $U(\tau)$, $V(\tau)$, $W(\tau)$,
$\mathcal{L}(\tau)$, and $\mathcal{L}_S(\tau)$, and so in particular they lie
in $\mathcal{P}$. Note that there are no terms with $j=2$ in
$\mathcal{G}^{(d)}_\pm(\tau,z)$, although such terms occur elsewhere in these
series expansions of $\K^{(d)}_\pm(\tau,z)$. In addition, \[\mathcal
G_{-1,j,\pm}^{(8)} =0,\] which will be important in Sections~\ref{sec:proofs}
and~\ref{sec:inequality}.  We furthermore define $\mathcal{G}^{(d)} =
\frac{1}{2}\big(\mathcal{G}^{(d)}_{+}+\mathcal{G}^{(d)}_{-}\big)$ to
correspond with $\K^{(d)}$, and set $\mathcal{G}^{(d)}_{k,j} =
\frac{1}{2}\big(\mathcal{G}^{(d)}_{k,j,+}+\mathcal{G}^{(d)}_{k,j,-}\big)$ so
that
\begin{equation}\label{Gjkexpansion}
\mathcal{G}^{(d)}(\tau,z) = \sum_{k=-1}^0 \sum_{j=0}^1 z^j e^{2\pi i k z} \mathcal{G}^{(d)}_{k,j}(\tau).
\end{equation}

We will need the following lemma in Section~\ref{sec:proofs}.

\begin{lemma}\label{lem:boundsonKnew}
Let
\begin{equation}\label{npmconstants}
n_{+,\tau}=0,\ \ n_{-,\tau}=1,\ \ n_{+,z}=2,\ \  n_{-,z}=1,\ \
\widehat{n}_\tau^{(8)}=2,\ \ \text{and}\ \ \widehat{n}_\tau^{(24)}=4.
\end{equation}
For each $\delta>0$ and $\gamma \in \PSL_2(\Z)$, there exists a constant
$C=C_{\gamma,\delta}>0$ such that for $d \in \{8,24\}$,
\begin{equation}\label{quotientbounds1new}
 \left|\big(\K^{(d)}_\pm|^\tau_{d/2} \gamma|^z_{2-d/2} S\big)(\tau,z)\right|   \le  C\left|\frac{ e^{\pi i (n_{\pm,\tau}\tau+n_{\pm,z}z)}\tau^2 z^2}{\Delta(\tau)\Delta(z)(j(\tau)-j(z))}\right|,
\end{equation}
\begin{equation}\label{quotientbounds1.5}
 \left|  \big((\K^{(d)}_\pm-\mathcal{G}^{(d)}_\pm)|^\tau_{d/2} \gamma\big)(\tau,z)\right|   \le  C\left|\frac{ e^{\pi i n_{\pm,z}z}\tau^2 z^2}{\Delta(\tau)\Delta(z)(j(\tau)-j(z))}\right|,
\end{equation}
and
\begin{equation}\label{quotientbounds2new}
\left|\big(\K^{(d)}_\pm |^z_{2-d/2} S\big)(\tau,z)\right| \le
C\left|\frac{ e^{\pi i (\widehat{n}_\tau^{(d)}\tau+  z)}\tau^2 z^2}{\Delta(\tau)\Delta(z)(j(\tau)-j(z))}\right|,
\end{equation}
for $\Im(\tau),\Im(z)\ge \delta$ with $j(\tau) \ne j(z)$.
\end{lemma}

\begin{proof}
The kernels $\K^{(d)}_\pm$ are annihilated by $\mathcal{I}_\pm$ under
$|^\tau_{d/2}$, and one can check that the same is true for
$\mathcal{G}^{(d)}_\pm$. Thus, to cover all $\gamma \in \PSL_2(\Z)$ it
suffices to check the cases $\gamma = I$, $T$, or $TS$, because they form a
basis of $R/\mathcal{I}_\pm$ by Proposition~\ref{prop:ideals2}.

For each of these three cases, the explicit formulas for the kernels show
that the expressions inside the absolute values on the left sides, when
multiplied by $\Delta(\tau)\Delta(z)(j(\tau)-j(z))$, are finite sums of the
form $\sum_{j=0}^2\sum_{k=0}^{2}\phi_{jk}(\tau,z)\tau^j z^k$, where
$\phi_{jk}$ is a holomorphic function on $\Hyp\times\Hyp$ satisfying
$\phi_{jk}(\tau+2,z)=\phi_{jk}(\tau,z)=\phi_{jk}(\tau,z+2)$. The claims then
follow from the form of the double power series expansion of
$\phi_{jk}(\tau,z)$ in $e^{\pi i \tau}$ and $e^{\pi iz}$, which can be shown
by direct calculation to vanish to the asserted order in these exponentials.
\end{proof}

Note that the denominators of $\Delta(\tau)\Delta(z)(j(\tau)-j(z))$ are not
an obstacle to computing asymptotics.  For example, if $\tau$ is fixed and
$z=it$ with $t \in (0,\infty)$, then $\Delta(\tau)\Delta(it)(j(\tau)-j(it))$
approaches $-\Delta(\tau)$ as $t\rightarrow\infty$, and it is asymptotic to
$-t^{-12}\Delta(\tau)$ as $t\rightarrow 0$ by modularity, as in the proof of
Lemma~\ref{lem:deltajtz}.

\section{Proof of the interpolation formula}\label{sec:proofs}

\subsection{The continuation of $F$ from $\mathcal D$ to $\Hyp$}\label{sec:sub:continueFtoH}

Let $d$ be $8$ or $24$, and let $\K$ and $\K_\pm$ be the kernels from
Theorem~\ref{thm: K} for this value of $d$ (we also suppress superscripts
$(d)$ for $\mathcal{G}$ and similar terms). As before, we identify radially
symmetric functions of $x\in \R^d$ with functions of $r=|x|$.  Analytic
continuation will be important in both $\tau$ and $r$, and so we view $r$ as
a complex variable.

Using the kernel $\K$, we can now construct the generating function $F$ we
need for Theorem~\ref{thm:FEimpliesIF}. We start with the formal expression
\eqref{FfromKlaplace}, which we decompose as
$F(\tau,r)=F_1(\tau,r)+F_2(\tau,r)$, where
\begin{equation}\label{FfromKlaplace2}
F_1(\tau,r)=e^{\pi i \tau r^2} \quad \text{and} \quad F_2(\tau,r) =
4 \sin\mathopen{}\big(\pi r^2/2\big)^2\mathclose{}\,\int_0^\infty \K(\tau,it)\,e^{-\pi r^2 t}\,dt.
\end{equation}
We have not yet addressed when the integral is defined. Part~(1) of
Theorem~\ref{thm: K} shows that the poles of the integrand in $\tau$ are
constrained to $\slz$-translates of the imaginary axis.  None of these
translates intersects $\mathcal D$ aside from the imaginary axis itself,
which in fact does not contribute any poles since part~(3) of
Theorem~\ref{thm: K} shows that $\K(\tau,it)$ is holomorphic at $\tau=it$ and
$\tau=i/t$ (as $\phi(I)=\phi(S)=0$).  Thus, poles are not an obstacle when
$\tau \in \mathcal{D}$.

To analyze the convergence of the integral in \eqref{FfromKlaplace2}, we must
understand how $\K(\tau,it)$ behaves as $t \to 0$ or $t \to \infty$.
Inequality~\eqref{quotientbounds1new} (with $\gamma=I$) in
Lemma~\ref{lem:boundsonKnew} shows that it decays exponentially in $1/t$ as
$t \to 0$ (see also the paragraph after Lemma~\ref{lem:boundsonKnew}), while
it grows at most exponentially as $t \to \infty$ because
$\K(\tau,it)=\mathcal G(\tau,it)+O(t^2 e^{-\pi t})$. Therefore the integral
in \eqref{FfromKlaplace2} is convergent for $\tau\in\mathcal D$ and $r$ in
some open set $\mathcal O\subseteq \C$ containing all sufficiently large real
numbers. More precisely, our analysis of $\mathcal{G}(\tau,it)$ in
Section~\ref{sec:sub:explicitkernels} shows that it suffices to take $|r|>0$
for $d=8$ and $|r|>\sqrt{2}$ for $d=24$.  Furthermore, $F_2(\tau,r)$ is
holomorphic as a function of either variable on these sets. (This is a
general principle about integrals of analytic functions: by Morera's theorem,
it suffices to show that contour integrals vanish, and that reduces to the
analyticity of the integrand by Fubini's theorem.)

Theorem~\ref{thm:FEimpliesIF} also involves the function
$\widetilde{F}=(e^{\pi i r^2\tau}-F)|^\tau_{d/2}S = - F_2|^\tau_{d/2}S$
defined just above \eqref{rerestated1.6to1.8}.  Since $S\mathcal D=\mathcal
D$ and $-{\mathcal K}|^\tau_{d/2}S=\frac{1}{2}(\mathcal K_+-\mathcal
K_{-})=\widehat{\mathcal K}$, we may use \eqref{FfromKlaplace2} to write
\begin{equation}\label{FtildefromKhatlaplace}
\widetilde{F}(\tau,r)= 4 \sin\mathopen{}\big(\pi r^2/2\big)^2\mathclose{}\,\int_0^\infty \widehat{\K}(\tau,it)\,e^{-\pi r^2 t}\,dt
\end{equation}
for $\tau\in\mathcal D$ and $r\in\mathcal O$.

We next analytically continue $r \mapsto F_2(\tau,r)$ to an open neighborhood
of $\R$ in $\C$ by breaking up the integral as follows. For $\tau \in
\mathcal D$, $r\in\mathcal O$, and fixed $p>0$ we decompose $F_2(\tau,r)$ as
\begin{equation}\label{decomposeF2}
F_2(\tau,r) = F_{2,\textup{low}}(\tau,r)+ F_{2,\textup{trunc}}(\tau,r)+ F_{2,\textup{high}}(\tau,r),
\end{equation}
where
\begin{equation}\label{truncatingF2}
\aligned
 F_{2,\textup{low}}(\tau,r) & =   4 \sin\mathopen{}\big(\pi r^2/2\big)^2\mathclose{}\int_0^p  \K(\tau,it)\,e^{-\pi r^2 t}\,dt, \\
  F_{2,\textup{trunc}}(\tau,r) & = 4 \sin\mathopen{}\big(\pi r^2/2\big)^2\mathclose{}  \int_p^\infty  (\K(\tau,it)-\mathcal{G}(\tau,it))\,e^{-\pi r^2 t}\,dt,  \quad \text{and}  \\
   F_{2,\textup{high}}(\tau,r) & = 4 \sin\mathopen{}\big(\pi r^2/2\big)^2\mathclose{}\int_p^\infty  \mathcal{G}(\tau,it)\,e^{-\pi r^2 t}\,dt
\\
 &=\frac{4 \sin\mathopen{}\big(\pi r^2/2\big)^2\mathclose{}}{\pi} \sum_{k=-1}^{0}e^{-p\pi (2k+r^2)}\bigg( \frac{ \mathcal{G}_{k,0}(\tau)}{2k+r^2}\\
 & \qquad\qquad\qquad\qquad\quad\phantom{}+\frac{i(1+2k p \pi + p \pi r^2) {\mathcal{G}}_{k,1}(\tau) }{\pi (2k+r^2)^2} \bigg),
\endaligned
\end{equation}
where $\mathcal{G}_{k,0}(\tau)$ and $\mathcal{G}_{k,1}(\tau)$ are as in
\eqref{Gjkexpansion}. As noted above, the first two integrals are absolutely
convergent because of the exponential decay in $z$ in
\eqref{quotientbounds1new} and \eqref{quotientbounds1.5} with $\gamma=I$.
Thus both integrals define holomorphic functions for $r$ in a neighborhood of
$\R$ in $\C$. Furthermore, they are both averages of Gaussians, with the
first weighted by $\K(\tau,it)$ (which decays exponentially in $1/t$ as
$t\rightarrow 0$ by \eqref{quotientbounds1new}, and hence damps the
contributions of $e^{-\pi r^2 t}$ for $t$ small) and the second by	
$\K(\tau,it)-\mathcal G(\tau,it)$ (which decays exponentially as
$t\rightarrow\infty$).  To analyze the Schwartz seminorms, we can use the
identity
\begin{equation}
\label{eq:moderateGaussian}
\max_{r \in \R} \big| r^c e^{-\pi  r^2 t}\big| = \left(\frac{c}{2\pi e t}\right)^{c/2}
\end{equation}
for $c\ge0$ and $t>0$, where we interpret the right side of
\eqref{eq:moderateGaussian} as $1$ when $c=0$. It follows by differentiating
under the integral sign and using \eqref{eq:moderateGaussian} to bound the
integrand that the first two integrals in \eqref{truncatingF2} define
Schwartz functions, and that for any fixed $k,\ell \ge 0$, the Schwartz
seminorms
\[
 \max_{r\in\R}\left|r^k\frac{d^\ell}{dr^\ell}F_{2,\textup{low}}(\tau,r)\right| \quad \text{and} \quad
 \max_{r\in\R}\left|r^k\frac{d^\ell}{dr^\ell}F_{2,\textup{trunc}}(\tau,r)\right|
\]
are bounded as $\tau$ ranges over any fixed compact subset of $\mathcal D$.
(Recall from Lemma~\ref{lem:Schwartz+radial} that it suffices to use radial
seminorms.)  Note that the uniformity in $\tau$ makes use of
\eqref{quotientbounds1.5} and \eqref{quotientbounds2new}, and not just our
previous estimate $\K^{(d)}(\tau,it)=\mathcal G^{(d)}(\tau,it)+O(t^2 e^{-\pi
t})$ as $t \to\infty$, although that would suffice to analyze the case of
fixed $\tau$.

The explicit evaluation of $F_{2,\textup{high}}(\tau,r)$ shows that it is in
fact entire in $r$, since the possible singularities at $r^2=-2k$ are
compensated for by the vanishing of $\sin\mathopen{}\big(\pi
r^2/2\big)^2\mathclose{}$ at those points. Furthermore, for any fixed $k,\ell
\ge 0$ the map $\tau\mapsto \max_{r\in\R}\big|r^k
\frac{d^\ell}{dr^\ell}F_{2,\textup{high}}(\tau,r)\big|$ is in $\mathcal P$
since each $\mathcal{G}_{k,j}\in\mathcal P$. Formulas \eqref{decomposeF2} and
\eqref{truncatingF2} accordingly serve as our definition of $F_2(\tau,r)$,
and in turn $F(\tau,r)=e^{\pi i \tau r^2}+F_2(\tau,r)$, for arbitrary
$\tau\in \mathcal D$ and $r\in \R$. Note that although the integrals in
\eqref{truncatingF2} all depend on a parameter $p$, their sum is independent
of $p$ for large $r$ by construction and hence for all $r\in\R$ by analytic
continuation. Aside from a technical point in the proof of
Proposition~\ref{prop:extendF2} that requires two different choices of $p$,
this parameter will be fixed and hence suppressed from the notation (in
Section~\ref{sec:inequality} a similar construction uses the value $p=1.01$).

Having defined $F_2(\tau,r)$ for $\tau\in\mathcal D$, we next turn to its
analytic continuation to the full upper half-plane $\Hyp$.  The anticipated
transformation laws \eqref{rerestated1.6to1.8} for $F$ can be restated via
\eqref{FfromKlaplace2} as
\begin{equation}\label{F2restated1.6to1.8}
\aligned
&F_2(\tau+2,r)-2F_2(\tau+1,r)+F_2(\tau,r)   = 4 e^{\pi i r^2(\tau+1)}\sin\mathopen{}\big(\pi r^2/2\big)^2\mathclose{} \quad \text{and}\\
&(F_2|_{d/2}^\tau S)(\tau+2,r) -2 (F_2|_{d/2}^\tau
S)(\tau+1,r) + (F_2|_{d/2}^\tau S)(\tau,r)     =0.
\endaligned
\end{equation}
In fact, we will use \eqref{F2restated1.6to1.8} to define the extension from
$\mathcal D$ to $\Hyp$, via Proposition~\ref{prop:propagationandextension}.
To satisfy the hypotheses of Proposition~\ref{prop:propagationandextension},
we will need the following extension of $F_2$ to an open neighborhood of the
closure $\overline{\mathcal{D}}$ of $\mathcal D$:

\begin{proposition}\label{prop:extendF2}
There exists an open subset $\mathcal{D}_+$ of $\Hyp$ containing
$\overline{\mathcal{D}}$ such that for each $r$, the function $\tau\mapsto
F_2(\tau,r)$ extends to a holomorphic function on $\mathcal{D}_+$, which
satisfies the recurrences \eqref{F2restated1.6to1.8} whenever the left sides
are defined. Moreover, $r\mapsto F_2(\tau,r)$ is a Schwartz function for each
$\tau$, and for any fixed integers $k,\ell\ge 0$ the Schwartz seminorm
$\max_{r\in \R}\left|r^k \frac{d^\ell}{dr^\ell}F_2(\tau,r)\right|$ is bounded
as $\tau$ ranges over any compact subset of $\mathcal{D}_+$.
\end{proposition}

\begin{proof}
We first show the continuation to the right of $\Re(\tau)=1$.  The
continuation to the left of $\Re(\tau)=-1$ is nearly identical, and those
across the bottom two semicircles of the boundary of $\mathcal D$ are
drastically simpler (owing to the homogeneity of the second equation in
\eqref{F2restated1.6to1.8} and the absence of poles as one crosses those
boundaries).

\begin{figure}
\begin{tikzpicture}[scale=4]
\fill[gray!40!white] (-1,0) rectangle (1,2);
\fill[gray!40!white] (1,0) arc (0:180:1);
\fill[gray!40!white] (1,0)--(1,1) arc (90:180:1);
\fill[gray!40!white] (0,0) arc (0:90:1)--(-1,0);
\fill[gray!40!white] (-1,0) rectangle (0,2);
\fill[gray!40!white] (-1,0) rectangle (0,2);
\fill[white]  (1,0) arc (0:180:0.5);
\fill[white]  (0,0) arc (0:180:0.5);
\draw[gray] (0,0)--(0,2);
\draw[black] (-1,0)--(-1,2);
\draw[black] (1,0)--(1,2);
\draw[gray] (0.5,{sqrt(3/4)}) arc (60:120:1);
\draw[gray] (1,1) arc (90:120:1);
\draw[gray] (-1,1) arc (90:60:1);
\draw[black] (0,0) arc (0:180:0.5);
\draw[black] (1,0) arc (0:180:0.5);
\draw[gray] (0.5,0.5) -- (0.5,{sqrt(3/4)});
\draw[gray] (-0.5,0.5) -- (-0.5,{sqrt(3/4)});

\draw (1.25,1.25) node {$\mathcal U$} ; \draw (-1.25,1.25) node
{$\mathcal U'$} ;

\draw (0.5,0.4) node {$TS\mathcal U$} ; \draw (-0.5,0.4) node {$
T^{-1}S\mathcal U'$} ;

\draw[black] (1.5,{sqrt(3/4)}) -- (1.5,2); \draw[black] (-1.5,{sqrt(3/4)}) -- (-1.5,2);

\draw[black] (1,0) arc (180:120:1); \draw[black] (-1,0) arc (0:60:1);

\draw[black] (1,0) arc (0:120:{1/3}); \draw[black] (0,0) arc
(180:60:{1/3});

\draw[black] (-1,0) arc (180:60:{1/3}); \draw[black] (0,0) arc
(0:120:{1/3});

\filldraw[black] (0,0) circle (0.01125); \draw  (0,0) node[below=2pt] {$0$};
\filldraw[black] (-1,0) circle (0.01125); \draw  (-1,0)node[below=2pt]
{$-1$}; \filldraw[black] (1,0) circle (0.01125); \draw (1,0) node[below=2pt]
{$1$} ;
\end{tikzpicture}
\caption{The regions $\mathcal U$, $\mathcal U'=\{\tau \in
\C:-\overline{\tau}\in\mathcal U\}$, $TS\mathcal U = S\mathcal U'$, and $T^{-1}S\mathcal
U'= S\mathcal U=\{\tau \in \C:-\bar{\tau}\in TS\mathcal U\}$ in the proof of continuation
of $F_2(\tau,r)$. The shaded region is the domain $\mathcal D$ from
\eqref{domainD}, and $\overline{\mathcal{D}}$, $\mathcal{U}$, $\mathcal{U'}$,
$TS\mathcal U$, and $T^{-1}S\mathcal U'$ together form
$\mathcal{D}_+$.\label{fig:regionU}}
\end{figure}
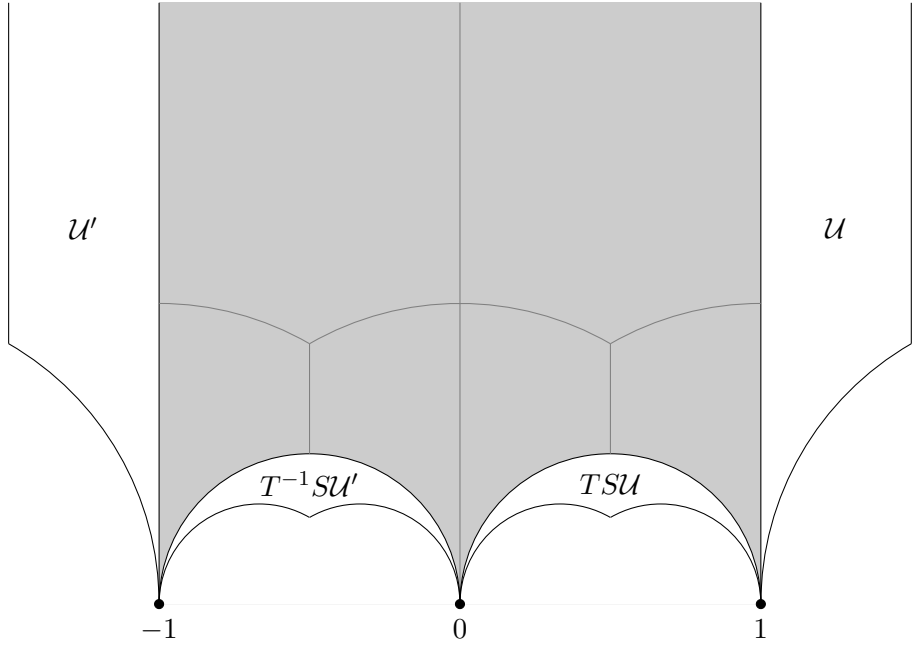

Let $\mathcal U$ denote the interior of the closure of
\[
T\left\{\tau\in \mathcal F : \Re(\tau)<\frac{1}{2}\right\}\cup TST^{-1}\left\{\tau\in
\mathcal F : \Re(\tau)>\frac{1}{2}\right\}.
\]
Then its closure $\overline{\mathcal U}$ includes the line $\Re(\tau)=1$ but
intersects no other $\slz$-translate of the imaginary axis (see
Figure~\ref{fig:regionU}). For $\tau\in \mathcal U$ and $|r|$ sufficiently
large, define
\[
F_2^\sharp(\tau,r):=4 \sin\mathopen{}\big(\pi r^2/2\big)^2\mathclose{}\int_0^\infty \K(\tau,it)\,e^{-\pi r^2 t}\,dt,
\]
i.e., by the same integral formula as in \eqref{FfromKlaplace2}.  It, too,
has an analytic continuation to a neighborhood of $\R$ in $\C$ using the
integral formulas in \eqref{truncatingF2}, for exactly the same reason as
before.

To relate $F_2^\sharp$ and $F_2$, we will examine the poles of $\K(\tau,z)$
using part~(3) of Theorem~\ref{thm: K} and equation \eqref{genpart5} (with
the group algebra element $r$ in this equation given by $r=1$).   For
$\tau\in\overline{\mathcal U}$ the only possible singularities of the
integrand with $z=it$ are at $z=T^{-1}\tau$ or $z = ST^{-1}\tau$, with $\tau$
on the left boundary of $\mathcal{U}$.  For $\alpha \in \PSL_2(\Z)$, there is
a pole at $z = \alpha \tau$ if and only if $\phi(\alpha^{-1}) \ne 0$, because
$z = \alpha \tau$ means $\tau = \alpha^{-1} z$, and by \eqref{genpart5} the
residue is
\[
\aligned
\operatorname{Res}_{z = \alpha \tau} \K(\tau,z) &= \big(\operatorname{Res}_{w = \alpha^{-1} z} \K(w,z)\big)\big|_{z = \alpha \tau}
\lim_{z \to \alpha\tau} \frac{z-\alpha \tau}{\tau - \alpha^{-1} z}\\
&= - \frac{j(\alpha^{-1},\alpha\tau)^{d/2-2}}{2\pi} \phi(\alpha^{-1})
\left(-\frac{1}{j(\alpha,\tau)^2}\right)\\
&=
\frac{\phi(\alpha^{-1})}{2\pi \,j(\alpha,\tau)^{d/2}}.
\endaligned
\]
Thus, there is no pole at $ST^{-1}\tau$, because $\phi(TS)=0$. On the other
hand, there is a pole with residue $1/(2\pi)$ at $z=T^{-1}\tau$, because
$\phi(T) = 1$. We must account for this pole if we wish to cross the line
$\Re(\tau)=1$.

\begin{figure}
\begin{tikzpicture}[scale=3]
\fill[gray!40!white] (-1,0) rectangle (1,2);
\fill[gray!40!white] (1,0) arc (0:180:1);
\fill[gray!40!white] (1,0)--(1,1) arc (90:180:1);
\fill[gray!40!white] (0,0) arc (0:90:1)--(-1,0);
\fill[gray!40!white] (-1,0) rectangle (0,2);
\fill[gray!40!white] (-1,0) rectangle (0,2);
\fill[white]  (1,0) arc (0:180:0.5);
\fill[white]  (0,0) arc (0:180:0.5);
\fill[gray!40!white] (1,1.5) circle (0.25);
\draw[gray] (0,0)--(0,2);
\draw[gray] (-1,0)--(-1,2);
\draw[gray] (1,0)--(1,2);
\draw[gray] (0.5,{sqrt(3/4)}) arc (60:120:1);
\draw[gray] (1,1) arc (90:120:1);
\draw[gray] (-1,1) arc (90:60:1);
\draw[gray] (0,0) arc (0:180:0.5);
\draw[gray] (1,0) arc (0:180:0.5);
\draw[gray] (0.5,0.5) -- (0.5,{sqrt(3/4)});
\draw[gray] (-0.5,0.5) -- (-0.5,{sqrt(3/4)});
\filldraw[black] (1,1.5) circle (0.015);
\draw (1,1.5) node[left=2pt] {$1+it_0$};
\filldraw[black] (0,1.5) circle (0.015);
\draw (0,1.5) node[left=2pt] {$it_0$};
\filldraw[black] (0,1.75) circle (0.015);
\draw (0,1.75) node[left=2pt] {$i(t_0+\varepsilon)$};
\filldraw[black] (0,1.25) circle (0.015);
\draw (0,1.25) node[left=2pt] {$i(t_0-\varepsilon)$};
\draw[thick] (0,1.25) arc (-90:90:0.25);
\draw[thick] (0,0)--(0,1.26);
\draw[thick] (0,1.75)--(0,2);
\end{tikzpicture}
\caption{A contour achieving analytic continuation for $\tau$ near $1+it_0$.}
\label{fig:contoursemicircle}
\end{figure}
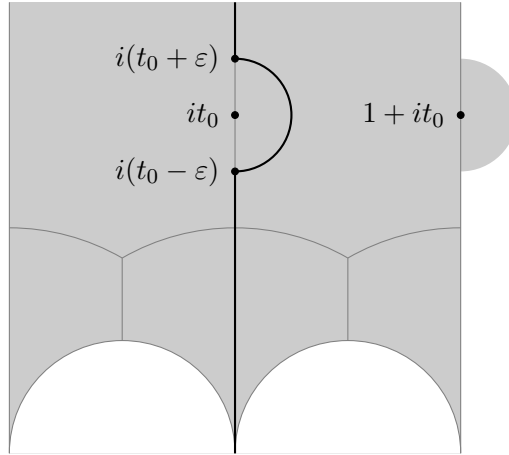

To cross this line, we return to the integrals defining $F_2(\tau,r)$ in
\eqref{decomposeF2} and \eqref{truncatingF2}, which are contour integrals
along pieces of the imaginary axis in the variable $z=it$.  Consider a point
$it_0$ on one of these contours; we wish to continue $\tau \mapsto
F_2(\tau,r)$ from $\Re(\tau)$ slightly less than $1$ to a neighborhood of
$1+it_0$. As shown in Figure~\ref{fig:contoursemicircle}, to do so we can
shift the integration in $z$ from the segment from $i(t_0-\varepsilon)$ to
$i(t_0+\varepsilon)$ to the semicircle $z=it_0+ e^{i\theta}\varepsilon$ with
$-\pi/2\le \theta\le \pi/2$, where the radius $\varepsilon=\varepsilon(t_0)$
is taken to be small enough that this semicircle remains inside $\mathcal D$;
it is at this point that we may need two different choices of $p$ (e.g.,
$p=1$ and $p=2$), so that we can ensure that $|p-t_0|>\varepsilon$ and thus
$ip$ is either above or below this semicircle. The contours having been moved
out of the way of the poles of the integrand, these integrals now give an
expression for $F_2(\tau,r)$ that is holomorphic for $T^{-1}\tau$ slightly to
the left of the contour, in particular, for $\tau$ in a ball of radius
$\varepsilon/2$ around $1+it_0$. We now claim that $F_2(\tau,r)$ is a
Schwartz function of $r$ whose Schwartz seminorms are bounded in the region
$\{\tau \in \C:|\tau-1-it_0|\le \varepsilon/2\}$.  Indeed, the contribution
of the integral from the undeformed contour along the imaginary axis retains
this property, just as it did for $\tau\in\mathcal D$ in the comments
following \eqref{eq:moderateGaussian}. Meanwhile, $\K(\tau,z)$ is continuous
and hence bounded in terms of $t_0$ for such $\tau$ and for $z$ on the
deformed semicircle, which establishes the claimed seminorm bound, just as
right after \eqref{eq:moderateGaussian}.

Having shown the analytic continuation of $F_2(\tau,r)$ past $\Re(\tau)=1$ to
an open subset of $\mathcal U$, we next claim that
\begin{equation}\label{matchF2sharp}
F_2(\tau,r) = F_2^\sharp(\tau,r) + 4 e^{\pi i r^2(\tau-1)}\sin\mathopen{}\big(\pi r^2/2\big)^2\mathclose{}
\end{equation}
on this common domain of definition.  For sufficiently large $r$, this
identity follows from moving the deformed semicircles back into place and
using $\operatorname{Res}_{z = \tau-1}\K(\tau,z) = 1/(2\pi)$, and hence it
holds for all $r\in \R$ by analytic continuation.

Arguing similarly, one continues $\tau\mapsto F_2(\tau,r)$ across
$\Re(\tau)=-1$ to the reflected region $\mathcal U'=\{\tau \in \C
:-\bar{\tau}\in \mathcal U\}$, as well as across the bottom semicircles
$|\tau\pm 1/2|=1/2$ to $TS\mathcal U$ and $T^{-1}S\mathcal U'=\{\tau \in \C :
-\bar{\tau}\in TS\mathcal U\}$. In particular, the integral
\eqref{FfromKlaplace2} extends holomorphically for large $r$ to these last
two domains, because part~(3) of Theorem~\ref{thm: K} shows there is no pole
on those bottom semicircles.  (Specifically, the residue of $\K(\tau,z)$ at
$\tau=\gamma z$ is proportional to $\phi(\gamma)$, and $\phi(STS) = \phi(ST)
= \phi\big(ST^{-1}\big) = \phi\big(ST^{-1}S\big)=0$; note also that
$\phi(TS)=\phi\big(T^{-1}S\big)=0$, so there is no need to take care with
inverses.)

So far, we have seen how to analytically continue $\tau \mapsto F_2(\tau,r)$
to all of $\mathcal{D}_+ := \overline{\mathcal{D}} \cup \mathcal{U} \cup
\mathcal{U}' \cup TS\mathcal{U} \cup T^{-1}S\mathcal{U}'$, and the
boundedness of the Schwartz seminorms holds for the same reason as above. All
that remains is to prove the recurrences \eqref{F2restated1.6to1.8}. We begin
with the first equation in \eqref{F2restated1.6to1.8}, namely
\[
F_2(\tau,r) \mid^\tau_{d/2} (T-I)^2   = 4 e^{\pi i r^2(\tau+1)}\sin\mathopen{}\big(\pi r^2/2\big)^2\mathclose{}
\]
whenever $\tau, \tau+1, \tau+2 \in \mathcal{D}_+$.  The set of such $\tau$ is
connected (it is the interior of the closure of $\mathcal{U}' \cup T^{-2}
\mathcal{U}$), and so by analyticity it suffices to prove this identity when
$r$ is sufficiently large and $\tau+2 \in \mathcal{U}$, in which case
$\tau,\tau+1 \in \mathcal{D}$. Then
\[
F_2(\tau+2,r) = F_2^\sharp(\tau+2,r) + 4 e^{\pi i r^2(\tau+1)}\sin\mathopen{}\big(\pi r^2/2\big)^2\mathclose{}
\]
by \eqref{matchF2sharp}, and therefore
\begin{align*}
F_2(\tau,r) \mid^\tau_{d/2} (T-I)^2&= 4 e^{\pi i r^2(\tau+1)}\sin\mathopen{}\big(\pi r^2/2\big)^2\mathclose{}\\
& \quad \phantom{}
+ F_2^\sharp(\tau,r) \mid^\tau_{d/2} T^2  -2  F_2(\tau,r) \mid^\tau_{d/2} T  + F_2(\tau,r)
\\
&= 4 e^{\pi i r^2(\tau+1)}\sin\mathopen{}\big(\pi r^2/2\big)^2\mathclose{}\\
& \quad \phantom{} +
4 \sin\mathopen{}\big(\pi r^2/2\big)^2\mathclose{}\,\int_0^\infty \big(\K(\tau,it) \mid^\tau_{d/2} (T-I)^2\big)\,e^{-\pi r^2 t}\,dt.
\end{align*}
Because $\K(\tau,it)
\mid^\tau_{d/2} (T-I)^2 = 0$, we obtain
\[
F_2(\tau,r) \mid^\tau_{d/2} (T-I)^2 = 4 e^{\pi i r^2(\tau+1)}\sin\mathopen{}\big(\pi r^2/2\big)^2\mathclose{},
\]
which is the first identity in \eqref{F2restated1.6to1.8}.  The second
identity states that
\[
F_2(\tau,r) \mid^\tau_{d/2} S(T-I)^2  = 0,
\]
and it is proved almost exactly the same way.  Because $\mathcal{D}_+$ is
invariant under $S$, the set of $\tau$ such that $S\tau, ST\tau, ST^2\tau \in
\mathcal{D}_+$ is again connected (it is the same as the set of $\tau$ such
that $\tau, T\tau, T^2\tau \in \mathcal{D}_+$). We can assume $ST^2 \tau \in
T^{-1}S\mathcal{U}' = S \mathcal{U}$ and $r$ is sufficiently large. Then $ST
\tau, S\tau \in \mathcal{D}$, and each of the three terms in $F_2(\tau,r)
\mid^\tau_{d/2} S(T-I)^2$ can be computed using the integral
\eqref{FfromKlaplace2} when $r$ is large enough. Thus, the second identity
follows from $\K(\tau,it) \mid^\tau_{d/2} S(T-I)^2 = 0$.
\end{proof}

We conclude this subsection with some implications of
Proposition~\ref{prop:extendF2} and both parts of
Proposition~\ref{prop:propagationandextension}.

\begin{corollary}\label{corin4.1}
The function $\tau\mapsto F_2(\tau,r)$ extends to a holomorphic function on
$\Hyp$ satisfying the identities
\begin{equation}\label{F2withslash}
F_2|^\tau_{d/2}(T-I)^2 = -e^{\pi i r^2\tau}|^\tau_{d/2}(T-I)^2 \quad \text{and} \quad
F_2|^\tau_{d/2}S(T-I)^2=0,
\end{equation}
and consequently $F=F_1+F_2$ extends to a holomorphic function on $\Hyp$
satisfying \eqref{rerestated1.6to1.8}; in particular, conditions~(1) and~(5)
of Theorem~\ref{thm:FEimpliesIF} hold with $\widetilde F=(e^{\pi i \tau |x|^2
}-F)|^\tau_{d/2}S$.  Furthermore, condition~(3) of
Theorem~\ref{thm:FEimpliesIF} holds if and only if $\max_{r\in \R}\big|r^k
\frac{d^\ell}{dr^\ell}F_2(\tau,r)\big|$ has moderate growth on
$\overline{\mathcal D}$.
\end{corollary}

Of course moderate growth on $\overline{\mathcal D}$ is equivalent to that on
$\mathcal D$, by continuity.

\begin{proof}
The extension and the identities in \eqref{F2withslash} follow immediately
from part~(1) of Proposition~\ref{prop:propagationandextension}, with
$f(\tau)=F_2(\tau,r)$, $h_1=-e^{\pi i r^2\tau}|^\tau_{d/2}(T-I)^2$, and
$h_2=0$. Equation~\eqref{rerestated1.6to1.8} is a restatement of
condition~(5) of Theorem~\ref{thm:FEimpliesIF}, since we have defined
$\widetilde{F}=(e^{\pi i r^2\tau}-F)|^\tau_{d/2}S$. As for condition~(3) of
Theorem~\ref{thm:FEimpliesIF}, the corresponding estimate for $F_2$, namely
that $\max_{r\in \R}\big|r^k \frac{d^\ell}{dr^\ell}F_2(\tau,r)\big|$ has
moderate growth on $\Hyp$, implies those for $F(\tau,r)=e^{\pi i
r^2\tau}+F_2(\tau,r)$ and $\widetilde{F}=-F_2|^\tau_{d/2}S$.  (Here we have
used that $e^{\pi i r^2\tau}$ and all its derivatives with respect to $r$
have moderate growth in $\tau$, uniformly in $r$, which follows from
\eqref{eq:moderateGaussian}.) To reduce the moderate growth to
$\overline{\mathcal D}\subseteq \Hyp$, we can apply part~(2) of
Proposition~\ref{prop:propagationandextension} with $f(\tau)=r^k
\frac{d^\ell}{dr^\ell} F_2(\tau,r)$, $h_1=-r^k \frac{d^\ell}{dr^\ell} e^{\pi
i r^2\tau}|^\tau_{d/2}(T-I)^2$, and $h_2=0$. For fixed $r\in \R$ the bound
\eqref{extendedmoderategrowth} then shows moderate growth in $\tau$ with
constants coming from the Schwartz seminorms of $F_2(\tau,r)$ and $h_1$, and
the final statement follows by maximizing over $r\in \R$.
\end{proof}

\subsection{Proof of Theorem~\ref{theorem:interpolation}}\label{sec:sub:verifypart3}

The interpolation formula \eqref{eqn:IF} will follow from
Theorem~\ref{thm:FEimpliesIF} once we verify all the latter's hypotheses. The
radial hypothesis (2) holds by construction, and hypotheses~(1) and~(5) were
just demonstrated in Corollary~\ref{corin4.1}, which also reduced
hypothesis~(3) to checking the moderate growth of $\max_{r\in \R}\big|r^k
\frac{d^\ell}{dr^\ell}F_2(\tau,r)\big|$ for $\tau\in\overline{\mathcal D}$.
Since the seminorm boundedness assertion in Proposition~\ref{prop:extendF2}
gives the boundedness of $\max_{r\in \R}\big|r^k
\frac{d^\ell}{dr^\ell}F_2(\tau,r)\big|$ for $\tau$ in any fixed compact
subset of $\overline{\mathcal D}$, it further suffices to verify the moderate
growth for $\tau$ lying in a neighborhood in $\overline{\mathcal D}$ of one
of its cusps. Thus to finish the proof of Theorem~\ref{theorem:interpolation}
we will show that
\begin{equation}\label{4.2toshow1}
\max_{r\in \R}\left|r^k \frac{d^\ell}{dr^\ell}F_2(\tau,r)\right| \ \text{has moderate growth for~}\tau\text{~near cusps of~} \overline{\mathcal D},
\end{equation}
and that there is an absolute constant $A>0$ such that
\begin{equation}\label{4.2toshow2}
F^{(8)}(\tau,x), \widetilde{F}^{(8)}(\tau,x) = O\big(|\tau|^A e^{-\pi \Im(\tau)}\big)
\end{equation}
and
\begin{equation}\label{4.2toshow3}
F^{(24)}(\tau,x), \widetilde{F}^{(24)}(\tau,x) = O\big(|\tau|^A e^{-3\pi \Im(\tau)}\big)
\end{equation}
as $\Im(\tau)\rightarrow\infty$ with $-1\le \Re(\tau)\le 1$. Indeed,
\eqref{4.2toshow2} and \eqref{4.2toshow3} are stronger than hypothesis~(4)
and the concluding statement in Theorem~\ref{thm:FEimpliesIF}, which allows
us to deduce $n_0=2$ in $d=24$ dimensions.

The fundamental domain $\mathcal D$ has four cusps, namely $-1$, $0$, $1$,
and $\infty$ (see Figure~\ref{fig:domainD}). Neighborhoods of these cusps are
respectively parametrized by the following elements of $\slz$ acting on
$\tau$ with large imaginary part in the following strips: $T^{-1}S$ applied
to $-1\le \Re(\tau)\le 0$; $S$ applied to $-1\le \Re(\tau)\le 1$; $TS$
applied to $0\le \Re(\tau)\le 1$; and $I$ applied to $-1\le \Re(\tau)\le 1$.
Since factors of automorphy do not affect moderate growth by
\eqref{jvsnorms}, assertion \eqref{4.2toshow1} is equivalent to the moderate
growth of $\max_{r\in \R}\left|r^k
\frac{d^\ell}{dr^\ell}(F_2|^\tau_{d/2}\gamma)(\tau,r)\right|$ for each of
these group elements $\gamma$ and $\tau$ with large imaginary part in its
corresponding strip.  Alternatively, since the corresponding seminorm growth
assertion holds for $F_1(\tau,r)=e^{\pi i \tau r^2}$, for each of these
$\gamma$ it is equivalent to replace $F_2$ by $F$ in this last assertion. For
the rest of this subsection we thus assume
\[-1\le \Re(\tau)\le 1 \text{ and
$\Im(\tau)$ is large}.
\]

Our key tool is a contour shift very similar to the proof of
\cite[Proposition~2]{V}. For $|r|$ and $\Im(\tau)$ both sufficiently large,
we write the factor $\sin\mathopen{}\big(\pi r^2/2\big)^2\mathclose{}$ in
terms of complex exponentials, incorporate them into the integrand from
\eqref{FfromKlaplace2}, and set $z=it$ to obtain
\[
\aligned
F_2(\tau,r) &= i \int_{1}^{1+i\infty} \K(\tau,z-1) \,e^{\pi i r^2 z} \, dz + i \int_{-1}^{-1+i\infty} \K(\tau,z+1) \,e^{\pi i r^2 z} \, dz\\
&\quad\phantom{}
-2i \int_0^{i\infty} \K(\tau,z) \,e^{\pi i r^2 z} \, dz,
\endaligned
\]
where all the contours are vertical rays.  We now claim that a shift to the
contours $\alpha_{-1}$, $\alpha_0$, $\alpha_1$, and $\alpha_\infty$ shown in
Figure~\ref{fig:alphacontours} yields
\begin{equation}\label{contourshift1}
\aligned
  F(\tau,r)&=e^{\pi i r^2\tau}+F_2(\tau,r)\\
&= i \int_{\alpha_1}\K(\tau,z-1)\,e^{\pi i r^2 z}\,dz
+ i \int_{\alpha_{-1}}\K(\tau,z+1)\,e^{\pi i r^2 z}\,dz
\\
& \quad \phantom{}-2i \int_{\alpha_{0}}\K(\tau,z)\,e^{\pi i r^2 z}\,dz\\
& \quad \phantom{} +i \int_{\alpha_{\infty}}\big(\K|^z_{2-d/2}(T+T^{-1}-2I)\big)(\tau,z)\, e^{\pi i r^2 z}\,dz.
\endaligned
\end{equation}
Specifically, we must shift the contours involving $\K(\tau,z-1)$ and
$\K(\tau,z+1)$, and the only poles that can intervene in this contour shift
are the poles of $z \mapsto \K(\tau,z\pm 1)$ at $z=\tau$, with residues
\[
\operatorname{Res}_{z = \tau} \K(\tau,z \pm 1) = \mp\frac{1}{2\pi}
\]
by part~(3) of Theorem~\ref{thm: K}.  When $\Re(\tau)>0$, we obtain a
contribution from $\K(\tau,z-1)$ at $z=\tau$, while $\K(\tau,z+1)$ has no
pole at $z=\tau-1$; when $\Re(\tau)<0$, we obtain a contribution from
$\K(\tau,z+1)$ at $z=\tau$, while $\K(\tau,z-1)$ has no pole at $z=\tau+1$;
finally, when $\Re(\tau)=0$ both contour shifts contribute half as much
(alternatively, this case follows by continuity). In each case, the residues
cancel the $e^{\pi i r^2\tau}$ term from $F(\tau,r)=e^{\pi i
r^2\tau}+F_2(\tau,r)$ and we obtain \eqref{contourshift1}.

\begin{figure}
\begin{tikzpicture}[scale=2.8]
\begin{scope}[xshift=-35]
\draw[gray] (0,0)--(0,2);
\draw[gray] (-1,0)--(-1,2);
\draw[gray] (1,0)--(1,2);
\draw[gray] (0.5,{sqrt(3/4)}) arc (60:120:1);
\draw[gray] (1,1) arc (90:120:1);
\draw[gray] (-1,1) arc (90:60:1);
\draw[gray] (0,0) arc (0:180:0.5);
\draw[gray] (1,0) arc (0:180:0.5);
\draw[gray] (0.5,0.5) -- (0.5,{sqrt(3/4)});
\draw[gray] (-0.5,0.5) -- (-0.5,{sqrt(3/4)});

\draw[black] (-0.86,0.5) node {$\alpha_{-1}$} ; \draw[black,thick]
(-1,0)--(-1,1); \draw[->,black,thick] (-1,0)--(-1,0.7); \draw[black,thick]
(-1,1) arc(90:60:1); \draw[black,thick] (-0.5,{sqrt(3/4)}) arc(120:90:1);
\draw[black,thick] (0,1)--(0,1.2); \draw[->,thick] (-1,1) arc(90:73:1);
\draw[->,thick] (-0.5,{sqrt(3/4)}) arc(120:103:1);

\draw[black] (0.09,0.5) node {$\alpha_{0}$} ; \draw[black,thick]
(0,0)--(0,1.2); \draw[->,black,thick] (0,0)--(0,0.7); \draw[->,black,thick]
(0,0)--(0,1.12);

\draw[black] (1.1,0.5) node {$\alpha_{1}$} ; \draw[black,thick] (1,0)--(1,1);
\draw[->,black,thick] (1,0)--(1,0.7); \draw[black,thick] (1,1)
arc(90:120:1); \draw[black,thick] (0.5,{sqrt(3/4)}) arc(60:90:1);
\draw[black,thick] (0,1)--(0,1.2); \draw[->,thick] (0.5,{sqrt(3/4)})
arc(60:77:1); \draw[->,thick] (1,1) arc(90:107:1);

\draw[black] (0.12,1.5) node {$\alpha_{\infty}$} ; \draw[black,thick]
(0,1.2)--(0,2); \draw[->,black,thick] (0,1.2)--(0,1.7);

\filldraw[black] (0,0) circle ({9/560}); \draw  (0,0) node[below=2pt] {$0$};
\filldraw[black] (-1,0) circle ({9/560}); \draw  (-1,0)node[below=2pt]
{$-1$}; \filldraw[black] (1,0) circle ({9/560}); \draw (1,0) node[below=2pt]
{$1$} ; \filldraw[black] (0,1.2) circle ({9/560}); \draw  (0,1.2)
node[right=1pt] {$6i/5$};
\end{scope}

\begin{scope}[xshift=35]
\draw[gray] (0,0)--(0,2); \draw[gray] (-1,0)--(-1,2); \draw[gray]
(1,0)--(1,2); \draw[gray] (0.5,{sqrt(3/4)}) arc (60:120:1); \draw[gray] (1,1) arc
(90:120:1); \draw[gray] (-1,1) arc (90:60:1); \draw[gray] (0,0) arc
(0:180:0.5); \draw[gray] (1,0) arc (0:180:0.5); \draw[gray] (0.5,0.5) --
(0.5,{sqrt(3/4)}); \draw[gray] (-0.5,0.5) -- (-0.5,{sqrt(3/4)});

\draw[black] (-0.74,0.69) node {$ST\alpha_{-1}$} ; \draw[black,thick]
(-0.5,0.5)--(-0.5,{sqrt(3/4)}); \draw[->,black,thick] (-0.5,{sqrt(3/4)})--(-0.5,0.7);
\draw[black,thick] (-0.5,0.5) arc(90:79.61:0.5); \draw[black,thick]
(-0.5,{sqrt(3/4)}) arc(120:90:1); \draw[black,thick] (0,1)--(0,2);
\draw[->,thick] (0,2)--(0,1.25);
\draw[->,thick] (0,1) arc(90:107:1);

\draw[black] (0.76,0.69) node {$ST^{-1}\alpha_{1}$} ; \draw[black,thick]
(0.5,0.5)--(0.5,{sqrt(3/4)}); \draw[->,black,thick] (0.5,{sqrt(3/4)})--(0.5,0.7);
\draw[black,thick] (0.5,0.5) arc(90:101.39:0.5); \draw[black,thick]
(0.5,{sqrt(3/4)}) arc(60:90:1); \draw[black,thick] (0,1)--(0,2);
\draw[->,thick] (0,1) arc(90:73:1);

\filldraw[black] (0,0) circle ({9/560}); \draw  (0,0) node[below=2pt] {$0$};
\filldraw[black] (-1,0) circle ({9/560}); \draw  (-1,0)node[below=2pt]
{$-1$}; \filldraw[black] (1,0) circle ({9/560}); \draw (1,0) node[below=2pt]
{$1$} ; \filldraw[black] ({25/61},{30/61}) circle ({9/560}); \draw
({25/61},{30/61}) node[below=2pt] {$\frac{-1}{-1+6i/5}$} ; \filldraw[black]
({-25/61},{30/61}) circle ({9/560}); \draw ({-25/61},{30/61}) node[below=2pt]
{$\frac{-1}{1+6i/5}$} ;
\end{scope}
\end{tikzpicture}
\caption{The contours $\alpha_{-1}$, $\alpha_0$, $\alpha_1$, $\alpha_\infty$,
$ST\alpha_{-1}$, and $ST^{-1}\alpha_{1}$ superimposed on the fundamental
domain $\mathcal D$ shown in Figure~\ref{fig:domainD}. All
four contours on the left share a common terminus at $6i/5$, with
$\alpha_{-1}$, $\alpha_0$, and $\alpha_1$ ending there and $\alpha_\infty$
starting there, while those on the right go from $i \infty$ to
$-1/(\pm1+6i/5)$. \label{fig:alphacontours}}
\end{figure}
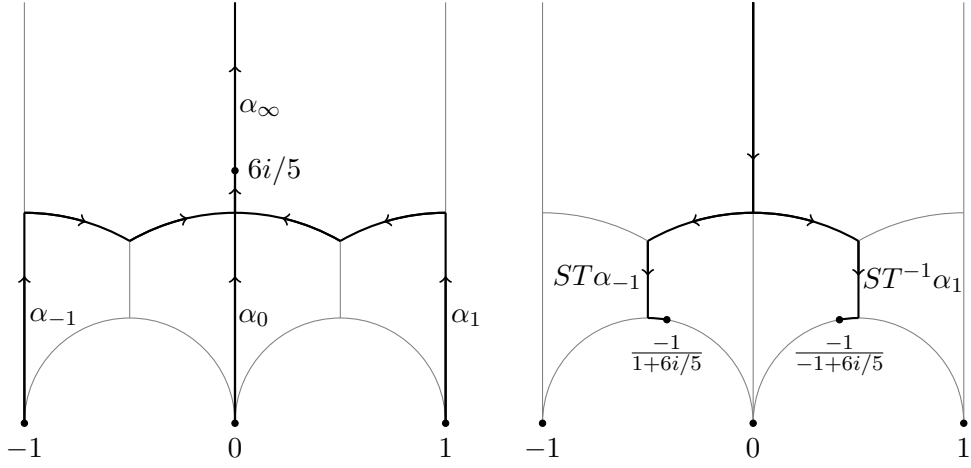

We can rewrite the last integrand in \eqref{contourshift1} using
\[
\aligned
  \K|^z_{2-d/2}(T+T^{-1}-2I) & =
  \frac{1}{2} \K_{+}|^z_{2-d/2}(T+T^{-1}-2I)\\
  & \quad \phantom{}+\frac{1}{2} \K_{-}|^z_{2-d/2}(T+T^{-1}-2I) \\
& =   \K_{+}|^z_{2-d/2}S-\K_{-}|^z_{2-d/2}S \\ &= 2\widehat{\mathcal K}|^z_{2-d/2}S\\
 &= -2\K|^\tau_{d/2}S|^z_{2-d/2}S,
\endaligned
\]
by Proposition~\ref{prop:Kpmslashtilde}, the definition of
$\widetilde{\mathcal I}_\pm$ from \eqref{freeideals}, and \eqref{Kpmdef}.
Thus,
\begin{equation}\label{contourshift1b}
\aligned
  F(\tau,r)
&= i \int_{\alpha_1}\K(\tau,z-1)\,e^{\pi i r^2 z}\,dz
+ i \int_{\alpha_{-1}}\K(\tau,z+1)\,e^{\pi i r^2 z}\,dz
\\
& \quad \phantom{}-2i \int_{\alpha_{0}}\K(\tau,z)\,e^{\pi i r^2 z}\,dz\\
& \quad \phantom{}-2i \int_{\alpha_{\infty}}\big(\K|^\tau_{d/2}S|^z_{2-d/2}S\big)(\tau,z)\,e^{\pi i r^2 z}\,dz.
\endaligned
\end{equation}

If we instead start with formula \eqref{FtildefromKhatlaplace} for
$\widetilde{F}=-F_2|^\tau_{d/2}S$, or with \eqref{FfromKlaplace2} for
$F_2|^\tau_{d/2}TS$ or $F_2|^\tau_{d/2}T^{-1}S$, then no poles are
encountered in these contour shifts, again because of the residue formula
from part~(3) of Theorem~\ref{thm: K}.  Thus after a change of variables
\eqref{contourshift1b} generalizes to
\begin{equation}\label{countorshift3gen}
\aligned
\Phi(\tau,r) &= i \int_{ST^{-1}\alpha_1}\big(\K|^\tau_{d/2}\gamma |^z_{2-d/2} S\big)(\tau,z)\,e^{\pi i r^2 (1-1/z)}\,\frac{dz}{z^{d/2}}  \\
& \quad \phantom{}+i \int_{ST\alpha_{-1}}\big(\K|^\tau_{d/2}\gamma |^z_{2-d/2} S\big)(\tau,z)\,e^{\pi i r^2 (-1-1/z)}\,\frac{dz}{z^{d/2}}\\
& \quad \phantom{}-2i \int_{S\alpha_0}\big(\K|^\tau_{d/2}\gamma |^z_{2-d/2} S\big)(\tau,z)\,e^{\pi i r^2 (-1/z)}\,\frac{dz}{z^{d/2}}\\
& \quad \phantom{}-2i \int_{\alpha_\infty}\big(\K|^\tau_{d/2}S\gamma |^z_{2-d/2} S\big)(\tau,z)\,
e^{\pi i r^2 z}\,
dz,
\endaligned
\end{equation}
where $(\Phi,\gamma)$ is one of the four pairs
\begin{equation}\label{dapairs}
\begin{array}{lll}
(F,\,I), && (-\widetilde{F},\,S),\\
(F_2|^\tau_{d/2}TS,\,TS), && (F_2|^\tau_{d/2}T^{-1}S,\,T^{-1}S),
\end{array}
\end{equation}
and we assume $-1\le \Re(\tau)\le 1$ and $\Im(\tau)$ is large.  Note that
these pairs are exactly the cases we must analyze to treat the four cusps of
$\mathcal{D}$. Since
\begin{equation}\label{slashKagain}
\K=\frac{1}{2}(\K_+ + \K_{-}) \quad \text{and} \quad \widehat{\K}=-\K|^\tau_{d/2}S=\frac{1}{2}(\K_+ - \K_{-}),
\end{equation}
estimates on the first two kernels in \eqref{dapairs} are provided in
\eqref{quotientbounds2new}, and  estimates on the last two kernels  are
provided in \eqref{quotientbounds1new}.

Though its derivation initially assumed $|r|$ large, formula
\eqref{countorshift3gen} actually gives an analytic continuation to all $r\in
\R$, as can be seen from \eqref{quotientbounds1new} and
\eqref{quotientbounds2new}, which show that each of its four integrals is
absolutely convergent and can be differentiated under the integral sign.  In
particular, letting
\[
e_{k,\ell}(z,r)=r^k \frac{d^\ell}{dr^\ell} e^{\pi i z r^2},
\]
we have
\begin{equation}\label{contourshift4}
\aligned
 r^k\frac{d^\ell}{dr^\ell}\Phi(\tau,r) &= i \int_{ST^{-1}\alpha_1}\big(\K|^\tau_{d/2}\gamma |^z_{2-d/2} S\big)(\tau,z)e_{k,\ell}(1-1/z,r)\,\frac{dz}{z^{d/2}}  \\
& \quad \phantom{}+i \int_{ST\alpha_{-1}}\big(\K|^\tau_{d/2}\gamma |^z_{2-d/2} S\big)(\tau,z)e_{k,\ell}(-1-1/z,r)\,\frac{dz}{z^{d/2}}\\
& \quad \phantom{}-2i \int_{S\alpha_0}\big(\K|^\tau_{d/2}\gamma |^z_{2-d/2} S\big)(\tau,z)e_{k,\ell}(-1/z,r)\,\frac{dz}{z^{d/2}}\\
& \quad \phantom{}-2i \int_{\alpha_\infty}\big(\K|^\tau_{d/2}S\gamma |^z_{2-d/2} S\big)(\tau,z)
e_{k,\ell}(z,r)\,
dz.
\endaligned
\end{equation}
Claims \eqref{4.2toshow1}--\eqref{4.2toshow3} are now reduced to bounding the
four integrals in \eqref{contourshift4}. Each of these four contours in
\eqref{contourshift4} lies above $\Im(z)=1/3$, and consists of a
semi-infinite ray along the imaginary axis (possibly) together with a compact
curve in $\Hyp$ (see Figure~\ref{fig:alphacontours}).  Thus for $\Im(\tau)$
sufficiently large the only $\slz$-translates of $\tau$ near any of these
contours are $\tau-1$, $\tau$, and $\tau+1$, which can only possibly
contribute poles for $z$ on the imaginary axis portion of the contour.

We shift the contours (if necessary) to keep $z$ at distance at least $1/4$
from any of these potential poles, in particular by taking the integration
over $C_1\cup C_2\cup C_3$, where $C_1$ runs along the compact curve and the
imaginary axis between $i$ and $i(\Im(\tau)-1/4)$, $C_2$ is a contour between
$i(\Im(\tau)-1/4)$ and $i(\Im(\tau)+1/4)$ keeping distance at least $1/4$
from any $\slz$-translate of $\tau$, and $C_3$ runs along the imaginary axis
between $i(\Im(\tau)+1/4)$ and $\infty$ (see Figure~\ref{fig:semicircle}).

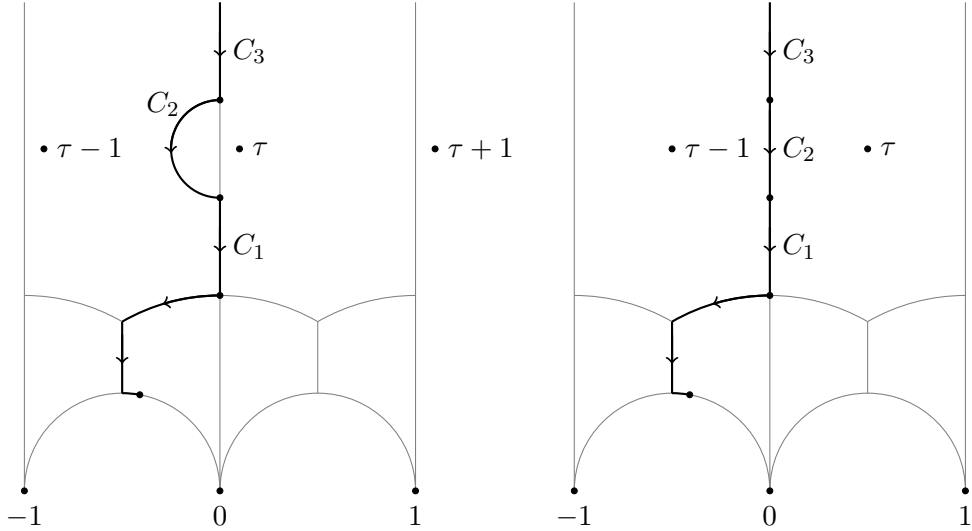
\begin{figure}
\begin{tikzpicture}[scale=2.6]
\begin{scope}[xshift=-40]
\draw[gray] (0,0)--(0,2.5);
\draw[gray] (-1,0)--(-1,2.5);
\draw[gray] (1,0)--(1,2.5);
\draw[gray] (0.5,{sqrt(3/4)}) arc (60:120:1);
\draw[gray] (1,1) arc (90:120:1);
\draw[gray] (-1,1) arc (90:60:1);
\draw[gray] (0,0) arc (0:180:0.5);
\draw[gray] (1,0) arc (0:180:0.5);
\draw[gray] (0.5,0.5) -- (0.5,{sqrt(3/4)});
\draw[gray] (-0.5,0.5) -- (-0.5,{sqrt(3/4)});

\filldraw[black] (0.1,1.75) circle ({9/520}); \draw  (0.1,1.75)
node[right=1pt] {$\tau$};

\filldraw[black] (1.1,1.75) circle ({9/520}); \draw  (1.1,1.75)
node[right=1pt] {$\tau+1$};

\filldraw[black] (-0.9,1.75) circle ({9/520}); \draw  (-0.9,1.75)
node[right=1pt] {$\tau-1$};

\draw[black] (0.15,1.25) node {$C_1$} ; \draw[black,thick] (0,1)--(0,1.5);
\draw[black,thick] (-0.5,0.5)--(-0.5,{sqrt(3/4)}); \draw[->,black,thick]
(-0.5,0.8)--(-0.5,0.65); \draw[black,thick] (-0.5,0.5) arc(90:79.61:0.5);
\draw[black,thick] (-0.5,{sqrt(3/4)}) arc(120:90:1); \draw[->,black,thick]
(0,1.35)--(0,1.22);

\draw[->,thick] (0,1) arc(90:107:1);

\draw[black] (-0.3,2) node {$C_2$} ; \draw[black,thick] (0,1.5) arc
(270:90:0.25); \draw[black,thick,->] (0,2) arc (90:185:0.25);

\draw[black] (0.15,2.25) node {$C_3$} ; \draw[black,thick] (0,2)--(0,2.5);
\draw[->,black,thick] (0,2.35)--(0,2.22);

\filldraw[black] (0,0) circle ({9/520}); \draw  (0,0) node[below=2pt] {$0$};
\filldraw[black] (-1,0) circle ({9/520}); \draw  (-1,0)node[below=2pt]
{$-1$}; \filldraw[black] (1,0) circle ({9/520}); \draw (1,0) node[below=2pt]
{$1$} ;

\filldraw[black] ({-25/61},{30/61}) circle ({9/520}); \filldraw[black] (0,1)
circle ({9/520}); \filldraw[black] (0,1.5) circle ({9/520}); \filldraw[black]
(0,2) circle ({9/520});
\end{scope}

\begin{scope}[xshift=40]
\draw[gray] (0,0)--(0,2.5); \draw[gray] (-1,0)--(-1,2.5); \draw[gray]
(1,0)--(1,2.5); \draw[gray] (0.5,{sqrt(3/4)}) arc (60:120:1); \draw[gray] (1,1)
arc (90:120:1); \draw[gray] (-1,1) arc (90:60:1); \draw[gray] (0,0) arc
(0:180:0.5); \draw[gray] (1,0) arc (0:180:0.5); \draw[gray] (0.5,0.5) --
(0.5,{sqrt(3/4)}); \draw[gray] (-0.5,0.5) -- (-0.5,{sqrt(3/4)});

\filldraw[black] (0.5,1.75) circle ({9/520}); \draw  (0.5,1.75)
node[right=1pt] {$\tau$};

\filldraw[black] (-0.5,1.75) circle ({9/520}); \draw  (-0.5,1.75)
node[right=1pt] {$\tau-1$};

\draw[black] (0.15,1.25) node {$C_1$} ; \draw[black,thick] (0,1)--(0,1.5);
\draw[black,thick] (-0.5,0.5)--(-0.5,{sqrt(3/4)}); \draw[->,black,thick]
(-0.5,0.8)--(-0.5,0.65); \draw[black,thick] (-0.5,0.5) arc(90:79.61:0.5);
\draw[black,thick] (-0.5,{sqrt(3/4)}) arc(120:90:1); \draw[->,black,thick]
(0,1.35)--(0,1.22);

\draw[black] (0.15,1.75) node {$C_2$} ; \draw[black,thick] (0,1.5) -- (0,2);
\draw[->,black,thick] (0,1.95)--(0,1.72);

\draw[black] (0.15,2.25) node {$C_3$} ; \draw[black,thick] (0,2)--(0,2.5);
\draw[->,black,thick] (0,2.35)--(0,2.22);

\draw[->,thick] (0,1) arc(90:107:1);

\filldraw[black] (0,0) circle ({9/520}); \draw  (0,0) node[below=2pt] {$0$};
\filldraw[black] (-1,0) circle ({9/520}); \draw  (-1,0)node[below=2pt]
{$-1$}; \filldraw[black] (1,0) circle ({9/520}); \draw (1,0) node[below=2pt]
{$1$} ;

\filldraw[black] ({-25/61},{30/61}) circle ({9/520}); \filldraw[black] (0,1)
circle ({9/520}); \filldraw[black] (0,1.5) circle ({9/520}); \filldraw[black]
(0,2) circle ({9/520});
\end{scope}

\end{tikzpicture}
\caption{The contours $C_1$, $C_2$, and $C_3$  (shown here for
$ST\alpha_{-1}$) keep distance at least $1/4$ from any $\Z$-translate of
$\tau$.\label{fig:semicircle}}
\end{figure}

Next we use \eqref{slashKagain} with the estimates \eqref{quotientbounds1new}
and \eqref{quotientbounds2new} from Lemma~\ref{lem:boundsonKnew} to bound the
kernel factors by positive linear combinations of terms of the form
\[
\left|\frac{e^{\pi i (B_\tau \tau +B_z z)}\tau^2
z^2}{\Delta(\tau)\Delta(z)(j(\tau)-j(z))}\right|,
\]
with the integers $B_\tau$ and $B_z$ coming from the exponents on the right
sides of those bounds. We claim this last expression is itself bounded by a
constant multiple of
\[
\left|\frac{e^{\pi i (B_\tau \tau +B_z z)}\tau^2 z^2}{e^{2\pi i \tau}-e^{2\pi i z}}\right|
\]
for $\Im(z)\ge 1/3$ and sufficiently large $\Im(\tau)$. Indeed, writing $j$'s
$q$-expansion \eqref{jfunction} as
\[
j(w)=J(e^{2\pi i w})=e^{-2\pi i w}+744+\widetilde{J}(e^{2\pi i n w}),
\]
where $\widetilde{J}(u)=\sum_{n\ge 1}c_j(n)u^n$, we see that
$\widetilde{J}'(u)=O(1)$ as $|u|\rightarrow 0$ and hence
\[
j(\tau)-j(z)=J(e^{2\pi i \tau})-J(e^{2\pi iz})=
e^{-2\pi i \tau}-e^{-2\pi i z}+\int_{e^{2\pi iz}}^{e^{2\pi i\tau}}\widetilde{J}'(u)\,du,
\]
which is $(e^{2\pi i z}-e^{2\pi i \tau})(e^{-2\pi i (\tau+z)}+O(1))$ for
$\Im(\tau),\Im(z)\ge 1/3$.  Furthermore, the product formula for $\Delta$
implies that $|\Delta(z)| \gg |e^{2\pi i z}|$ for $\Im(z) \ge 1/3$, where
$\gg$ indicates inequality up to a positive constant factor. Hence for
sufficiently large $\Im(\tau)$, the $O(1)$ term is less than half the
$e^{-2\pi i (\tau+z)}$ to which it is added, so
$|\Delta(\tau)\Delta(z)(j(\tau)-j(z))|\gg|e^{2\pi i \tau}-e^{2\pi i z}|$, and
the claim follows.

To deal with the functions $e_{k,\ell}$ appearing in \eqref{contourshift4},
we note that for each $k$ and $\ell$, there exists a constant $A$ such that
\[
\max_{r \in \R} \left| r^k \frac{d^\ell}{dr^\ell} e^{\pi i z r^2} \right| \ll |z|^A,
\]
and the same holds if $z$ is replaced with $-1/z$, $1-1/z$, or $-1-1/z$ on
the left side of the inequality (this follows from
\eqref{eq:moderateGaussian} and $\Im(-1/z) = \Im(z)/|z|^2$). We can assume
that $A \ge 0$ because $\Im(z)$ is bounded away from $0$. We have thus
reduced the verification of \eqref{4.2toshow1}--\eqref{4.2toshow3} to
estimates for large $\Im(\tau)$ of the three integrals
\begin{equation}\label{threecontourintegrals}
I_j(\tau) = \int_{C_j}\left|\frac{e^{\pi i (B_\tau \tau +B_z z)}}{e^{2\pi i \tau}-e^{2\pi i z}}\right| |z|^A \,dz
\end{equation}
for $j=1,2,3$, where $A$ is a positive integer depending on $k$ and $\ell$.
In all cases
\[
1\le B_z\le 2 \quad \text{and} \quad B_\tau+B_z \ge 2,
\]
by \eqref{npmconstants}.

For $z\in C_1$, the denominator $|e^{2\pi i \tau}-e^{2\pi i z}|$ in
\eqref{threecontourintegrals} is at least a constant multiple of $|e^{2\pi i
z}|$, and
\[
  |I_1(\tau)| \ll |\tau|^A e^{-\pi B_\tau \Im(\tau)}\left(1+ \int_1^{\Im(\tau)-1/4} e^{ (2-B_z)\pi t}\,dt\right),
\]
which is $O\big(e^{-\pi(B_\tau+B_z-2)\Im(\tau)}|\tau|^{A+1}\big)$; note that
in the displayed formula above, the term $1$ inside the parentheses accounts
for the portion of the contour below imaginary part~$1$. Next, let $z\in
C_2$, which keeps distance at least 1/4 from all integral translates of
$\tau$, but has imaginary part within $1/4$ of $\Im(\tau)$; that is,
\[
z-\tau \in \left\{w\in\C: |\Im(w)|\le \frac 14 \text{ and } |w-n|\ge \frac 14 \text{ for all }
n\in\Z\right\}.
\]
The image of this last region under the map $w\mapsto e^{2\pi i w}$ is
compact but omits $1$, and hence is bounded away from 1.  It follows that the
denominator $|e^{2\pi i \tau}-e^{2\pi i z}|=e^{-2\pi\Im(\tau)}|1-e^{2\pi i
(z-\tau)}|$ is at least some constant multiple of $e^{-2\pi\Im(\tau)}$.
Consequently,
\[
 |I_2(\tau)| \ll e^{-\pi(B_\tau+B_z-2)\Im(\tau)}|\tau|^A.
\]
Finally, for $z\in C_3$ the denominator satisfies $|e^{2\pi i \tau}-e^{2\pi i
z}|\gg |e^{2\pi i \tau}|$, and the fact that $B_z\ge 1$ allows us to show
\[
\aligned
 |I_3(\tau)| & \le e^{-\pi(B_\tau-2)\Im(\tau)}\int_{\Im(\tau)+1/4}^\infty e^{-\pi B_z t}t^A\,dt\\
& \ll e^{-\pi(B_\tau+B_z-2)\Im(\tau)}|\tau|^{A}.
\endaligned
\]
Combined,
$I_1(\tau)+I_2(\tau)+I_3(\tau)=O\big(e^{-\pi(B_\tau+B_z-2)\Im(\tau)}|\tau|^{A+1}\big)$.
Thus \eqref{4.2toshow1} follows, since $B_\tau+B_z=2$ in all cases in
\eqref{quotientbounds1new} (recall \eqref{slashKagain}).  Finally, the
estimates \eqref{quotientbounds2new} imply \eqref{4.2toshow2} (with
$B_\tau+B_z=3$) and \eqref{4.2toshow3} (with $B_\tau+B_z=5$), which completes
the proof of Theorem~\ref{theorem:interpolation}.

\subsection{Proof of Theorem~\ref{theorem:interpolation-isom}}\label{sec:sub:Thm110proof}

The following lemma is a direct consequence of applying
Lemma~\ref{lem:noPsolutions} to the eigenfunctions $g\pm\widetilde{g}$
of $|_{d/2}S$:

\begin{lemma}
Let $(d,n_0)$ be $(8,1)$ or $(24,2)$. If $g,\widetilde{g}\in\mathcal P$
satisfy
\begin{equation}\label{DRinsert}
\aligned
g(\tau+2)-2g(\tau+1)+g(\tau)=0,\\
\widetilde{g}(\tau+2)-2\widetilde{g}(\tau+1)+\widetilde{g}(\tau)=0,\\
g(\tau)+(i/\tau)^{d/2}\widetilde{g}(-1/\tau) = 0,
\endaligned
\end{equation}
and $g(\tau),\widetilde{g}(\tau) = o\big( e^{-2\pi (n_0-1)\Im(\tau)}\big)$ as
$\Im(\tau)\to\infty$ with $-1 \le \Re(\tau) \le 1$, then $g=\widetilde{g}=0$.
\end{lemma}

Recall that the conditions of Theorem~\ref{thm:FEimpliesIF} for $F(\tau,x)$
and $\widetilde{F}(\tau,x)$ were shown to hold in the previous subsection.
Therefore $F(\tau,x)$ and $\widetilde{F}(\tau,x)$ have the form
\eqref{eq2:FFourier} and \eqref{eq2:FtFourier}, and it follows that they and their partial
derivatives of all orders in $x$ are $o\big(e^{-2\pi (n_0-1)\Im(\tau)}\big)$
as $\Im(\tau)\rightarrow\infty$, for any fixed $x$.  Furthermore, for any
fixed $x$, the functions $F(\tau,x)$ and $\widetilde{F}(\tau,x)$ satisfy the
first two equations in \eqref{DRinsert}, and an inhomogeneous variant of the
third, in which the right side is replaced by $e^{\pi i\tau|x|^2}$.  For
$m\ge n_0$ and $|x|=\sqrt{2m}$, the pair of functions
$g(\tau)=F\big(\tau,\sqrt{2m}\big)-e^{2\pi i m\tau}$ and
$\widetilde{g}(\tau)=\widetilde{F}\big(\tau,\sqrt{2m}\big)$ satisfies
\eqref{DRinsert}, since $e^{2\pi i m \tau}$ is periodic.
Thus the lemma shows $g$ and $\widetilde{g}$ are
identically zero, i.e.,
\[
F\big(\tau,\sqrt{2m}\big)=e^{2\pi i m\tau}\quad \text{and} \quad \widetilde{F}\big(\tau,\sqrt{2m}\big)=0.
\]
In terms of the coefficients of the Fourier series expansion
\eqref{eq:FFourier}, we deduce for $m,n\ge n_0$ that
$a_n\big(\sqrt{2m}\big)=\delta_{n,m}$ and $b_n\big(\sqrt{2m}\big)=\widetilde
a_n\big(\sqrt{2m}\big)=\widetilde b_n\big(\sqrt{2m}\big)=0$.

Next consider the radial derivatives (i.e., derivatives with respect to
$r=|x|$) of equations
\eqref{eq:interpolation-identity}--\eqref{eq:2nddiffFt}, which again have
unique solutions for any fixed $x$ by the same logic as above.  Suppose
$n,m\ge n_0$.  At $|x|=\sqrt{2m}$ we similarly deduce that $\frac{\partial
F}{\partial r}\big(\tau,\sqrt{2m}\big)=2\pi i \sqrt{2m} \tau e^{2\pi im\tau}$
and $\frac{\partial \widetilde{F}}{\partial r}\big(\tau,\sqrt{2m}\big)=0$,
from which we obtain $b_n'\big(\sqrt{2m}\big)=\delta_{n,m}$ and
$a'_n\big(\sqrt{2m}\big)=\widetilde
a_n\,\!\!\!'\big(\sqrt{2m}\big)=\widetilde
b_n\,\!\!\!'\big(\sqrt{2m}\big)=0$.

Applying the Fourier transform $\mathcal F_x$ in $x$ and replacing $\tau$
with $-1/\tau$ in equations
\eqref{eq:interpolation-identity}--\eqref{eq:2nddiffFt} shows that if
$(F(\tau,x),\widetilde{F}(\tau,x))$ is a solution to these three equations,
then so is $(\mathcal F_x \widetilde{F}(\tau,x),\mathcal F_x F(\tau,x))$.
Applying the lemma to the difference of these two solutions shows that
$\widetilde{F}(\tau,x)=\mathcal F_x F(\tau,x)$.  In terms of the Fourier
series \eqref{eq:FFourier} and \eqref{eq:FtFourier},
$\widetilde{a}_n=\widehat{a}_n$ and $\widetilde{b}_n=\widehat{b}_n$. This
shows that \eqref{eqn:IF} is in fact inverse to the map in the statement of
Theorem~\ref{theorem:interpolation-isom}.  The image of the latter map is
contained in $\mathcal S(\N)^4$ because of the decay of Schwartz functions,
and the image of the inverse map is a radial Schwartz function because the
radial seminorms of the interpolation basis grow at most polynomially by
Theorem~\ref{thm:FEimpliesIF}. This completes the proof of
Theorem~\ref{theorem:interpolation-isom}.

\subsection{Integral formulas for the interpolation basis} \label{subsec:basisformulas}

We can now prove integral formulas for the interpolation basis, which
generalize the formulas for the sphere packing auxiliary functions from
\cite{V,CKMRV}. These formulas are not needed to prove the interpolation
theorem or universal optimality, but they are of interest in their own right,
and they help clarify the relationship with the sphere packing constructions.

We begin by noting that for $z$ in any fixed compact subset of $\Hyp$,
whenever $\Im(\tau)$ is sufficiently large the kernels can be expanded as
\begin{equation} \label{eq:Kexpansion}
\K(\tau,z) = \sum_{n \ge n_0} \alpha_n(z)\,e^{2\pi i n \tau}+2\pi i \tau\sum_{n \ge
n_0}\sqrt{2n}\, \beta_n(z)\,e^{2\pi i n \tau}
\end{equation}
and
\begin{equation} \label{eq:Khatexpansion}
\widehat{\K}(\tau,z) = \sum_{n \ge n_0} \widetilde{\alpha}_n(z)\,e^{2\pi i n \tau}+2\pi i \tau\sum_{n \ge
n_0}\sqrt{2n}\, \widetilde{\beta}_n(z)\,e^{2\pi i n \tau}
\end{equation}
for some functions $\alpha_n$, $\beta_n$, $\widetilde{\alpha}_n$, and
$\widetilde{\beta}_n$ that depend only on the dimension $d \in \{8,24\}$. The
existence of expansions of these forms is ensured by parts~(2) and~(4) of
Theorem~\ref{thm: K}.  To obtain the explicit expansions, we use the
$q$-series in $\tau$ for the quasimodular forms appearing in the explicit
constructions of $\K$ and $\widehat{\K}$, as well as the analogous expansions
of $\mathcal{L}$ and $\mathcal{L}_S$, and we write
\[
\frac{1}{j(\tau)-j(z)} = \sum_{n \ge 0} \frac{j(z)^n}{j(\tau)^{n+1}}
\]
to deal with the factor of $j(\tau)-j(z)$ in the denominator.  (Note that
$1/j(\tau)$ has an expansion in terms of positive powers of $e^{2\pi i
\tau}$.)

\begin{proposition} \label{prop:ibf}
For $d \in \{8,24\}$, the interpolation basis functions from
Theorem~\ref{theorem:interpolation} satisfy
\[
a_n(r) = 4 \sin\mathopen{}\big(\pi r^2/2\big)^2\mathclose{}\,\int_0^\infty \alpha_n(it)\, e^{-\pi r^2 t} \, dt
\]
and
\[
b_n(r) = 4 \sin\mathopen{}\big(\pi
r^2/2\big)^2\mathclose{}\,\int_0^\infty \beta_n(it) \,e^{-\pi r^2 t} \, dt
\]
whenever $r^2 > 2n$, and
\[
\widetilde{a}_n(r) = 4 \sin\mathopen{}\big(\pi r^2/2\big)^2\mathclose{}\,\int_0^\infty \widetilde{\alpha}_n(it) \,e^{-\pi r^2 t} \, dt
\]
and
\[
\widetilde{b}_n(r) = 4 \sin\mathopen{}\big(\pi
r^2/2\big)^2\mathclose{}\,\int_0^\infty \widetilde{\beta}_n(it) \,e^{-\pi r^2
t} \, dt
\]
whenever $r^2 > 0$ for $d=8$, and whenever $r^2>2$ for $d=24$.
\end{proposition}

In other words, the interpolation basis functions are integral transforms of
the coefficients in the series expansions of $\K$ and $\widehat{\K}$. As
usual, one can meromorphically continue the integrals in $r$ by removing the
non-decaying terms from the $q$-expansion of the integrand and handling them
separately.  Specifically, one can show (using the proof of
Proposition~\ref{prop:ibf} and the explicit formulas for the kernels $\K$ and
$\widehat{\K}$) that the coefficients have expansions of the form
\begin{equation} \label{eq:coeffexpansions}
\aligned
\alpha_n(z) &= -iz \,e^{-2\pi i n z} + \sum_{m \ge m_0} \big(\alpha_{0,m,n} + \alpha_{1,m,n} z + \alpha_{2,m,n} z^2 \big) \,e^{\pi i m z},\\
\beta_n(z) &= \frac{e^{-2 \pi i n z}}{2\pi\sqrt{2n}} + \sum_{m \ge m_0} \big(\beta_{0,m,n} + \beta_{1,m,n} z + \beta_{2,m,n} z^2 \big) \,e^{\pi i m z}, \\
\widetilde{\alpha}_n(z) &= \sum_{m \ge m_0} \big(\widetilde{\alpha}_{0,m,n} + \widetilde{\alpha}_{1,m,n} z + \widetilde{\alpha}_{2,m,n} z^2 \big) \,e^{\pi i m z}, \quad\text{and}\\
\widetilde{\beta}_n(z) &= \sum_{m \ge m_0} \big(\widetilde{\beta}_{0,m,n} + \widetilde{\beta}_{1,m,n} z + \widetilde{\beta}_{2,m,n} z^2 \big) \,e^{\pi i m z}\\
\endaligned
\end{equation}
for some constants $\alpha_{\ell,m,n}$, $\beta_{\ell,m,n}$,
$\widetilde{\alpha}_{\ell,m,n}$, and $\widetilde{\beta}_{\ell,m,n}$, where
$m_0=0$ for $d=8$ and $m_0=-2$ for $d=24$. However, we will not need these
expansions.

\begin{proof}[Proof of Proposition~\ref{prop:ibf}]
As in \eqref{eq:anformula} and \eqref{eq:bnformula}, we can write the basis
functions as integrals involving the generating functions $F$ and
$\widetilde{F}$.  We begin with $\widetilde{F}$, for which the subsequent
analysis is a little simpler. For any $y>0$, the analogues of
\eqref{eq:anformula} and \eqref{eq:bnformula} for $\widetilde{a}_n$ and
$\widetilde{b}_n$ are
\begin{equation} \label{eq:atildenr}
\widetilde{a}_n(r) = \int_{-1+iy}^{iy} \left(\widetilde{F}(\tau,r) -
\tau\big(\widetilde{F}(\tau+1,r)-\widetilde{F}(\tau,r)\big)\right) \,e^{-2\pi i n \tau} \, d\tau
\end{equation}
and
\begin{equation} \label{eq:btildenr}
\widetilde{b}_n(r) = \frac{1}{2\pi i \sqrt{2n}} \int_{-1+iy}^{iy}
\big(\widetilde{F}(\tau+1,r)-\widetilde{F}(\tau,r)\big) \,e^{-2\pi i n \tau} \, d\tau.
\end{equation}
Now we use the identity \eqref{FtildefromKhatlaplace}, i.e.,
\begin{equation} \label{eq:Ftildeformula}
\widetilde{F}(\tau,r)= 4 \sin\mathopen{}\big(\pi
r^2/2\big)^2\mathclose{}\,\int_0^\infty \widehat{\K}(\tau,it)\,e^{-\pi r^2
t}\,dt,
\end{equation}
which holds for $\tau \in \mathcal{D}$ and $|r|$ sufficiently large
(specifically, $r^2>0$ when $d=8$ and $r^2>2$ when $d=24$).  We would like to
use the formulas
\begin{equation} \label{eq:alphatilden}
\widetilde{\alpha}_n(z) = \int_{-1+iy}^{iy} \left(\widehat{\K}(\tau,z) -
\tau\big(\widehat{\K}(\tau+1,z)-\widehat{\K}(\tau,z)\big)\right) \,e^{-2\pi i n \tau} \, d\tau
\end{equation}
and
\begin{equation} \label{eq:betatilden}
\widetilde{\beta}_n(z) = \frac{1}{2\pi i \sqrt{2n}} \int_{-1+iy}^{iy}
\big(\widehat{\K}(\tau+1,z)-\widehat{\K}(\tau,z)\big) \,e^{-2\pi i n \tau} \, d\tau,
\end{equation}
which hold for sufficiently large $y$ (depending on $z$) because of the
series expansions \eqref{eq:Kexpansion} and \eqref{eq:Khatexpansion}.  Before
we can use these formulas, we must check whether there is a single choice of
$y$ that works for all $z$ on the imaginary axis.  Fortunately, any $y \ge 1$
will work.  Specifically, it follows from the residue formulas in part~(3) of
Theorem~\ref{thm: K} that $\tau \mapsto \widehat{\K}(\tau,z)$ has no poles at
$z$, $-1/z$, $z\pm1$, or $-1/z\pm1$. In particular, if $z$ is on the
imaginary axis, then this function is holomorphic for $\Im(\tau)\ge1$ and $-1
\le \Re(\tau) \le 1$. Therefore the contour integral of the integrands in
\eqref{eq:alphatilden} and \eqref{eq:betatilden} around the rectangle with
vertex set $\{-1+iy$, $iy$, $iy'$, $-1+iy'\}$ is zero for any $y,y' \geq 1$.
The contributions from the vertical sides cancel since the integrand is
invariant under $\tau \mapsto \tau+1$, which proves that the integrals in
\eqref{eq:alphatilden} and \eqref{eq:betatilden} are independent of $y$ as
long as $y \geq 1$.

Now we substitute \eqref{eq:Ftildeformula} into \eqref{eq:atildenr} and
\eqref{eq:btildenr} for an arbitrary choice of $y \ge 1$.  By the
Fubini-Tonelli theorem we can interchange the order of integration. (To prove
absolute convergence of the iterated integral, we can use inequalities
\eqref{quotientbounds1new} and \eqref{quotientbounds1.5} from
Lemma~\ref{lem:boundsonKnew} with $\gamma=I$ to obtain bounds on the
integrand as $t \to 0$ or $t \to \infty$ that are uniform in $\tau$. Again we
need $r^2>0$ when $d=8$ and $r^2>2$ when $d=24$, because of the behavior of
$\mathcal G_{\pm}^{(d)}$ in these cases.) When we do so and apply
\eqref{eq:alphatilden} and \eqref{eq:betatilden}, we arrive at the desired
formulas for $\widetilde{a}_n$ and $\widetilde{b}_n$ in terms of
$\widetilde{\alpha}_n$ and $\widetilde{\beta}_n$.

The case of $a_n$ and $b_n$ involves two complications.  The function $\tau
\mapsto \K(\tau,z)$ has no poles at $z$, $-1/z$, or $-1/z\pm1$, but it has
poles at $z+1$ and $z-1$ with residues $-1/(2\pi)$ and $1/(2\pi)$,
respectively. Furthermore, we must account for the $e^{\pi i \tau r^2}$ term
in \eqref{FfromKlaplace2}, which says that
\begin{equation} \label{eq:restateF}
F(\tau,r)= e^{\pi i \tau r^2} + 4 \sin\mathopen{}\big(\pi
r^2/2\big)^2\mathclose{}\,\int_0^\infty \K(\tau,it)\,e^{-\pi r^2
t}\,dt,
\end{equation}
as long as $\tau \in \mathcal{D}$ and $r$ satisfies $r^2>0$ when $d=8$ and
$r^2>2$ when $d=24$. The $e^{\pi i \tau r^2}$ term will turn out to cancel
the contributions from the poles.

\begin{figure}
\begin{tikzpicture}[scale=1.5]
\filldraw[black] (-2,2) circle (0.03);
\draw (-2,2) node[left=2pt] {$-1+iy$};
\filldraw[black] (0,2) circle (0.03);
\draw (0,2) node[right=2pt] {$iy$};
\filldraw[black] (-2,0) circle (0.03);
\draw (-2,0) node[left=2pt] {$-1+i$};
\filldraw[black] (0,0) circle (0.03);
\draw (0,0) node[right=2pt] {$i$};
\filldraw[black] (-2,1) circle (0.03);
\draw (-2,1) node[left=2pt] {$z-1$};
\filldraw[black] (0,1) circle (0.03);
\draw (0,1) node[left=2pt] {$z$};
\draw[thick] (-2,0)--(-1,0);
\draw[->,thick] (0,0)--(-1,0);
\draw[thick,->] (-2,2)--(-1,2);
\draw[thick] (-1,2)--(0,2);
\draw[thick] (0,0.75) arc (-90:90:0.25);
\draw[thick] (-2,0.75) arc (-90:90:0.25);
\draw[thick] (0,0)--(0,0.75);
\draw[thick] (0,1.25)--(0,2);
\draw[thick] (-2,0)--(-2,0.75);
\draw[thick] (-2,1.25)--(-2,2);
\end{tikzpicture}
\caption{Contour shift when $\Im(z)>1$.}
\label{fig:cshift}
\end{figure}
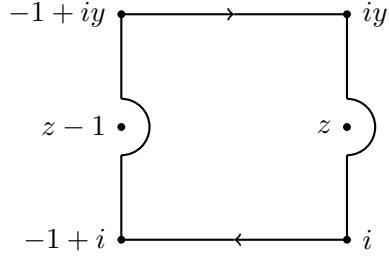

We begin by analyzing the effects of the poles.  As above,
\[
\alpha_n(z) = \int_{-1+iy}^{iy} \left(\K(\tau,z) -
\tau\big(\K(\tau+1,z)-\K(\tau,z)\big)\right) \,e^{-2\pi i n \tau} \, d\tau
\]
and
\[
\beta_n(z) = \frac{1}{2\pi i \sqrt{2n}} \int_{-1+iy}^{iy}
\big(\K(\tau+1,z)-\K(\tau,z)\big) \,e^{-2\pi i n \tau} \, d\tau
\]
when $z$ is on the imaginary axis and $y$ is sufficiently large. If $\Im(z) <
1$, then we can shift the contour in these integrals for $\alpha_n$ and
$\beta_n$ to $y=1$, as in the case of $\widetilde{\alpha}_n$ and
$\widetilde{\beta}_n$, because the integrand is holomorphic (since $\tau
\mapsto \K(\tau,z)$ has no poles at $-1/z$ and $-1/z\pm1$). We may and shall
ignore the case of $z=i$, since it has measure zero in the integrals for
$a_n$ and $b_n$ and the integrand is not singular there. If $\Im(z)>1$, then
the integrand has poles at $z-1$ and $z$. If we shift the contour to $y=1$ by
passing to the right of these poles, as in Figure~\ref{fig:cshift}, then this
contour shift contributes a residue from integrating the terms involving
$\K(\tau+1,z)$ clockwise around the pole at $z$. Thus, we obtain the formulas
\[
\alpha_n(z) = -iz\,e^{-2\pi i n z} + \int_{-1+i}^{i} \left(\K(\tau,z) -
\tau\big(\K(\tau+1,z)-\K(\tau,z)\big)\right) \,e^{-2\pi i n \tau} \, d\tau
\]
and
\[
\beta_n(z) = \frac{1}{2\pi i \sqrt{2n}} \left(i e^{-2\pi i n z} + \int_{-1+i}^{i}
\big(\K(\tau+1,z)-\K(\tau,z)\big) \,e^{-2\pi i n \tau} \, d\tau\right)
\]
for $\Im(z)>1$ (note that the terms coming from the poles are the dominant
terms in the first two asymptotic expansions listed in
\eqref{eq:coeffexpansions}).

Now substituting \eqref{eq:restateF} into \eqref{eq:anformula} and
\eqref{eq:bnformula} and interchanging the order of integration yields
\[
\aligned
a_n(r) &= \int_{-1+i}^{i} \left(e^{\pi i \tau r^2} -
\tau\left(e^{\pi i (\tau+1) r^2}-e^{\pi i \tau r^2}\right)\right)\, e^{-2\pi i n \tau} \, d\tau\\
& \quad \phantom{} +
4 \sin\mathopen{}\big(\pi
r^2/2\big)^2\mathclose{}\,\int_0^1 \alpha_n(it)\,e^{-\pi r^2
t}\,dt\\
& \quad \phantom{} +
4 \sin\mathopen{}\big(\pi
r^2/2\big)^2\mathclose{}\,\int_1^\infty \big(\alpha_n(it) -te^{2\pi n t}\big)\,e^{-\pi r^2
t}\,dt
\endaligned
\]
and
\[
\aligned
b_n(r) &= \frac{1}{2\pi i \sqrt{2n}}\int_{-1+i}^{i} \left(e^{\pi i (\tau+1) r^2}-e^{\pi i \tau r^2}\right)\, e^{-2\pi i n \tau} \, d\tau\\
& \quad \phantom{} +
4 \sin\mathopen{}\big(\pi
r^2/2\big)^2\mathclose{}\,\int_0^1 \beta_n(it)\,e^{-\pi r^2
t}\,dt\\
& \quad \phantom{} +
4 \sin\mathopen{}\big(\pi
r^2/2\big)^2\mathclose{}\,\int_1^\infty \left(\beta_n(it) -\frac{e^{2\pi n t}}{2\pi \sqrt{2n}}\right)\,e^{-\pi r^2
t}\,dt.
\endaligned
\]
As above, the estimates \eqref{quotientbounds1new} and
\eqref{quotientbounds1.5} imply absolute convergence for the original
integrals, and then the Fubini-Tonelli theorem shows that the interchanged
integrals converge absolutely and the interchange is justified, as long as
$r$ satisfies $r^2>0$ when $d=8$ and $r^2>2$ when $d=24$, to justify the
application of \eqref{eq:restateF}. If furthermore $r^2>2n$, then $t \mapsto
\alpha_n(it)\,e^{-\pi r^2 t}$ and $t \mapsto \beta_n(it)\,e^{-\pi r^2 t}$ are
integrable over $[1,\infty)$, because the equations
\[
\alpha_n(it)\,e^{-\pi r^2 t} = \big(\alpha_n(it) -te^{2\pi n t}\big)\,e^{-\pi r^2
t} + te^{-\pi(r^2-2n)t}
\]
and
\[
\beta_n(it)\,e^{-\pi r^2 t} = \left(\beta_n(it) -\frac{e^{2\pi n t}}{2\pi \sqrt{2n}}\right)\,e^{-\pi r^2
t} + \frac{e^{-\pi(r^2-2n)t}}{2\pi \sqrt{2n}}
\]
express them as the sum of integrable functions.  Thus,
\[
\int_1^\infty \big(\alpha_n(it) -te^{2\pi n t}\big)\,e^{-\pi r^2
t}\,dt =\int_1^\infty \alpha_n(it)\,e^{-\pi r^2
t}\,dt - \int_1^\infty te^{-\pi(r^2-2n)t} \, dt
\]
and
\[
\int_1^\infty \left(\beta_n(it) -\frac{e^{2\pi n t}}{2\pi \sqrt{2n}}\right)\,e^{-\pi r^2
t}\,dt = \int_1^\infty \beta_n(it)\,e^{-\pi r^2
t}\,dt - \int_1^\infty \frac{e^{-\pi(r^2-2n)t}}{2\pi \sqrt{2n}}\,dt.
\]
All the extraneous terms not involving $\alpha_n$ or $\beta_n$ cancel, and we
obtain the desired formulas for $a_n$ and $b_n$.
\end{proof}

\section{Positivity of kernels and universal optimality} \label{sec:inequality}

\subsection{Sharp bounds for energy} \label{subsec:sharpbounds}

In this section we will prove that $E_8$ and the Leech lattice are
universally optimal (Theorem~\ref{theorem:univopt}).  Let $d$ be $8$ or $24$,
with $\Lambda_d$ being the corresponding lattice and $F$ the generating
function from Theorem~\ref{thm:FEimpliesIF}.  The key inequality is the
following proposition:

\begin{proposition}\label{prop:positivity}
Suppose $\Re(\tau)=0$ and $r>0$.  Then $\widetilde{F}(\tau,r)\ge 0$, with
equality if and only if $r^2$ is an even integer and $r^2 \ge 2n_0$, where
$n_0=1$ if $d=8$ and $n_0=2$ if $d=24$.
\end{proposition}

The rest of this section is devoted to proving
Proposition~\ref{prop:positivity}, but first we show in this subsection
that it implies universal optimality.

\begin{lemma} \label{lemma:optimalfn}
For $\alpha>0$, the function $f \colon \R^d \to \R$ defined by $f(x) =
F(i\alpha/\pi,x)$ is a Schwartz function and satisfies the inequalities $f(x)
\le e^{-\alpha |x|^2}$ and $\widehat{f}(x) \ge 0$ for all $x \in \R^d$. When
$x \ne 0$, equality holds if and only if $|x|$ is the length of a nonzero
vector in $\Lambda_d$.
\end{lemma}

In fact, equality does not hold when $x=0$ either.  That follows from
\eqref{eq:conj61} below, but we will not need it.

\begin{proof}
This function is a Schwartz function since $F$ satisfies condition~(3) from
Theorem~\ref{thm:FEimpliesIF}, so the substantive content of the lemma is the
inequalities. The second inequality follows directly from
Proposition~\ref{prop:positivity}, because $\widehat{f}(x) =
\widetilde{F}(i\alpha/\pi,x)$ (as shown at the end of
Section~\ref{sec:sub:Thm110proof}). To prove that $f(x) \le e^{-\alpha
|x|^2}$, we use the functional equation $F(\tau,x) + (i/\tau)^{d/2}
\widetilde{F}(-1/\tau,x) = e^{\pi i \tau |x|^2}$ with $\tau = i\alpha/\pi$ to
obtain
\[
e^{-\alpha |x|^2} - f(x) = (\pi/\alpha)^{d/2} \widetilde{F}(i\pi/\alpha,x).
\]
Thus, the first inequality amounts to Proposition~\ref{prop:positivity} as
well, as do the conditions for equality.
\end{proof}

Although the inequalities $f(x) \le e^{-\alpha |x|^2}$ and $\widehat{f}(x)
\ge 0$ look different, the preceding proof derives them from the same
underlying inequality. More generally, Cohn and Miller \cite[Section~6]{CM}
observed a duality principle when the potential function $p$ is a Schwartz
function: the auxiliary function $f$ proves a bound for $p$-energy in
Proposition~\ref{prop:LP} if and only if $\widehat{p}-\widehat{f}$ proves a
bound for $\widehat{p}$-energy, and this transformation interchanges the two
inequalities.  Furthermore, a lattice $\Lambda$ attains the $p$-energy bound
proved by $f$ if and only if $\Lambda^*$ attains the $\widehat{p}$-energy
bound proved by $\widehat{p}-\widehat{f}$.

It follows immediately from Lemma~\ref{lemma:optimalfn} that $\Lambda_d$
minimizes energy for all Gaussian potential functions (recall the
conditions~\eqref{eq:necessaryconds1} for equality in the linear programming
bounds). Furthermore, we can prove uniqueness among periodic configurations
under Gaussian potential functions
as follows.  If $\mathcal{C}$ is any periodic configuration in $\R^d$ of
density~$1$ with the same energy as $\Lambda_d$ under some Gaussian, then the
distances between points in $\mathcal{C}$ must be a subset of those in
$\Lambda_d$, because of the equality conditions. Without loss of generality
we can assume $0 \in \mathcal{C}$.  Then, by \cite[Lemma~8.2]{CE},
$\mathcal{C}$ is contained in an even integral lattice (namely, the subgroup
of $\R^d$ generated by $\mathcal{C}$), because all the distances between
points in $\mathcal{C}$ are square roots of even integers. Because
$\mathcal{C}$ has density~$1$, it must be the entire lattice. We conclude
that it must be isometric to $\Lambda_d$, because there is only one such
lattice with minimal vector length $\sqrt{2n_0}$ (see \cite[Chapters~16
and~18]{SPLAG}).  Thus, Theorem~\ref{theorem:univopt} holds for Gaussian
potential functions.

Handling other potential functions via linear programming bounds is slightly
more technical, because the potential function might decrease too slowly for
any Schwartz function to interpolate its values.  For example, no Schwartz
function can directly prove a sharp bound for energy under an inverse power law
potential in Proposition~\ref{prop:LP}. Nevertheless, we will show that
Schwartz functions come arbitrarily close to a sharp bound. Suppose we are
using a completely monotonic function of squared distance $p \colon
(0,\infty) \to \R$. By Bernstein's theorem \cite[Theorem~9.16]{Simon2}, there
is some measure $\mu$ on $[0,\infty)$ such that
\[
p(r) = \int e^{-\alpha r^2} \, d\mu(\alpha)
\]
for all $r \in (0,\infty)$ (which implies that $\mu$ must be locally finite).
Without loss of generality we can assume $\mu(\{0\})=0$, since otherwise all
configurations of density~$1$ have infinite energy. We would like to use
\[
f(x) = \int F(i\alpha/\pi,x) \, d\mu(\alpha)
\]
as an auxiliary function for the potential function $p$, and it might
plausibly work under the weaker hypotheses for linear programming bounds
proved in \cite[Proposition~2.2]{CdCI}. However, it will not be a Schwartz
function in general, and we will not analyze the behavior of this integral.
Instead, let
\[
f_\varepsilon(x) = \int_{\varepsilon}^{1/\varepsilon} F(i\alpha/\pi,x) \, d\mu(\alpha),
\]
which defines a Schwartz function for each $\varepsilon>0$ because $F$
satisfies condition~(3) of Theorem~\ref{thm:FEimpliesIF}.  Then
\[
\widehat{f_\varepsilon}(y) = \int_{\varepsilon}^{1/\varepsilon} \widetilde{F}(i\alpha/\pi,y) \, d\mu(\alpha),
\]
and the inequalities $f_\varepsilon(x) \le p\big(|x|\big)$ for all $x \in
\R^d\setminus\{0\}$ and $\widehat{f_\varepsilon}(y) \ge 0$ for all $y \in
\R^d$ follow from Lemma~\ref{lemma:optimalfn}. Thus, every configuration in
$\R^d$ of density~$1$ has lower $p$-energy at least
$\widehat{f_\varepsilon}(0)-f_\varepsilon(0)$, by Proposition~\ref{prop:LP}.
Because of the sharp bound for each $\alpha$,
\[
\aligned
\widehat{f_\varepsilon}(0)-f_\varepsilon(0) &= \int_{\varepsilon}^{1/\varepsilon} E_{r \mapsto e^{-\alpha r^2}}(\Lambda_d) \, d\mu(\alpha)\\
&= \int_{\varepsilon}^{1/\varepsilon} \sum_{x \in \Lambda_d \setminus\{0\}} e^{-\alpha |x|^2} \, d\mu(\alpha)\\
&= \sum_{x \in \Lambda_d \setminus\{0\}} \int_{\varepsilon}^{1/\varepsilon}  e^{-\alpha |x|^2} \, d\mu(\alpha).
\endaligned
\]
As $\varepsilon \to 0$, this bound converges to the $p$-energy of $\Lambda_d$
by monotone convergence since $\mu(\{0\})=0$, and we conclude that
$\Lambda_d$ has minimal $p$-energy.

Uniqueness among periodic configurations also follows from Bernstein's
theorem.  Suppose $\mathcal{C}$ is any periodic configuration of density~$1$
that is not isometric to $\Lambda_d$.  We have seen that the energy of
$\mathcal{C}$ under $r \mapsto e^{-\alpha r^2}$ is strictly greater than that
of $\Lambda_d$ for each $\alpha>0$, and these energies are continuous
functions of $\alpha$ by \eqref{eq:periodicenergy}.  By continuity, for each
compact subinterval $I$ of $(0,\infty)$, there exists $\varepsilon>0$ such
that the energy gap between $\mathcal{C}$ and $\Lambda_d$ is at least
$\varepsilon$ for all $\alpha \in I$.  Thus, Bernstein's theorem and monotone
convergence show that $E_p(\mathcal{C}) \ge E_p(\Lambda_d) + \delta$ for some
$\delta>0$. In particular, $E_p(\mathcal{C}) > E_p(\Lambda_d)$ if
$E_p(\Lambda_d) < \infty$, as desired. This completes the proof of
Theorem~\ref{theorem:univopt}, except for proving
Proposition~\ref{prop:positivity}.

In addition to proving universal optimality, our construction also
establishes other properties of the optimal auxiliary functions.  For
example, the following proposition follows directly from \eqref{FfromKlaplace2},
\eqref{decomposeF2},
and \eqref{truncatingF2} for $\tau \in \mathcal{D}$, and for all $\tau \in
\Hyp$ by analytic continuation:

\begin{proposition} \label{prop:specialvalues}
For all $\tau \in \Hyp$ and $d \in \{8,24\}$,
\begin{align*}
F(\tau,0) &= 1+i\,\mathcal{G}_{0,1}(\tau),\\
F\big(\tau,\sqrt{2}\big) &= e^{2\pi i \tau} + i\,\mathcal{G}_{-1,1}(\tau),\\
\left.\frac{\partial^2}{\partial r^2}\right|_{r=0}F(\tau,r) &= 2\pi i \tau + 2\pi\,\mathcal{G}_{0,0}(\tau), \quad \text{and}\\
\left.\frac{\partial}{\partial r}\right|_{r=\sqrt{2}}F(\tau,r) &= 2\pi i \sqrt{2} e^{2\pi i \tau} + 2\pi\sqrt{2}\,\mathcal{G}_{-1,0}(\tau).\\
\end{align*}
\end{proposition}

Equivalently, the proposition specifies these values and derivatives of the
interpolation basis functions.  It gives another interpretation of the
non-decaying asymptotics $\mathcal{G}_{k,j}$ of the kernel $\K$, and it
generalizes the computation of special values of the optimal sphere packing
auxiliary functions in \cite[Propositions~4 and~8]{V} and \cite[Sections~2
and~3]{CKMRV}.

If one computes $\widetilde{F}(\tau,0)$ using
Proposition~\ref{prop:specialvalues} and the functional equation relating $F$
and $\widetilde{F}$, one obtains $\tau$ times an explicit quasimodular form.
Using the theta series for $\Lambda_d$ and Ramanujan's derivative formulas
for modular forms \cite[Section~5]{Z}, a straightforward calculation shows
that
\begin{equation} \label{eq:conj61}
\widetilde{F}(\tau,0) = -\frac{2\pi i \tau}{d} \sum_{x \in \Lambda_d} |x|^2 e^{\pi i |x|^2 \tau}.
\end{equation}
Equivalently, if $f$ is the auxiliary function from
Lemma~\ref{lemma:optimalfn} for the potential function $r \mapsto e^{-\alpha
r^2}$ (corresponding to $\tau = i \alpha/\pi$ in the formula above), then
\[
\widehat{f}(0) = \frac{2\alpha}{d} E_{r \mapsto r^2 e^{-\alpha r^2}} (\Lambda_d),
\]
in agreement with the prediction in \cite[Conjecture~6.1]{CM}.

\subsection{Reduction to positivity of kernels}

To complete the proof of universal optimality, all that remains is to prove
Proposition~\ref{prop:positivity}, i.e., the inequality
$\widetilde{F}(\tau,r) \ge 0$ and the conditions for equality. For the rest
of the section we thus assume $\Re(\tau)=0$ (in particular, $\tau\in\mathcal
D$).

The first obstacle to proving that $\widetilde{F}(\tau,r) \ge 0$ is dealing
with the integral transform that defines $\widetilde{F}$ in terms of the
kernel $\widehat{\K}$.  The kernel is written explicitly in terms of
well-known special functions, and we will deduce the positivity of the
integral at the level of the kernel itself. As in the sphere packing papers
\cite{V} and \cite{CKMRV}, that will involve additional complications for
$d=24$ beyond those that occur in the case of $d=8$.

Specifically, recall from \eqref{FtildefromKhatlaplace} that
\begin{equation}\label{Ftildedefagain}
\widetilde{F}(\tau,r)=4 \sin\mathopen{}\big(\pi r^2/2\big)^2\mathclose{} \int_0^\infty \widehat{\K}(\tau,it)e^{-\pi r^2 t}\,dt,
\end{equation}
which is absolutely convergent for $\tau\in\mathcal D$ and $|r|$ sufficiently
large, and has an analytic continuation to $r$ in some open neighborhood of
$\R$ in $\C$.   We showed in Section~\ref{sec:sub:continueFtoH} that this
continuation of \eqref{Ftildedefagain} can be achieved by subtracting pieces
of the asymptotics of $\mathcal K(\tau,it)$ as $t \to \infty$, as in
\eqref{decomposeF2} and \eqref{truncatingF2}.  Since
Proposition~\ref{prop:positivity} does not involve the point $r=0$, here it
suffices to perform a milder truncation by subtracting only the $k=-1$ terms
in \eqref{Gjkexpansion}, and only for dimension $d=24$, since for $d=8$ the
$k = -1$ terms vanish and the integral in \eqref{Ftildedefagain} is
absolutely convergent for all $r>0$.

When $d=8$, our strategy for proving Proposition~\ref{prop:positivity} is to
show that $\widehat{\K}^{(8)}(\tau,it)>0$, which immediately implies both the
desired inequality and the equality conditions. The analogous inequality
$\widehat{\K}^{(24)}(\tau,it)>0$ for $d=24$ holds as well, but additional
work is needed to deal with small $r$. Specifically, we write
$\widehat\K^{(24)}(\tau,it)=\widehat{\mathcal E}(\tau,it)+O(t)$ as $t \to
\infty$, where
\[
\widehat{\mathcal E}(\tau,z)  = \frac{e^{-2\pi i  z}}{3456\pi}( z \widehat{\mathcal E}_1(\tau)+\widehat{\mathcal E}_0(\tau))
\qquad \text{and}\qquad
 \widehat{\mathcal E}_j(\tau)=\tau \widehat{\mathcal E}_{j,1}(\tau)  + \widehat{\mathcal E}_{j,0}(\tau),
\]
with
\begin{align*}
\widehat{\mathcal E}_{0,0}& =-6912 \log(2) \Delta-36 E_2 E_4 E_6+16 E_4^3+20E_6^2+108 E_4  \sqrt{\Delta} (V\mathcal L + W \mathcal L_S),\\
\widehat{\mathcal E}_{0,1}&=-\pi i  \big(6 E_2^2 E_4 E_6-5E_2E_4^3-7 E_2 E_6^2+6 E_4^2E_6\big),\\
\widehat{\mathcal E}_{1,0}&=12 \pi i  \big({-E_2E_4E_6} + E_6^2 + 720\Delta\big),\quad \text{and} \\
\widehat{\mathcal E}_{1,1}&=2\pi ^2  \big(E_2^2 E_4E_6-2E_2E_6^2-1728E_2 \Delta+E_4^2E_6 \big)
\end{align*}
expressed in terms of quasimodular forms, $\mathcal L$, and $\mathcal{L}_S$.

For $d=24$ the integral in \eqref{Ftildedefagain} converges absolutely for
$|r|>\sqrt{2}$.  For the range $0<|r|\le \sqrt{2}$ we will use the truncation
method from \cite{CKMRV}, by instead setting $p=1.01$ and writing
\[
\int_0^\infty \widehat{\K}^{(24)}(\tau,it)e^{-\pi r^2 t}\,dt=  \int_0^\infty \widehat{\K}^{\textup{trunc}}(\tau,it)e^{-\pi r^2 t}\,dt
+ \int_p^\infty \widehat{\mathcal E}(\tau,it)e^{-\pi r^2t}\,dt,
\]
where
\begin{equation}\label{truncation}
\widehat{\K}^{\textup{trunc}}(\tau,it)=\left\{
\begin{array}{ll}
\widehat{\K}^{(24)}(\tau,it) & \text{for $t<p$, and} \\
\widehat{\K}^{(24)}(\tau,it)-\widehat{\mathcal E}(\tau,it) & \text{for $t\ge p$.}
\end{array}
\right.
\end{equation}
The value of $p$ has been chosen to ensure the positivity properties below,
in particular so that $\lambda(ip)<0.49$, where $\lambda$ is the modular
function from Section~\ref{sec:sub:sub:modularlambda} (for comparison,
$\lambda(i)=1/2$). The last integral in \eqref{truncation} can be evaluated
as
\begin{multline*}\int_p^\infty\frac{it \widehat{\mathcal E}_1(\tau)+\widehat{\mathcal E}_0(\tau)}{3456\pi} e^{\pi(2- r^2)t} \,dt=\\
\frac{e^{-p \pi (r^2-2)}}{(r^2-2)^2}\cdot\frac{\widehat{\mathcal E}_0(\tau)\pi(r^2-2)+i \widehat{\mathcal E}_1(\tau)(1+p \pi (r^2-2))}{3456\pi^3}.
\end{multline*}
The singularities from the factor of $(r^2-2)^2$ in the denominator are
compensated for by the vanishing of the $\sin\mathopen{}\big(\pi
r^2/2\big)^2\mathclose{}$ factor in \eqref{Ftildedefagain}. Because $n_0=2$
and $\sin\mathopen{}\big(\pi r^2/2\big)^2\mathclose{}$ vanishes at other
$r^2\in 2\Z$, we deduce that Proposition~\ref{prop:positivity} is a
consequence of the following three statements for $\Re(\tau)=0$ and $t \in
(0,\infty)$:
\begin{equation}\label{imply123}
 \aligned
(1) \ &  \widehat{\K}^{(d)}(\tau,it) >0 \text{~for~} d\in\{8,24\},\\
(2)\ & \widehat{\K}^{\textup{trunc}}(\tau,it) >0 \text{~for~} t \ge p,\text{~and}\\
(3)\  & \widehat{\mathcal E}_0(\tau)\pi(r^2-2)+i \widehat{\mathcal E}_1(\tau)(1+p \pi (r^2-2))>0 \text{~for~}
r\le \sqrt{2}.
\endaligned
\end{equation}
Statement (3) is itself a consequence of
\[ % ad hoc spacing to match previous display
\aligned
\quad\!& (\textup{3a}) \  \widehat{\mathcal E}_0(\tau)+ip \widehat{\mathcal E}_1(\tau)<0,\text{~and}
\qquad\qquad\qquad\qquad\qquad\quad  \ \
\\
\quad\!& (\textup{3b})\   i\widehat{\mathcal E}_1(\tau)>0.\qquad\qquad\qquad\qquad\qquad\quad \ \
\endaligned
\]
We have no simple proof of these inequalities, but we will outline below how
we have proved them by mathematically rigorous computer calculations.

Proving inequalities of this sort for quasimodular forms also arose in the
sphere packing papers \cite{V} and \cite{CKMRV}, but the computations are
much more challenging in our case.  In the sphere packing cases, all that was
needed was to prove positivity for relatively simple functions of a single
variable. Asymptotic calculations reduce the proof to analyzing functions on
compact intervals, and that is a straightforward and manageable computation
using any of several techniques (\cite{V} used interval arithmetic and
\cite{CKMRV} used $q$-expansions).

By contrast, inequalities~(1) and~(2) in \eqref{imply123} involve much more
complicated functions of two variables. We must analyze singularities along
curves, which are more subtle than the point singularities in one dimension.
Furthermore, these curves intersect, and the intersection points are
particularly troublesome. In the rest of this section we explain how to
overcome these obstacles.

Inequalities~(3a) and~(3b) involve only a single variable $\tau$, and can be
verified using the methods from \cite[Appendix~A]{CKMRV}, or the $\lambda$
function coordinates that we use below for the other inequalities. We thus
focus on inequalities~(1) and~(2), describing the underlying mathematical
ideas that were rigorously verified by a computer calculation.

\subsection{Passing to the unit square} \label{sub:sec:passing}

Recall from Section~\ref{sec:sub:sub:modularlambda} that $t \mapsto
\lambda(it)$ is a decreasing function mapping $(0,\infty)$ onto $(0,1)$. Our
first step is to express the kernels $\widehat{\K}^{(d)}$ and
$\widehat{\K}^\textup{trunc}$ in terms of functions on the interior of the
unit square by inverting the map $(\tau,z) \mapsto
(\lambda(\tau),\lambda(z))$. Rewriting the kernels in this way is not
logically necessary, but it has the advantage of expressing everything in
terms of a small number of functions that can be bounded systematically and
efficiently, and we can take advantage of relationships between these
functions to obtain more accurate estimates when proving bounds.

Specifically, if we use the identities \eqref{thetatoEisenstein},
\eqref{thetafromK}, \eqref{zintermsofKlambda}, and \eqref{E2fromK} to write
modular forms, $z$, $\tau$, and $E_2$ in terms of $\lambda$  and write
$\mathcal{L} = \log(\lambda)$ and $\mathcal{L}_S = \log (1-\lambda)$ on the
imaginary axis, we obtain functions $L^{(d)}$ and $L^\textup{trunc}$ on
$(0,1)\times (0,1)$ such that
\[
\K^{(d)}(\tau,z) = L^{(d)}(\lambda(\tau),\lambda(z)) \qquad \text{and}\qquad
\K^{\textup{trunc}}(\tau,z) = L^{\textup{trunc}}(\lambda(\tau),\lambda(z))
\]
when $\Re(\tau)=\Re(z)=0$. The resulting functions $L^{(d)}(x,y)$ are
rational functions of $x$, $y$, and the logarithms and complete elliptic
integrals of $x$, $1-x$, $y$, and $1-y$, while $L^{\textup{trunc}}$ is
slightly more complicated, as described below.

The inequalities (1) and (2) in \eqref{imply123} thus transform into the
assertions that
\[
L^{(d)}(x,y)>0
\]
for $0<x,y<1$ and $d=8$ or $24$, and that
\[
L^{\textup{trunc}}(x,y)=L^{(24)}(x,y)-\psi^\textup{trunc}(x,y)>0
\]
for $0<x<1$ and $0<y<0.49$, where
$\psi^\textup{trunc}(\lambda(\tau),\lambda(z))=\widehat{\mathcal E}(\tau,z)$
by \eqref{truncation} and the constant $0.49$ is just slightly larger than
$\lambda(ip)=\lambda(1.01i)=0.4891135\dots.$  One complication with
$L^{\textup{trunc}}$ is that $\widehat{\mathcal E}(\tau,z)$ involves a factor
of $e^{-2\pi i z}$, which becomes $e^{2\pi K(1-y)/K(y)}$ by
\eqref{zintermsofKlambda} when we set $y=\lambda(z)$. We write
\[
\psi^\textup{trunc}(x,y)=e^{2\pi
K(1-y)/K(y)}\widetilde{\psi}^\textup{trunc}(x,y),
\]
where
$\widetilde{\psi}^\textup{trunc}(\lambda(\tau),\lambda(z))=(z{\widehat{\mathcal E}}_1(\tau)+{\widehat{\mathcal
E}}_0(\tau))/(3456\pi)$. Observe that
\[
L^\textup{trunc}(x,y) \ge L^{(24)}(x,y)
\]
whenever $\widetilde{\psi}^\textup{trunc}(x,y)\le 0$, and otherwise
$L^\textup{trunc}(x,y)$ is bounded below by
\[
\widetilde{L}^\textup{trunc}(x,y):=
L^{(24)}(x,y)-\left(\frac{256}{y^2}-\frac{256}{y}+24+\frac{4\cdot 10^9}{970299}y^2\right)\widetilde{\psi}^\textup{trunc}(x,y)
\]
according to \eqref{expKratiobound} below. In particular, inequality (1) in
\eqref{imply123} and the positivity of $\widetilde{L}^\textup{trunc}(x,y)$
for $0<x<1$ and $0<y<0.49$ together imply inequality (2).

The lower bound used above can be obtained by truncating the Taylor series of
$e^{2\pi K(1-y)/K(y)}$ and bounding the omitted coefficients.  It is most
convenient to obtain such bounds via complex analysis.  For example, the
functions $A_j(z)$ in \eqref{EandKnear1}, along with $E(z)$, $K(z)$, and
$\log(1-z)$, all have modulus bounded by $5$ on $\{z \in \C : |z|=0.99\}$,
and so their Taylor series coefficients of $z^n$ are bounded above by $5\cdot
0.99^{-n}$. (This bound of $5$, which is easily improved for some of these
individual functions, comes from the constant sign of the coefficients of
$z^n$ for $n>1$ and the value at $z=0.99$.) For the bound needed above, one
can check that $z^2 e^{2\pi K(1-z)/K(z)}$ is holomorphic on the open unit
disk and
\begin{equation}\label{expKratiobound}
e^{2\pi K(1-y)/K(y)}\le \frac{256}{y^2}-\frac{256}{y}+24+ \frac{4\cdot 10^9}{970299} y^2
\end{equation}
for $0<y<1/2$, where the error term comes from the bound $|K(z)|\ge 1.3$ for
$|z|=0.99$ (which itself can be shown by evaluation at close points on the
circle and derivative bounds).

The next several subsections describe the verification of inequalities (1)
and (2) in \eqref{imply123}.  The primary difficulty is dealing with
singularities. Since our formulas for the kernels involve denominators of
$j(\tau)-j(z)$, which vanish when $\tau=z$ or $\tau=-1/z$, our formulas for
$L^{(d)}(x,y)$ and $\widetilde{L}^\textup{trunc}(x,y)$ naively yield $0/0$
when $x=y$ or $x=1-y$. The kernels themselves are not actually singular along
these lines, because $\phi(I)=\phi(S)=0$ in the residue formulas from
part~(3) of Theorem~\ref{thm: K}, but in practice we must treat the diagonal
lines $x=y$ and $x=1-y$ as singularities in the formulas. In addition, there
are singularities at the edges of the unit square coming from
\eqref{EandKnear1}. In this rest of this section, we describe the methods
used to treat these singularities, starting away from any singularities and
working our way up to the most singular points: the four corners of the unit
square, at which three singularities meet.

In our numerical calculations, it is convenient to remove obviously positive
factors from $L^{(d)}(x,y)$ and $\widetilde{L}^\textup{trunc}(x,y)$. To do
so, we multiply each of them by
\[
\big(1-xy\big)\big(1-x(1-y)\big)\big(1-y(1-x)\big)\big(1-(1-x)(1-y)\big),
\]
and we furthermore multiply $L^{(8)}(x,y)$ by
\[
\frac{\pi^4 K(y)^2K(1-x)}{K(x)^4},
\]
$L^{(24)}(x,y)$ by
\[
\frac{6\pi^4 y^2(1-y)^2 K(y)^{10}K(1-x)}{K(x)^{12}},
\]
and $\widetilde{L}^\textup{trunc}(x,y)$ by
\[
\frac{54\pi^{14} y^2(1-y)^2 K(y)^{12}}{K(x)^{11}}.
\]
For simplicity of notation, in the remainder of Section~\ref{sec:inequality}
we use the notation $L^{(d)}(x,y)$ and $\widetilde{L}^\textup{trunc}(x,y)$ to
refer to the functions after removing these factors.

\subsection{Interval bounds for elliptic integrals}
\label{subsec:elliptic}

Away from all singularities it is possible to prove positivity via interval
arithmetic estimates on $L^{(d)}(x,y)$ and
$\widetilde{L}^\textup{trunc}(x,y)$.  Interval arithmetic provides rigorous
upper and lower bounds on the values of a function over a given interval
(see, for example, \cite{MKC}).  It works beautifully for small intervals and
well-behaved functions, but the bounds become much less tight for large
intervals or near singularities.  In practice, instead of simply subdividing
intervals to improve the bounds, we obtained better results by using interval
arithmetic to evaluate Taylor series expansions, while controlling the error
terms by using crude interval arithmetic bounds on partial derivatives,
because these error bounds do not need to be tight.

Interval arithmetic for polynomials and logarithms is standard and is part of
many software packages. We will next describe how to obtain rigorous interval
bounds for the complete elliptic integrals $E$ and $K$ by adapting the
arithmetic-geometric mean algorithms for computing them from \cite[Chapter
1]{BB}. Consider the sequences $(a_n)_{n\ge 0}$ and $(b_n)_{n\ge 0}$ with
$a_0 \ge b_0>0$ and
\[
a_{n+1} = \frac{a_n+b_n}{2}\quad\text{and}\quad b_{n+1}=\sqrt{a_nb_n}.
\]
Then $b_n\le b_{n+1}\le a_{n+1}\le a_n$ and both $a_n$ and $b_n$ converge to
the same limit $M(a_0,b_0)$ (the \emph{arithmetic-geometric mean} of $a_0$
and $b_0$), which is related to the complete elliptic integral $K$ by
\[
K(x) = \frac{\pi}{2 M\big(1,\sqrt{1-x}\big)}
\]
for $0<x<1$. Since the interval $[b_n,a_n]$ contains $M(a_0,b_0)$, these
recurrences give a fast interval-arithmetic algorithm for $K$.

Computations of $E$ are more subtle and use the formula
\begin{equation}\label{BBEK}
\frac{E(x)}{K(x)}=1-\sum_{n\ge 0}2^{n-1}c_n^2
\end{equation}
from \cite[Algorithm~1.2]{BB}, where $c_0=\sqrt{x}$ and
\[
c_{n+1}=\frac{a_n-b_n}{2}
\]
with $a_n$ and $b_n$ defined as above, starting with
$(a_0,b_0)=\big(1,\sqrt{1-x}\big)$; one can show that
$c_{n+1}=c_n^2/(4a_{n+1})$, which avoids potential precision loss from
subtracting $a_n$ and $b_n$. Since
\[
c_n = \frac{a_{n-1}-b_{n-1}}{2} \leq \frac{a_{n-1}}{2} \leq a_n \leq
2a_{n+1},
\]
we have
\[
c_{n+1}=\frac{c_n^2}{4a_{n+1}}\le \frac{c_n}{2},
\]
and therefore the tail of the series is
\[
\sum_{n> m}2^{n-1}c_n^2 \le \sum_{n>
m}2^{n-1}(2^{m+1-n}c_{m+1})^2=2^{m+1}c_{m+1}^2.
\]
From this tail bound, truncating the series in \eqref{BBEK} yields interval
bounds for the ratio $E(x)/K(x)$, and hence $E(x)$ itself in light of the
algorithm for $K(x)$ above.

\subsection{Near the diagonals and their crossing}

After removing factors that are obviously positive as discussed above, the
kernels $L^{(d)}(x,y)$ and $\widetilde{L}^\textup{trunc}(x,y)$ have
denominator $(x-y)(1-x-y)$, and numerators that vanish at $x=y$ and $x=1-y$,
as they must in order to be well defined on the interior of the unit square.
We use Taylor expansions in one of the variables (along with rigorous
interval bounds on partial derivatives, which are themselves expressible in
terms of polynomials, logarithms, $E$, and $K$) to prove that the kernels are
positive near the diagonals.  For this purpose it is convenient to work with
the coordinate system $(u,v)=(x-y,1-x-y)$ and take partial derivatives in $u$
and $v$.

The most subtle point is the diagonal crossing point
$(x,y)=(\frac{1}{2},\frac{1}{2})$.  There both $u=x-y$ and $v=1-x-y$ vanish,
and so $L^{(d)}(x,y)$ and $\widetilde{L}^\textup{trunc}(x,y)$ can each be
written in the form $\frac{f(u,v)}{uv}$, where $f(u,0) = f(0,v) = 0$ for all
$u$ and $v$. To analyze these functions, we obtain bounds on their Taylor
series coefficients by writing
\[
\frac{f(u,v)}{uv} = \int_0^1 \int_0^1 f^{(1,1)}(us,vt) \, ds \, dt,
\]
where $f^{(i,j)}(u,v)$ denotes $(\partial^i/\partial u^i)(\partial^j/\partial
v^j) f(u,v)$, and hence
\[
\frac{\partial^i}{\partial u^i} \frac{\partial^j}{\partial v^j} \frac{f(u,v)}{uv} = \int_0^1 \int_0^1 s^i t^j f^{(i+1,j+1)}(us,vt) \, ds \, dt.
\]
It follows that
\[
\left|\frac{\partial^i}{\partial u^i} \frac{\partial^j}{\partial v^j} \frac{f(u,v)}{uv}\right|
\le \frac{\max_{0 \le s,t \le 1} \big|f^{(i+1,j+1)}(us,vt)\big|}{(i+1)(j+1)}.
\]
Interval arithmetic bounds on the derivatives then gives rigorous upper
bounds on the error terms in the Taylor expansion, which can be used to show
positivity close to $(x,y)=(\frac{1}{2},\frac{1}{2})$.

\subsection{Near the edges}\label{sec:sub:nearedges}

For the rest of this section, we reduce to the situation of $0<x,y\le
\frac{1}{2}$ by changing coordinates from $x$ to $1-x$ and from $y$ to $1-y$
as needed. The singularities at the edges come from logarithms as well as the
logarithmic behavior of the elliptic integrals $E(z)$ and $K(z)$ near $z=1$
from \eqref{EandKnear1}. After substituting those formulas for $E$ and $K$,
we arrive at a sum of powers of logarithms times holomorphic functions.  We
compute Taylor series for these holomorphic functions in the direction
orthogonal to the edge in question, and use interval arithmetic in the
tangential direction, while taking into account the Taylor coefficient bounds
for special functions mentioned in the paragraph containing
\eqref{expKratiobound}.  On narrow strips near the edges we obtain a lower
bound for $L^{(d)}(x,y)$ and $\widetilde{L}^\textup{trunc}(x,y)$ as a linear
or quadratic polynomial in the logarithm that is responsible for the
singularity, whose positivity is straightforward to verify (e.g., using
Sturm's theorem or interval arithmetic).

\subsection{The corners}

The corners are intersections of three distinct singular curves.  As in
Section~\ref{sec:sub:nearedges}, we change coordinates if necessary so that
the corner is at $(x,y)=(0,0)$. After substituting the formulas in
\eqref{EandKnear1}, we obtain expressions for $L^{(d)}(x,y)$ and
$\widetilde{L}^\textup{trunc}(x,y)$ of the form
\begin{equation}\label{curlyH1}
\mathcal{H}(x,y) = \frac{1}{x-y}\sum_{i,j\ge 0}h_{i,j}(x,y)\log(x)^i\log(y)^j,
\end{equation}
where the sum contains only finitely many terms and each coefficient function
$h_{i,j}(x,y)$ is holomorphic in $\{x \in \C : |x|<1\}\times \{y \in \C :
|y|<1\}$.  In particular, these coefficient functions can be written in terms
of $\log(1-u)$, $E(u)$, $K(u)$, and the functions $A_j(u)$ from
\eqref{EandKnear1} for $u=x$ and $u=y$.

Since the kernel ${\mathcal H}(x,y)$ is well defined on the diagonal, the sum
in \eqref{curlyH1} must vanish there, but the individual coefficients
$h_{i,j}(x,y)$ might not.  To remedy this, we choose $\beta \in \Z$ and write
\begin{equation}\label{curlyH}
\aligned \mathcal{H}(x,y)&=\sum_{i,j\ge 0}\frac{\widetilde{h}_{i,j}(x,y)}{x-y}\log(x)^i\log(y)^j\\
&\quad\phantom{} +
\left(\frac yx\right)^{\beta}\sum_{i,j}h_{i,j}(x,x)\frac{\log(x)^i\log(y)^j}{x-y}
,
\endaligned
\end{equation}
with
\[
\widetilde{h}_{i,j}(x,y)=h_{i,j}(x,y)-\left(\frac
yx\right)^{\beta}h_{i,j}(x,x),
\]
so that $\widetilde{h}_{i,j}(x,x)=0$ and therefore
$\widetilde{h}_{i,j}(x,y)/(x-y)$ is holomorphic at $x=y$. This procedure is
required in only one corner for $L^{(8)}(x,y)$ and $L^{(24)}(x,y)$ (with
$\beta=0$), and one corner of $\widetilde{L}^\textup{trunc}(x,y)$ (with
$\beta=2$, in order to obtain better estimates).

Each function $\widetilde{h}_{i,j}(x,y)$ has a Taylor series expansion
$\sum_{n,m\ge 0}b_{n,m}x^n y^m$ that vanishes on the diagonal.  In other
words,
\[
\sum_{n,m\ge 0}b_{n,m}y^n y^m = 0,
\]
and therefore
\[
\frac{\widetilde{h}_{i,j}(x,y)}{x-y}= \sum_{n,m\ge 0}b_{n,m}\frac{x^n
y^m-y^n y^m}{x-y}=\sum_{n,m\ge 0}b_{n,m}(x^{n-1}+x^{n-2}y+\cdots +
y^{n-1})y^m
\]
gives a series expansion of this ratio.  The individual coefficients
$b_{n,m}$ can be computed from the Taylor series of $\log(1-z)$, $E(z)$,
$K(z)$, and $A_j(z)$, while at the same time the $5\cdot 0.99^{-n}$ bound on
the Taylor coefficients of these special functions gives an upper bound on
all $b_{n,m}$.  Computing a finite number of coefficients explicitly and
using this upper bound on the rest, we obtain a rigorous lower bound on the
first sum in \eqref{curlyH}.

When $h_{i,j}(x,x)$ is nonzero, direct computations show that it has a fairly
simple form, from which its positivity for small values of $x$ is manifest.
For example, $\widetilde{L}^\textup{trunc}(x,y)$ has $h_{i,j}(x,x)$ nonzero
only for one corner and two choices of indices $(i,j)$, where it equals
\[
(2-x) (x-1)^4 (x+1) (1-2 x)(x^2-x+1)^2 K(x)^2
\]
up to a positive constant factor (as well as the obviously positive factors
that were removed from $L^{(d)}$ and $\widetilde{L}^\textup{trunc}$ earlier,
as mentioned above). Similar, and in fact simpler, formulas hold in all other
cases, and direct computation shows that the second sum in \eqref{curlyH},
namely
\[
\sum_{i,j}h_{i,j}(x,x)\frac{\log(x)^i\log(y)^j}{x-y},
\]
can be rewritten as $(\log(y)-\log(x))/(y-x)$ times an explicit, manifestly
positive holomorphic function $d(x,y)$. From this we obtain a rigorous lower
bound for $\mathcal H(x,y)$ for $0<x,y\le \frac{1}{2}$, of the form
\begin{equation}\label{curlyHriglowerbd}
\sum_{i,j}p_{i,j}(x,y)\log(x)^i\log(y)^j + d(x,y)\frac{\log(y)-\log(x)}{y-x}
\end{equation}
with $p_{i,j}(x,y)\in \R[x,y]$ coming from the Taylor series argument above.
The coefficient functions $p_{i,j}(x,y)$ and $d(x,y)$ can all be approximated
well on small regions using interval arithmetic, as can the quotient
$(\log(y)-\log(x))/(y-x)$. To obtain the positivity as both $x$ and $y$
approach zero, we found it efficient in situations where $d(x,y)$ vanishes
identically and $p_{i,j}(0,0)=0$ to deduce lower bounds on
\eqref{curlyHriglowerbd} and similar expressions via lower bounds on
gradients, for which interval arithmetic had better behavior. These arguments
complete the proof of Proposition~\ref{prop:positivity}, and thus
Theorem~\ref{theorem:univopt}.

\subsection{Computer implementation}

Code to prove the inequalities needed for Proposition~\ref{prop:positivity}
is available from the MIT Libraries at
\url{https://hdl.handle.net/1721.1/141226}.
We provide two implementations, one via Mathematica notebooks \cite{wolfram}
and the other via Sage code \cite{sage}.

The Mathematica notebooks closely follow the techniques we have outlined in
this section. They run for about a week on a 16-core server. Documentation is
included in the notebooks, and we provide a README file that gives an
overview of how to use the notebooks.

The Sage code is a revised and optimized version of the proof. The
fundamental approach to treating the diagonals, edges, and corners remains
the same, but instead of using the arithmetic-geometric mean algorithm from
Section~\ref{subsec:elliptic} to compute elliptic integrals, we instead
expand all the functions we use as multivariate power series and derive
bounds for the error introduced by truncation. These techniques and bounds
are explained in the documentation provided with the code. The resulting
calculations finish in less than an hour on a single core.

\section{Generalizations and open questions}
\label{sec:generalizations}

Theorem~\ref{theorem:interpolation} is ideally suited to proving universal
optimality for $E_8$ and the Leech lattice, but the underlying analytic
phenomena are not limited to $8$ and $24$ dimensions. Instead, it seems that
the remarkable aspect of these dimensions is the existence of the lattices,
while interpolation theorems may hold much more broadly.

\begin{problem}
\label{prob:interpolation} Let $d$ and $k$ be positive integers.  If $f \in
\Schw_{\textup{rad}}(\R^d)$ satisfies $f^{(j)}\big(\sqrt{kn}\big) =
\widehat{f}\,^{(j)}\big(\sqrt{kn}\big) = 0$ for all integers $n \ge 0$ and $0
\le j < k$, then must $f$ vanish identically?
\end{problem}

The answer is yes when $k=d=1$ by \cite[Corollary~1]{RV}, and it follows for
$k=1$ and all dimensions using the techniques of \cite{RV} without much
difficulty. It also holds when $(d,k)=(8,2)$ or $(d,k)=(24,2)$ by
Theorem~\ref{theorem:interpolation}, and likely holds more broadly for $k=2$.
As far as we know, interpolation theorems of this form would not lead to any
optimality theorems in packing or energy minimization beyond $E_8$ and the
Leech lattice.

The special nature of the interpolation points $\sqrt{2n}$ plays an essential
role in our proofs.  For example, the functional equations
\[
F(\tau+2,x)-2F(\tau+1,x)+F(\tau,x)=0
\]
and
\[
\widetilde{F}(\tau+2,x)-2\widetilde{F}(\tau+1,x)+\widetilde{F}(\tau,x)=0
\]
encode algebraic properties of $\sqrt{2n}$, without which we would be unable
to construct the interpolation basis.  This framework (i.e.,
Theorem~\ref{thm:FEimpliesIF}) generalizes naturally to the interpolation
points $\sqrt{kn}$ from Open~Problem~\ref{prob:interpolation}, with the
corresponding functional equations using a $k$-th difference operator in
$\tau$.  However, our methods cannot apply to $k>2$ without serious
modification.

There is no reason why interpolation theorems should be restricted to radii
that are square roots of integers. We expect that far more is true, at the
cost of giving up the algebraic structure behind our proofs. In particular,
optimizing the linear programming bound for Gaussian energy seems to lead to
interpolation formulas, except in low dimensions. Recall that the optimal
functions for linear programming bounds seem to work as specified by the
following conjecture, which is a variant of \cite[Section~7]{CE} and
\cite[Section~9]{CK07}:

\begin{conjecture} \label{conj:LP}
Fix a dimension $d \ge 3$, density $\rho=1$, and $\alpha>0$. Then the optimal
linear programming bound $\widehat{f}(0) - f(0)$ from
Proposition~\ref{prop:LP} for the potential $p(r) = e^{-\alpha r^2}$ is
achieved by some radial Schwartz function~$f$, the radii $|x|$ for which
$f(x) = e^{-\alpha|x|^2}$ are the same as the radii $|y|$ for which
$\widehat{f}(y)=0$, they form a discrete, infinite set, and these radii do
not depend on $\alpha$.
\end{conjecture}

Note that at least one implication holds among the assertions of this
conjecture: the radii at which $f(x) = e^{-\alpha|x|^2}$ must be the same as
those for which $\widehat{f}(y)=0$ if all these radii are independent of
$\alpha$, thanks to the duality symmetry from \cite[Section~6]{CM}.

\begin{problem} \label{prob:LPinterp}
Let $r_1,r_2,\dots$ be the radii from Conjecture~\ref{conj:LP} for some $d$.
For which $d$ is every $f \in \Schw_{\textup{rad}}(\R^d)$ uniquely determined
by the values $f(r_n)$, $f'(r_n)$, $\widehat{f}(r_n)$, and
$\widehat{f}\,'(r_n)$ for $n \ge 1$ through an interpolation formula?
\end{problem}

Numerically constructing the interpolation basis seems to work about as well
for general $d \ge 3$ as it does for $d=8$ or $d=24$, which suggests that
interpolation holds for many or even all such $d$. In particular, simple
variants of the algorithms from \cite{CE} and \cite{CG} yield interpolation
formulas of this sort for all functions of the form $x \mapsto p(|x|^2)
e^{-\pi|x|^2}$, where $p$ is a polynomial of degree at most some bound~$N$.
As $N$ grows, these interpolation formulas seem to converge to well-behaved
limits when $d \ge 3$.

From an interpolation perspective, the numerical evidence suggests that these
dimensions behave much like $d=8$ and $d=24$. On the other hand, we know of
no simple description of the interpolation points $r_1,r_2,\dots$ when $d
\not\in\{8,24\}$, and we are not aware of any point configurations in $\R^d$
that meet the optimal linear programming bounds in these cases.  In other
words, $d=8$ and $d=24$ differ both algebraically and geometrically from the
other dimensions.

When $d \le 2$, universally optimal configurations exist (conjecturally for
$d=2$), but the radii $r_1,r_2,\dots$ are more sparsely spaced, seemingly too
much so to allow for an interpolation theorem.  For $d=1$, we can in fact
rule out an interpolation theorem corresponding to the point configuration
$\Z$:

\begin{lemma} \label{lemma:dim1fails}
Radial Schwartz functions $f \colon \R \to \R$ are not uniquely
determined by the values $f(n)$, $f'(n)$, $\widehat{f}(n)$, and
$\widehat{f}\,'(n)$ for integers $n \ge 1$.
\end{lemma}

This lemma follows immediately from \cite[Theorem~2]{RV}, but for
completeness we will give a simpler, direct proof.

\begin{proof}
Let $t$ be the tent function
\[
t(x) = \begin{cases} 1-|x| & \textup{if $|x| \le 1$, and}\\
0 & \textup{otherwise.}
\end{cases}
\]
Then
\[
\widehat{t}(y) = \left(\frac{\sin \pi y}{\pi y}\right)^2,
\]
which vanishes to second order at all nonzero integers. Now let $b$ be any
even, smooth, nonnegative function with support contained in
$[-\frac{1}{2},\frac{1}{2}]$, and consider the convolution $f = b*t$.  This
function is smooth and compactly supported, and it is therefore a Schwartz
function, whose Fourier transform $\widehat{f} = \widehat{b} \,\widehat{t}$
again vanishes to second order at all nonzero integers.  Furthermore, $f$
vanishes to infinite order at all nonzero integers other than $\pm 1$,
because its support is contained in $[-3/2,3/2]$. If the interpolation
property in the lemma statement held, then $f$ would be uniquely determined
by $f(1)$ and $f'(1)$. In other words, there would be just a two-dimensional
space of functions of this form. Furthermore, those two dimensions would have
to correspond to the support of $f$ and scalar multiplication.  Thus, $f$
would be completely determined by the support and total integral of $b$.
However, that is manifestly false: the value
\[
f(0) = \int_{-1}^1 b(x)(1-|x|)\,dx
\]
is not determined by the support and integral of $b$.
\end{proof}

The two-dimensional case is more difficult to analyze. If we scale the
hexagonal lattice so that it has density $1$, then the distances between the
lattice points are given by
\[
(4/3)^{1/4}\sqrt{j^2+jk+k^2},
\]
where $j$ and $k$ range over the integers.  Bernays \cite{Bernays} proved
that as $N \to \infty$, the number of such distances between $0$ and $N$
(counted without multiplicity) is asymptotic to
\[
\frac{CN^2}{\sqrt{\log N}}
\]
for some positive constant $C$. Thus, the distances in the hexagonal lattice
are slightly sparser than those in $E_8$ or the Leech lattice, for which the
corresponding counts of distinct distances are $N^2/2 + O(1)$. This sparsity
suggests that the interpolation property might fail in $\R^2$, and numerical
computations indicate that it does:

\begin{conjecture} \label{conj:interpolation-failure}
Let $r_1,r_2,\dots$ be the positive real numbers of the form
$(4/3)^{1/4}\sqrt{j^2+jk+k^2}$, where $j$ and $k$ are integers. Then radial
Schwartz functions $f \colon \R^2 \to \R$ are not uniquely determined by the
values $f(r_n)$, $f'(r_n)$, $\widehat{f}(r_n)$, and $\widehat{f}\,'(r_n)$ for
integers $n \ge 1$.
\end{conjecture}

Indeed, this conjecture has subsequently been proved by Sardari \cite{Sar}.
More generally, in each dimension we expect that interpolation fails whenever
the sequence $r_1,r_2,\dots$ contains only $o(N^2)$ elements between $0$
and~$N$ as $N \to \infty$, and perhaps even $(c+o(1))N^2$ elements for some
constant $c<\frac{1}{2}$.  However, we have not explored this possibility
thoroughly. Note that the interpolation property cannot depend solely on the
asymptotic growth rate of the radii, because it is sensitive to deleting a
single interpolation point. We have no characterization of when the
interpolation property holds, and it is unclear just how sensitive it is. For
example, does moving (but not removing) finitely many interpolation points
preserve the interpolation property?

Cohn and Kumar \cite{CK07} conjectured that the linear programming bounds are
sharp in two dimensions and prove the universal optimality of the hexagonal
lattice. There is strong numerical evidence in favor of this conjecture, and
in fact the numerics converge far more quickly than in eight or twenty-four
dimensions. However, it seems surprisingly difficult to prove that they
converge to a sharp bound. Assuming
Conjecture~\ref{conj:interpolation-failure} holds, one cannot prove universal
optimality in $\R^2$ via a straightforward adaptation of the interpolation
strategy used in $\R^8$ and $\R^{24}$. Instead, a more sophisticated approach
may be needed.

Despite Lemma~\ref{lemma:dim1fails}, universal optimality in $\R^1$ can be
proved using an interpolation theorem for a different function space, namely
Shannon sampling for band-limited functions (see~\cite[p.~142]{CK07}).  Is it
possible that universal optimality in $\R^2$ also corresponds to an
interpolation theorem for some space of radial functions?  That would
establish a satisfying pattern, but we are unable to propose what the
relevant function space might be.

\section*{Acknowledgements}

We are grateful to Ganesh Ajjanagadde, Andriy Bondarenko, Matthew de
Courcy-Ireland, Bill Duke, Noam Elkies, Yuri Gurevich, Tom Hales, Mathieu
Lewin, Ed Saff, Peter Sarnak, Sylvia Serfaty, Rich Schwartz, Sal Torquato,
Ramarathnam Venkatesan, Akshay Venkatesh, and Don Zagier for helpful
conversations.  We also thank Princeton and the Institute for Advanced Study
for hosting Viazovska during part of this work.

\end{document}